\documentclass[12pt]{amsart}
\usepackage{amsfonts,amssymb,amsmath,amscd,amsthm, amstext,bm,color}
\usepackage[colorlinks, citecolor=blue,pagebackref,hypertexnames=false]{hyperref}
\usepackage{comment}
\usepackage{enumerate}
\usepackage{epstopdf}
\usepackage{extarrows} 
\usepackage{fancybox} 
\usepackage{fancyhdr}
\usepackage{geometry}                
\usepackage{graphicx}
\usepackage{mathrsfs}
\usepackage{txfonts}
\usepackage{tikz}
\usepackage{ulem}  
\usepackage{url}  

\newcommand{\R}{\mathbb{R}}

\newcommand{\D}{\mathcal{D}}

\allowdisplaybreaks
\newtheorem{theorem}{Theorem}[section]
\newtheorem{proposition}[theorem]{Proposition}
\newtheorem{lemma}[theorem]{Lemma}

\newtheorem{remark}[theorem]{Remark}

\newtheorem{definition}[theorem]{Definition}

\DeclareMathOperator{\Sha}{\mathrm{III}^{\mathcal D}}
\title[Schatten class in the Bloom setting]{Schatten classes and commutators of Riesz transforms in the two weight setting}
 
\author{Michael Lacey}
\address{Michael Lacey, Department of Mathematics, Georgia Institute of Technology Atlanta, GA 30332, USA}
\email{lacey@math.gatech.edu}

\author{Ji Li}
\address{Ji Li, School of Mathematical and Physical Sciences, Macquarie University, NSW 2109, Australia}
\email{ji.li@mq.edu.au}

\author{Brett D. Wick}
\address{Brett D. Wick, Department of Mathematics, Washington University - St. Louis, St. Louis, MO 63130-4899 USA}
\email{wick@math.wustl.edu}

\author{Liangchuan Wu}
\address{Liangchuan Wu,  School of Mathematical Science, Anhui University, Hefei, 230601, P.R.~China,  and 
School of Mathematical and Physical Sciences, Macquarie University, NSW 2109, Australia}
\email{wuliangchuan@ahu.edu.cn}


\begin{document}
\begin{abstract} 
We characterize the Schatten class $S^p$ of the commutator of Riesz transforms 
$[b,R_j]$ in $\mathbb R^n$ ($j=1,\ldots, n$) in the two weight setting for $n< p<\infty$,  
by introducing the condition that the symbol $b$ being in Besov spaces associated with the given two weights. 
At the critical index $p=n$,  the commutator $[b,R_j]$ belongs to Schatten class $S^{n}$ if and only if $b$ is a constant,  and to the weak Schatten class $S^{n,\infty}$ 
if and only if $b$ is in an oscillation sequence space associated with the given two weights.  
As a direct application, we have the Schatten class estimate for A. Connes' quantised derivative in the two weight setting.
\end{abstract}

\maketitle
\tableofcontents 


\section{Introduction}
\setcounter{equation}{0}

Suppose $1<p<\infty$. Let $\mu,\, \lambda$ be in the Muckenhoupt class $A_p$,  $\nu= \mu^{1\over p}\lambda^{-{1\over p}}$, 
and  denote by $R_j$  the $j$\,th Riesz transform on $\mathbb R^n$, $j=1,\ldots, n$.  
In \cite{B} Bloom first showed that the commutator $[b,R_j]$, defined as  $[b,R_j] f(x)=b(x) R_jf(x)-R_j(bf)(x)$, 
is bounded from $L^p_\mu(\mathbb R^n)$ to $L^p_\lambda(\mathbb R^n)$ if and only if $b$ is in ${\rm BMO}_\nu(\mathbb R^n)$, 
which extends the remarkable result of Coifman--Rochberg--Weiss \cite{CRW1976}. 
Just recently, the first and second authors \cite{LL2022} showed that 
$[b,R_j]$ is compact from $L^p_\mu(\mathbb R^n)$ to $L^p_\lambda(\mathbb R^n)$ 
if and only if $b$ is  in ${\rm VMO}_\nu(\mathbb R^n)$ via a constructive proof. 
This was also addressed in \cite{HL1, HL2} via extrapolation.   
A next step in this program is to understand the Schatten class membership 
(the definition is given in \S\ref{s:Schatten}) of the commutator as this is a more refined property than compactness.

Schatten class estimates provide a better understanding of the behaviour of compact operators in infinite dimensional spaces 
and connect to non-commutative analysis, geometry, and complex analysis 
(see for example \cite{Zhu91, CST1994, Pisier, JX10, I2013, Pau16, LMSZ2017}). 
It is well-known from Peller \cite{P2003} that the commutator with the Hilbert transform $[b,H]$ is in the Schatten class $S^p$ 
if and only if $b$ is in the Besov space $B^p(\mathbb R)$ for $0<p<\infty$. 
While in the higher dimensional case, there is a ``cut-off'' phenomenon in the sense that 
when $p\leq n$, if $[b,R_j]$, $j=1,\ldots,n$, is in $S^p$ then $b$ must be a constant. 
At the critical index $p=n$,  $[b,R_j]$ is in the weak Schatten class $S^{n,\infty}$ 
if and only if $b$ is in the homogeneous Sobolev space $\dot{W}^{1,n}(\mathbb R^n)$. 
This further links to the $S^{n,\infty}$ estimate for the quantised derivative 
$\bar{d} b:=i\left[\operatorname{sgn}(\mathfrak{D}), 1 \otimes M_b\right]$, $i^2=1$,
of Alain Connes \cite[Chapter $\mathrm{IV}$]{C1994}, where $\mathfrak D$ is the $n-$dimensional Dirac operator, 
and $M_b$ is the multiplication operator: $M_bf(x) = b(x)f(x)$. Details of these notation will be stated in the last section. 
This has been intensively studied, see for example \cite{FLMSZ, FSZ2022, GG2017, LMSZ2017, RS1988, RS1989, JW1982,CST1994}. 
We note that in \cite{LMSZ2017,FSZ2022} they implemented a new approach to prove that for 
$b\in {\rm BMO}(\mathbb R^n)$, $\bar{d} b$ is in the $S^{n,\infty}$ 
if and only if $b$ is in  the homogeneous Sobolev space $\dot{W}^{1,n}(\mathbb R^n)$.

The Schatten class for  $[b,H]$ in the two weight setting was first studied in \cite{LLW2022}, 
where only the case $p=2$ was solved, and the questions for $p\not=2$ and for the higher dimensional case were raised.

In this paper, we address these questions by establishing the Schatten class $S^p$ characterisation for the Riesz transform commutator $[b,R_j]$, $j=1,\ldots,n$, 
in the two weight setting in $\mathbb R^n$. 
We also characterise the critical index in the two weight setting for both Schatten class $S^n$ and weak Schatten class $S^{n,\infty}$. 
The key is to introduce the suitable Besov space and sequence spaces 
associated with the given two weights and to use the approach of martingales. The idea and techniques in this paper give the answer to Schatten class $S^p$ for  $[b,H]$ in the two weight setting for $1<p<\infty$. Note that the case for $0<p\leq1$ is still not clear.

Throughout this paper, we assume $n\geq2$. 

We now introduce the Besov spaces associated with the weight $\nu\in A_2$.

\begin{definition}\label{def:Sequence}
Suppose  $\nu\in A_2$ and $0<p<\infty$. Let  $\mathcal D$ be any dyadic system in $\mathbb R^n$.
 For $b\in L^1_{loc}(\mathbb R^n)$, 
we say that $b\in \mathbf B_{\nu}^p(\mathbb R^n)$ if the following  sequence
$$ 	
    \{R_Q\}_{Q\in\mathcal D}
    :=  \left\{\frac{1}{\nu(20\sqrt{n}Q)}\int_{20\sqrt{n}Q} \left|b(x)-
    \langle b\rangle _{20\sqrt{n}Q}\right|dx\right\}_{Q\in\mathcal D}
$$
is in $\ell^{p}$.  Here $\langle b\rangle _{20\sqrt{n}Q}$ is the average of $b$ over $20\sqrt{n}Q$.   We set 
$ 
    \|b\|_{\mathbf B_{\nu}^p(\mathbb R^n)}
    =\left\|\left\{R_Q\right\}_{Q\in\mathcal D}\right\|_{ \ell^{p}}.
$
\end{definition}

The specific value of the parameter $20\sqrt{n}$ here is not essential, we choose it 
to ensure the subsequent arguments proceed smoothly. 
We will show in Section \ref{sec:Besov} that $\mathbf B_\nu^p(\mathbb R^n)$ is contained in ${\rm VMO}_\nu(\mathbb R^n)$.

The main result is the following, which addresses the question in \cite{LLW2022}.

\begin{theorem}\label{thm:main1}
Let $R_j$ be the  $j$th  Riesz transform on $\mathbb R^n$ with $n\geq 2$, $j=1,2,\ldots, n$.  
Suppose  $n< p<\infty$, $\lambda,\, \mu\in A_2$ and set $\nu=\mu^{1\over 2}\lambda^{-{1\over 2}}$. 
Suppose $b\in {\rm VMO}_\nu(\mathbb R^n)$. 
Then commutator $[b,R_j]$ belongs to $S^p \big (L_{\mu}^2(\mathbb R^n), L_{\lambda}^2(\mathbb R^n)\big )$ 
if and only if $b\in \mathbf B_{\nu}^p(\mathbb R^n)$. 
Moreover, we have
$$
    \|b\|_{\mathbf B_\nu^p(\mathbb R^n)}  
    \approx \left\|[b,R_j]\right\|_{S^p \big(L_{\mu}^2(\mathbb R^n), L_{\lambda}^2(\mathbb R^n)\big)},
$$
where the implicit constants depend on $p,$ $n,$ $[\lambda]_{A_2}$ and $[\mu]_{A_2}$.
\end{theorem}

\smallskip

The proof of Theorem~\ref{thm:main1} proceeds by translating the inequality into a form more symmetric with respect to the weights, 
and then relying upon familiar (but in this context, new) dyadic methods.  
Using the approach of Petermichl, Treil and Volberg in \cite{PTV}, we can represent the Riesz transforms as dyadic shifts. 
The commutator is then seen to be a sum of paraproduct type operators. 
The dyadic formalism allows easy access to the important method of 
nearly weakly orthogonal sequences of Rochberg and Semmes \cite{RS1989}. 
The dyadic methods give the continuous result, as the  Besov space is the intersection of a finite number of dyadic Besov spaces.   We note that following the same idea and techniques, the characterization in Theorem~\ref{thm:main1} also holds for the Schatten--Lorentz class $S^{p,q} \big (L_{\mu}^2(\mathbb R^n), L_{\lambda}^2(\mathbb R^n)\big )$, $n<p<\infty, \, 0< q\leq\infty$, via assuming the sequences $ \{R_Q\}_{Q\in\mathcal D}$ defined in Definition \ref{def:Sequence} to be in $\ell^{p,q}$.

The tools we need include a dyadic  characterization for the Besov space $\mathbf B_\nu^p(\mathbb R^n)$, 
that is, $\mathbf B_\nu^p(\mathbb R^n)$ coincides with the intersection of several `translated' copies of 
the dyadic weighted Besov space $\mathbf B_{\nu}^p(\mathbb{R}^n, \mathcal D)$ 
(see Theorem \ref{thm:Besov-Intersect} in Section \ref{sec:Besov});
and the Schatten class characterization of the paraproduct $\Pi_b^{\mathcal D}$ with symbol $b$ defined as
\begin{equation}\label{e:Paraproduct}
    \Pi_b^{\mathcal D}
    := \sum_{Q\in\mathcal{D}:\, \epsilon\not\equiv 1} \langle b, h_Q^\epsilon\rangle 
    h_Q^\epsilon\otimes \frac{\mathsf{1}_Q}{\left\vert Q\right\vert}
\end{equation}
and its adjoint (see Proposition \ref{prop:para-Besov} in Section \ref{sec:thm1}):
$$
    \left\|\Pi_b^{\mathcal D} \right\|_{S^p\big(L_{\mu}^2(\mathbb R^n), L_{\lambda}^2 (\mathbb R^n)\big)}  
    \approx  \|b\|_{\mathbf B_\nu^p(\mathbb R^n, \mathcal D)}.
$$

It is natural to further ask the question on the critical index $p=n$ as this is an important fact 
due to the work of Janson--Wolff \cite{JW1982} (see also \cite{F2022,FSZ2022}). 
We give a criterion in this two weight setting, {i.e. $\mu,\ \lambda\in A_2$, and $\nu =\mu^{1\over 2}\lambda^{-{1\over 2}}$}.  
By using these weights we recover the classical phenomenon in \cite{JW1982} when the $A_2$ weight $\nu$ is the Lebesgue measure.

\begin{theorem}\label{thm:main2}
Let $R_j$ be the  $j$th  Riesz transform on $\mathbb R^n$ with $n\geq 2$, $j=1,2,\ldots, n$.   
Suppose  $\lambda,\, \mu\in A_2$ and set $\nu=\mu^{1\over 2}\lambda^{-{1\over 2}}$. 
Suppose $b\in {\rm VMO}_\nu(\mathbb R^n)$ and $0<p\leq n$.
If $[b,R_j]$ belongs to $S^p \big (L_{\mu}^2(\mathbb R^n), L_{\lambda}^2(\mathbb R^n)\big )$, 
then  $b$ must be a constant almost everywhere.  
\end{theorem}

Next, it is natural to explore the weak Schatten class estimate for the critical index $p=n$.  
We now introduce the corresponding new function space of weighted oscillation. 

\begin{definition}\label{def:Weight-Sobolev}
Suppose  $\nu\in A_2$.  For $b\in L^1_{loc}(\mathbb R^n)$, 
we say that $b\in \mathcal W_{\nu}(\mathbb R^n)$ if the following  sequence
$$ 	
    \{R_Q\}_{Q\in\mathcal D}
    :=  \left\{\frac{1}{\nu(20\sqrt{n}Q)}\int_{20\sqrt{n}Q} \left|b(x)-
    \langle b\rangle _{20\sqrt{n}Q}\right|dx\right\}_{Q\in\mathcal D}
$$
is in $\ell^{n,\infty}$ {\rm (}a sequence $\{a_k\}$ is in $\ell^{n,\infty}$ if $\sup_{k} \big\{k^{1/n}a_k^{*}\big\}<+\infty$ 
with $\big\{a_k^*\big\}$ the non-increasing rearrangement of $\{a_k\}${\rm )}. 
Here $\langle b\rangle _{20\sqrt{n}Q}$ is the average of $b$ over $20\sqrt{n}Q$.  We set 
$$ 
    \|b\|_{\mathcal W_{\nu}(\mathbb R^n)}
    =\left\| \left\{R_Q\right\}_{Q\in\mathcal D}\right\|_{ \ell^{n,\infty}}.
$$
\end{definition}

From this definition we see directly that $\mathcal W_{\nu}(\mathbb R^n)\subset {\rm BMO}_{\nu}(\mathbb R^n)$ 
(explained at the beginning of Section \ref{sec:6}). 
In fact, we can further show that $\mathcal W_{\nu}(\mathbb R^n)\subset {\rm VMO}_{\nu}(\mathbb R^n)$ 
(hence $[b,R_j]$ is compact when $b\in\mathcal W_{\nu}(\mathbb R^n)$). We will prove this in Section \ref{sec:6}. 
We also note that when $\mu$ and $\lambda$ are constants, $\mathcal W_{\nu}(\mathbb R^n)$ will become 
the standard homogeneous Sobolev space $\dot W^{1,n}(\mathbb R^n)$ (see for example \cite{F2022}).

To continue, we  establish the following equivalent characterization of $\mathcal W_{\nu}(\mathbb R^n)$ with 
$\nu=\mu^{1\over 2}\lambda^{-{1\over 2}}$, where $\lambda,\, \mu\in A_2$. 
This is crucial to the characterization for the weak Schatten class $S^{n,\infty}$.

\begin{theorem}\label{thm:Weight-Sobolev-revised}
Suppose  $\lambda,\, \mu\in A_2$ and set $\nu=\mu^{1\over 2}\lambda^{-{1\over 2}}$.  
$\mathcal W_{\nu}(\mathbb R^n)$ can be characterized equivalently as follows:
\begin{align*}
    \|b\|_{\mathcal W_{\nu}(\mathbb R^n)} 
    &\approx\bigg \|  \bigg\{\bigg[{1\over \mu(20\sqrt{n}Q)} \int_{20\sqrt{n}Q} 
    \left|b(x)- \langle b\rangle_{20\sqrt{n}Q}\right|^{2} \lambda(x) \,dx\bigg]^{1\over 2}
    \bigg\}_{Q\in\mathcal D}  \bigg \|_{ \ell^{n,\infty}} \\
    &\approx \bigg \| \bigg\{ \bigg[\frac{1}{\lambda^{-1}(20\sqrt{n}Q)} 
    \int_{20\sqrt{n}Q} \left|b-\langle b\rangle_{20\sqrt{n}Q}\right|^{2} 
    \mu^{-1} (x) \, dx \bigg]^{1\over {2}} \bigg\}_{Q\in\mathcal D}\bigg \|_{ \ell^{n,\infty}},
\end{align*}
where the implicit constants are independent of $b$.
\end{theorem}

Then we have the following characterization for the weak Schatten class $S^{n,\infty}$.

\begin{theorem}\label{thm:main3}
 Let $R_j$ be the  $j$th  Riesz transform on $\mathbb R^n$ with $n\geq 2$, $j=1,2,\ldots, n$. 
 Suppose  $\lambda,\, \mu\in A_2$ and set $\nu=\mu^{1\over 2}\lambda^{-{1\over 2}}$. 
 Suppose $b\in L^1_{loc}(\mathbb R^n)$. 
 Then $[b,R_j]$ belongs to $S^{n,\infty} \big (L_{\mu}^2(\mathbb R^n), L_{\lambda}^2(\mathbb R^n)\big )$ 
 if and only if  $b\in \mathcal W_{\nu}(\mathbb R^n)$, with 
$$  
    \|b\|_{\mathcal W_{\nu}(\mathbb R^n)}
    \approx  \|[b,R_j]\|_{S^{n,\infty} \big (L_{\mu}^2(\mathbb R^n), L_{\lambda}^2(\mathbb R^n)\big )}, 
$$
where the implicit constants depends on $p,$ $n,$ $[\lambda]_{A_2}$ and $[\mu]_{A_2}$.
\end{theorem}

The key idea is to introduce a suitable version of oscillation sequence space and characterize its structures, 
and then expand the localized Riesz commutator kernel via martingales. We highlight that with the two weights involved, this expansion does not fall into the scope of the classical ones.
Part of the expansion further links to the nearly weakly orthogonal sequences (NWOs) of functions associated with the weights $\mu$ and $\lambda$, while the other part is not NWO but requires a more direct verification.

It is still an open question whether $\mathcal W_{\nu}(\mathbb R^n)$ is equivalent to some two-weight Sobolev spaces.

Regarding the two weight Besov space in dimension one introduced by the first, second and third authors in \cite{LLW2022}, 
it is natural to explore a two weight Besov space of the following form, i.e., via the Sobolev--Slobodeckii norms. 
\begin{definition}\label{def:Besov V1}
Suppose $2\leq p<\infty$, $\lambda,\, \mu\in A_2$ and set $\nu=\mu^{1\over p}\lambda^{-{1\over p}}$. 
Let $b\in L_{\rm loc}^1(\mathbb R^n)$. We say that $b$ belongs to the weighted Besov space $B_{\nu}^p(\mathbb R^n)$ if
\begin{equation}\label{e:Bp}
   \|b\|_{B_{\nu}^{p}(\mathbb{R}^n)}^p
   := \int_{\mathbb{R}^n} \int_{\mathbb{R}^n}  \frac{|b(x)-b(y)|^p}{|x-y|^{2n}} \lambda (x)\mu^{-1} (y)  \,dy dx<\infty.
\end{equation}
\end{definition}

We point out that this will lead to a slightly different version of characterisations: $[b,R_j]$ belongs to $S^p \big (L_{\mu^{2/p}}^2(\mathbb R^n), L_{\lambda^{2/p}}^2(\mathbb R^n)\big )$ 
if and only if $b\in B_{\nu}^p(\mathbb R^n)$. 
Moreover, we have
$$
    \|b\|_{B_\nu^p(\mathbb R^n)}  
    \approx \left\|[b,R_j]\right\|_{S^p \big(L_{\mu^{2/p}}^2(\mathbb R^n), L_{\lambda^{2/p}}^2(\mathbb R^n)\big)}.
$$
Note that this setting is well-defined since $\lambda,\, \mu\in A_2$ imply that $\lambda^{p\over2},\, \mu^{p\over2}\in A_2$, $p\geq2$. Hence, $\nu=\mu^{1\over p}\lambda^{-{1\over p}}\in A_2$.
As this is of independent interest regarding the different versions of two weight Besov spaces, we refer readers to the full details in Version 1 of our paper \cite{LLWW}.

This paper is organized as follows. 
In Section \ref{Preliminaries} we introduce the necessary preliminaries in the two weight setting and the Schatten classes. 
In Section \ref{sec:Besov} we characterize the dyadic structure for the weighted Besov space $\mathbf B_{\nu}^p(\mathbb R^n)$. 
In Section \ref{sec:thm1} we present the proof of Theorem \ref{thm:main1}. 
In Section \ref{sec:thm2} we provide the proof of Theorem  \ref{thm:main2} for the critical index $p=n$. 
In Section \ref{sec:6} we study the properties of the weighted oscillation space $\mathcal W_{\nu}(\mathbb R^n)$ 
and prove Theorem \ref{thm:Weight-Sobolev-revised}, 
and in Section \ref{sec:thm3} we give the proof of Theorem \ref{thm:main3}.
As an application, we have the Schatten class estimate for the quantised derivative of A. Connes, 
which will be briefly addressed in the last section.


\section{Preliminaries}\label{Preliminaries}
\setcounter{equation}{0}

We recall some preliminaries for the dyadic systems and the Schatten class.


\subsection{Dyadic system in $\mathbb{R}^n$}

\begin{definition}
Let the collection $\mathcal{D}^{0}=\mathcal{D}^{0}(\mathbb{R}^n)$ 
denote the standard system of dyadic cubes in $\mathbb{R}^n$, where
$$
    \mathcal{D}^{0}(\mathbb{R}^n)
    =\bigcup_{k \in \mathbb{Z}} \mathcal{D}^{0}_{k}(\mathbb{R}^n)
$$ 
with
$$ 
    \mathcal{D}^{0}_{k}(\mathbb{R}^n)
    =\left\{2^{-k}([0,1)^{n}+m):k\in\mathbb{Z},m\in\mathbb{Z}^{n}\right\} .
$$
\end{definition}

Next, we recall shifted dyadic systems of dyadic cubes in $\mathbb{R}^n$.

\begin{definition}[\cite{HK2012,LPW}] \label{def:adjacent-dyadic}
For $\omega=(\omega_1,\omega_2,\ldots,\omega_n)\in \left\{0,\frac{1}{3},\frac{2}{3}\right\}^{n}$, 
we can define the dyadic system  
$\mathcal{D}^{\omega}=\mathcal{D}^{\omega}(\mathbb{R}^n)$,
$$
    \mathcal{D}^{\omega}(\mathbb{R}^n)
    =\bigcup_{k \in \mathbb{Z}} \mathcal{D}^{\omega}_{k}(\mathbb{R}^n),
$$
where
$$ 
    \mathcal{D}^{\omega}_{k}(\mathbb{R}^n)
    =\left\{2^{-k}\left([0,1)^{n}+m+(-1)^{k}\omega\right):k\in\mathbb{Z},m\in\mathbb{Z}^{n}\right\} .
$$
\end{definition}

It is straightforward to check that $\mathcal{D}^{\omega}$ inherits the nestedness property of $\mathcal{D}^{0}$: 
if $Q, Q' \in \mathcal{D}^{\omega}$, then $Q \cap Q' \in\{Q, Q', \varnothing\}$.  
See \cite{HK2012,LPW} for more details. 
When the particular choice of $\omega$ is unimportant, 
the notation $\mathcal{D}$ is sometimes used for a generic dyadic system.

The fundamental property of the adjacent dyadic systems is the following: for any ball  $B$ in $\mathbb R^n$, 
there exists a dyadic cube $Q\in \mathcal{D}^{\omega}$ for some $\omega\in \left\{0,\frac{1}{3},\frac{2}{3}\right\}^{n}$, 
such that  
\begin{equation}\label{eqn:adj-prop}
 B\subset Q\quad \text{and}  \quad  |Q|\leq C |B|, 
\end{equation} 
where $C$ is a positive constant independent of $B$, $Q$ and $\omega$. 


\subsection{An expression of Haar functions}

For any dyadic cube $Q\in \mathcal{D}$, 
there exist dyadic intervals $I_1, \cdots, I_n$ on $\mathbb{R}$ with common length $\ell(Q)$, 
such that $Q=I_1 \times  \cdots\times I_n$. Then $Q$ is associated with $2^n$ Haar functions:
$$
    h^{\epsilon}_{Q}(x):=h_{I_1\times\cdots\times I_n}^{(\epsilon_1,\cdots, \epsilon_n)}(x_1,\cdots,x_n)
    :=\prod_{i=1}^{n}h^{(\epsilon_i)}_{I_i}(x_i),
$$
where $\epsilon = (\epsilon_1,\cdots,\epsilon_n)\in \{0,1\}^{n}$ and
$$ 
    h^{(1)}_{I_i}
    := \frac{1}{\sqrt{I_i}}\mathsf{1}_{I_i}  \qquad\text{and}\qquad h^{(0)}_{I_i}
    :=  \frac{1}{\sqrt{I_i}}(\mathsf{1}_{I_i-}-\mathsf{1}_{I_i+}).
$$
Writing $\epsilon\equiv 1$ when $\epsilon_{i}\equiv1$ for all $i=1,2,\ldots,n$, 
$h^{1}_{Q}:=\frac{1}{\sqrt{Q}}\mathsf{1}_Q$ is non-cancellative; 
on the other hand, when $\epsilon\not\equiv1$, 
the rest of the $2^{n}-1$ Haar functions $h^{\epsilon}_{Q}$ associated with $Q$ satisfy the following properties:

\begin{lemma}\label{haar} 
For $\epsilon\not\equiv1$, we have
\begin{itemize}
\item [(1)] $h^{\epsilon}_{Q}$ is supported on $Q$ and $\int_{\mathbb{R}^n}h^{\epsilon}_{Q}(x)\,dx=0$;  

\item [(2)] $\langle h^{\epsilon}_{Q},h^{\eta}_{Q}\rangle=0$, for $\epsilon\not\equiv\eta$;  
	
\item [(3)] if $h^{\epsilon}_{Q}\neq0$, then 
$\|h^{\epsilon}_{Q}\|_{L^{p}(\mathbb{R}^{n})}= |Q|^{\frac{1}{p}-\frac{1}{2}} $ for $1\leq p\leq\infty$; 
	
\item [(4)] $ \|h^{\epsilon}_{Q}\|_{L^{1}(\mathbb{R}^{n})}\cdot\|h^{\epsilon}_{Q}\|_{L^{\infty}(\mathbb{R}^{n})}= 1$;  
   
\item[(5)] noting that the average of a function $b$ over a dyadic cube $Q$: 
$ \langle b\rangle_{Q}:=\frac{1}{|Q|}\int_{Q}b(x) \,dx $ can be expressed as:
$$ 
    \langle b\rangle_{Q}
    =\sum_{\substack{P\in \mathcal{D}, Q\subsetneq P\\\epsilon\not\equiv1}}\langle b, h^{\epsilon}_{P}\rangle h^{\epsilon}_{P}(Q), 
$$
where $h^{\epsilon}_{P}(Q)$ is a constant; 

\item[(6)] fixing a cube $Q$, and expanding $b$ in the Haar basis, we have
$$
    (b(x)- \langle b\rangle_{Q})\mathsf{1}_{Q}(x)
    =\sum_{\substack{R\in \mathcal{D},R\subset Q\\\epsilon\not\equiv1}}\langle b, h^{\epsilon}_{R}\rangle h^{\epsilon}_{R};
$$

\item[(7)] denote by $E_k b(x)=\sum_{Q\in \mathcal{D}_k} \langle b\rangle_Q \mathsf{1}_Q (x)$, then we have
$$ 
    E_{k+1} b(x)-E_k b(x) 
    =\sum_{Q\in \mathcal{D}_k} \sum_{\epsilon\not\equiv1} \langle b, h^{\epsilon}_{Q}\rangle h^{\epsilon}_{Q}.
$$
\end{itemize}
\end{lemma}

We will also use the notation 
\begin{equation} \label{e:Delta}
    \Delta _Q  b = \sum_{\epsilon\not\equiv1} \langle b, h^{\epsilon}_{Q}\rangle h^{\epsilon}_{Q} 
    = (E_{k+1}b-E_k b) \cdot \mathsf{1}_Q. 
\end{equation}


\subsection{Characterization of Schatten Class}\label{s:Schatten}

Let $\mathcal{H}_1$ and $\mathcal{H}_2$ be separable complex Hilbert spaces. 
Suppose $T$ is a compact operator from $\mathcal{H}_1$ to  $\mathcal{H}_2$, 
let $T^*$ be the adjoint operator, it is clear that $|T|=(T^*T)^{\frac{1}{2}}$ is a compact, 
self-adjoint, non-negative operator from $\mathcal{H}_1$ to  $\mathcal{H}_1$. 
Let $\left(\psi_k\right)_k$ be an orthonormal basis for $\mathcal{H}_1$ consisting of eigenvectors of $|T|$, 
and let $s_k(T)$ be the eigenvalue corresponding to the eigenvector $\psi_k, k\in \mathbb{Z}^+$.
The numbers $s_1(T) \geq s_2(T) \geq \cdots \geq s_n(T) \geq \cdots \geq 0$, are called the singular values of $T$. 
If $0<p<\infty$, $0<q\leq \infty$ and the sequence of singular values is $\ell^{p,q}$-summable (with respect to a weight), 
then $T$ is said to belong to the Schatten--Lorentz class $S^{p,q}\left(\mathcal{H}_1, \mathcal{H}_2\right)$.
That is, $\|T\|_{S^{p,q}\left(\mathcal{H}_1, \mathcal{H}_2\right)}
=\big(\sum_{k\in \mathbb{Z}^{+}}\left(s_{k}\left(T\right)\right)^{q}k^{\frac{q}{p}-1}\big)^{1\over q}$, $q<\infty$, and
$\|T\|_{S^{p,\infty}\left(\mathcal{H}_1, \mathcal{H}_2\right)}
=\sup_{k\in \mathbb{Z}^{+}}s_{k}\left(T\right) k^{\frac{1}{p}}$, $ q=\infty$.  
Clearly, $S^{p,p}\left(\mathcal{H}_1, \mathcal{H}_2\right)=S^{p}\left(\mathcal{H}_1, \mathcal{H}_2\right)$. 

In 1989, Rochberg and Semmes \cite{RS1989} introduced methods that are very useful 
in estimating Schatten norms in the setting of commutators. 
A key concept is that of \textit{nearly weakly orthogonal} ({\rm NWO}) sequences of functions.

\begin{definition}\label{def:NWO}
Let $\{e_{Q}\}_{Q\in \mathcal{D}}$ be a collection of functions. 
We say $\{e_Q\}_{Q\in \mathcal{D}}$ is a {\rm NWO} sequence, 
if  ${\rm supp}\,  e_Q \subset Q$ and the maximal function $f^*$ is bounded on $L^2(\mathbb{R}^n)$, 
where $f^*$ is defined as
$$
    f^*(x)=\sup_{Q}\frac{|\langle f,e_Q \rangle|}{|Q|^{1\over 2}}\mathsf{1}_{Q}(x).
$$
\end{definition}

In this paper, we work with weighted versions. We will use the following result proved by Rochberg and Semmes.

\begin{lemma}[\cite{RS1989}]\label{lem:NWO}
If the collection of functions $\left\{e_Q\right\}_{Q \in \mathcal{D}}$ are supported on $Q$ 
and satisfy for some $2<r<\infty$, $\|e_Q\|_r \lesssim|Q|^{{1 \over  r}-{1 \over  2}}$, 
then $\left\{e_Q\right\}_{Q\in \mathcal{D}}$ is a {\rm NWO} sequence.
\end{lemma}

Let $\mathcal{H}_1$, $\mathcal{H}_2$ be separable complex Hilbert spaces. 
In \cite{RS1989} Rochberg and Semmes provided a substitute for 
the Schmidt decomposition of the operator $T$ 
$$
    T=\sum_{Q\in \mathcal{D}}\lambda_{Q}\langle\cdot,e_{Q}\rangle f_{Q}
$$
with $\{e_{Q}\}_{Q\in \mathcal{D}}$ and $\{f_{Q}\}_{Q\in \mathcal{D}}$ 
being {\rm NWO} sequences and $\{\lambda_{Q}\}_{Q\in \mathcal{D}}$ is a sequence of scalars. 
From \cite[Corollary 2.8]{RS1989} we see that
\begin{align}\label{eq-NWO1}
    \|T\|_{S^{p,q}(\mathcal{H}_1,\mathcal{H}_2)}
    \lesssim 
    \left\|\left\{\lambda_Q\right\}_{Q\in \mathcal D}\right\|_{\ell^{p,q}}, \quad 0<p<\infty, 0<q\leq\infty.
\end{align}
When $1<p=q<\infty$, Rochberg and Semmes also obtained

\begin{lemma}[\cite{RS1989}]\label{eq-NWO}
For any bounded compact operator $T$ on $L^2(\mathbb{R}^{n})$, and $\{e_{Q}\}_{Q\in \mathcal{D}}$ 
and $\{f_{Q}\}_{Q\in \mathcal{D}}$ {\rm NWO} sequences, then for $1<p<\infty$, 
$$
    \left[\sum_{Q\in \mathcal{D}}\left|\langle T e_{Q},f_{Q}\rangle\right|^{p}\right]^{\frac{1}{p}}
    \lesssim 
    \|T\|_{S^{p}(L^2(\mathbb R^n),L^2(\mathbb R^n) )}.
$$
\end{lemma}


\subsection{Muckenhoupt weights}

Let $w(x)$ be a nonnegative locally integrable function on $\mathbb R^n$. 
For $1 < p < \infty$, we say $w$ is an $A_p$ {weight}, written $w\in A_p$, if
$$
    [w]_{A_p} := \sup_Q \Bigg({1\over |Q|}\int_Q w(x) \,dx\Bigg)\, 
    \Bigg({1\over |Q|}\int_Q  {w(x)}^{-{1\over p-1}}\,dx\Bigg)^{p-1} < \infty,
$$
where the supremum is taken over all cubes $Q\subset \mathbb R^n$. 
The quantity $[w]_{A_p}$ is called the {$A_p$ constant of $w$}.

It is well known that $A_p$ weights are doubling. Namely,

\begin{lemma}[\cite{Gra1}] \label{double}
Let $w \in A_p$. Then for every $\lambda>1$ and for every cube $Q \subset \mathbb{R}^n$,
$$
    w(\lambda Q) \leq [w]_{A_p} \lambda^{pn} \,w(Q).
$$
\end{lemma}

In this article, we will also use the reverse H\"{o}lder inequality for $A_2$ weights.

\begin{lemma}[\cite{Gra1}]\label{reverse}
Let $w\in A_{2}$. 
There is a reverse doubling index $\sigma_{w}>0$, 
such that for every cube $Q\subset \mathbb{R}^n$, we have
\begin{equation}\label{eqn:reverse}
    \bigg[\frac{1}{|Q|}\int_{Q}w^{1+\sigma_{w}}(x)\,dx  \bigg]^{\frac{1}{1+\sigma_{w}}}
    \lesssim 
    \frac{w(Q)}{|Q|}.
\end{equation}
\end{lemma}

\begin{lemma}[\cite{Gra1}]\label{power}
Suppose $w\in A_p$ for some $p\in(1,\infty)$. 
For $0<\delta<1$, $w^{\delta}\in A_q$ for $q=\delta+1<2 $.
Moreover, $[w^\delta]_{A_q}\leq [w]_{A_p}^\delta$.
\end{lemma}


\subsection{Weight BMO and VMO spaces}

We recall the weighted BMO space and VMO space and their dyadic versions.

\begin{definition}\label{MWbmo}
Suppose $w \in A_\infty$. 
A function $b\in L^1_{\rm loc}(\mathbb R^n)$ belongs to ${ \rm BMO}_{w}(\mathbb R^n)$ if
$$
    \|b\|_{{\rm  BMO}_w(\mathbb R^n)}
    :=
    \sup_{B}{1\over w(B)}\int_{B} |b(x)-\langle b\rangle _B| \,dx<\infty,
$$
where the supremum is taken over all balls $B\subset \mathbb R^n$.  
The dyadic weighted {\rm BMO} space ${\rm  BMO}_{w}(\mathbb R^n,\mathcal D)$ associated with $\mathcal D$ 
consists of functions $b\in L^1_{\rm loc}(\mathbb R^n)$ such that 
$\|b\|_{{\rm  BMO}_{w}(\mathbb R^n,\mathcal D)}<\infty$, 
where the $\|b\|_{{\rm  BMO}_{w}(\mathbb R^n,\mathcal D)}$ is defined the same as above with intervals replaced by dyadic cubes.
\end{definition}

\begin{definition}\label{vmo}
Suppose $w \in A_\infty$. 
A function $b\in {\rm  BMO}_w(\mathbb R^n)$ belongs to ${\rm  VMO}_w(\mathbb R^n)$ if
\begin{itemize}
\item [(i)] $\displaystyle \ \lim_{a\to0}\sup_{B:\  r_B=a}{1\over w(B)}\int_{B} \left|b(x)-\langle b\rangle _B\right| dx=0$,

\item[(ii)]  $\displaystyle  \lim_{a\to\infty}\sup_{B:\  r_B=a}{1\over w(B)}\int_{B}\left |b(x)-\langle b\rangle _B\right| dx=0$,

\item[ (iii)] $\displaystyle \lim_{a\to\infty}\sup_{B\subset \mathbb R^n \backslash B(x_0,a)}
{1\over w(B)}\int_{B} \left|b(x)-\langle b\rangle _B\right| dx=0$,
\end{itemize}
where $x_0$ is any fixed point in $\mathbb R^n$.  
A function $b\in { \rm BMO}_{w}(\mathbb R^n,\mathcal D)$ belongs to the dyadic weighted {\rm VMO}  space
${ \rm VMO}_{w}(\mathbb R^n,\mathcal D)$ if the above three limits hold with intervals replaced by dyadic cubes in $\mathcal D$.
\end{definition}


\section{Weighted Besov space $\mathbf B_\nu^p(\mathbb R^n)$ and its dyadic versions}\label{sec:Besov}
\setcounter{equation}{0}

 As a companion to the definition of the weighted Besov space $\mathbf B_{\nu}^p(\mathbb R^n)$, in Definition \ref{def:Sequence},  
we have its dyadic versions on $\mathbb R^n$ defined here. 

\begin{definition}\label{def:Besov-endpoint-Delta}Suppose $0< p<\infty$.
Suppose $\nu \in A_2$ and $0<p<\infty$.   
Let $b\in L_{\rm loc}^1(\mathbb{R}^n)$ and $\D$ be an arbitrary dyadic system in $\mathbb{R}^n$. 
Then $b$ belongs to the weighted dyadic Besov space $\mathbf B_{\nu}^p(\mathbb{R}^n, \mathcal D)$ 
associated with  $\mathcal D$, if and only if
\begin{equation} \label{e:DyadicB}
    \|b\|_{\mathbf  B_{\nu}^p(\mathbb{R}^n, \mathcal D)} ^{p}:= 
    \sum_{Q\in \mathcal D}  \left( \frac{1}{ \nu(Q)}  \int_Q  \left|b(x)-\langle b\rangle _Q\right| dx \right)^p  <\infty.
\end{equation}
\end{definition}

This space has several alternate definitions, which are useful in different ways.  
Among them, an equivalent set of definitions in terms of Haar functions will be very useful 
in the analysis of the Haar shift paraproducts to come.   
 
\begin{proposition}  
Suppose $0<p<\infty$, $\mu,\lambda\in A_2$ and set $\nu=\mu^{\frac{1}{2}}\lambda^{-\frac{1}{2}}$.
We have the equivalences below  
\begin{equation}\label{eqn:equvi-character-seq}
     \|b\|_{\mathbf B_{\nu}^p(\mathbb{R}^n, \mathcal D)} ^{p}  
     \approx \sum_{Q\in \mathcal D:\, \epsilon\not\equiv 1}   
     \left( \frac {|\langle b, h_Q^\epsilon  \rangle|}{ \sqrt{\lvert Q \rvert}}  \right)^p
R(Q)^p  ,
\end{equation}
where the term $R(Q)$ can be any of the expressions below:
$$
    R(Q)\in\left\{\frac{|Q| }  {\nu(Q)} , \frac{|Q|}{\lambda^{-1}(Q)^{\frac{1}{2}}\mu(Q)^{\frac{1}{2}}}, 
    \frac{\nu^{-1}(Q)}{|Q|},\frac{\lambda(Q)^{\frac{1}{2}}\mu^{-1}(Q)^{\frac{1}{2}}}{|Q|}\right\}.
$$
\end{proposition}

\begin{proof}
For each $Q\in \mathcal D$, it follows from H\"older's inequality that 
$$
    \frac{|Q|}{\lambda^{-1}(Q)^{\frac{1}{2}}\mu(Q)^{\frac{1}{2}}}
    \leq \frac{|Q|}{\nu(Q)}
    \approx  \frac{\nu^{-1}(Q)}{|Q|}
    \leq \frac{\lambda(Q)^{\frac{1}{2}}\mu^{-1}(Q)^{\frac{1}{2}}}{|Q|},
$$
where the ``$\approx$'' above holds by $\nu\in A_2$. 
Similarly, the last term above is approximately equal to the first one, due to $\mu,\lambda \in A_2$.  

Therefore, it suffices to prove \eqref{eqn:equvi-character-seq} for taking $\displaystyle R_Q=\frac{|Q|}{\nu(Q)}$.

{\it Step I.}
Note that 
\begin{align*}
   \frac{1}{\nu(Q)}\int_Q \left|b(x)-\langle b\rangle _Q\right|dx 
   &= \frac{1}{\nu(Q)} \int_Q \bigg| \sum_{\substack{P\in \mathcal{D},P\subset Q\\\epsilon\not\equiv1}} 
   \langle b, h^{\epsilon}_{P}\rangle h^{\epsilon}_{P} \bigg| dx\\
   &\leq   \sum_{\substack{P\in \mathcal{D},P\subset Q\\\epsilon\not\equiv1}}  
   \frac{|\langle b, h^{\epsilon}_{P}\rangle|}{\sqrt{|P|}} \frac{|P|}{\nu(P)}   \cdot \frac{\nu(P)}{\nu(Q)}.
\end{align*}
When $0<p\leq 1$, by using concavity property of $(\sum |a_k|)^p\leq \sum |a_k|^p$, we  see that 
\begin{align*}
    \|b\|_{\mathbf B_{\nu}^p(\mathbb R^n,\mathcal D)}^p &\leq \sum_{Q\in \mathcal D}   
    \sum_{\substack{P\in \mathcal{D},P\subset Q\\\epsilon\not\equiv1}}  
    \left(\frac{ |\langle b, h^{\epsilon}_{P}\rangle| }{\sqrt{|P|}} \frac{|P|}{\nu(P)} \right)^p  
    \cdot \left(\frac{\nu(P)}{\nu(Q)}\right)^p\\
    & \leq \sum_{ P\in \mathcal{D}:\, \epsilon\not\equiv1}  
    \bigg[\left(\frac{ |\langle b, h^{\epsilon}_{P}\rangle|}{\sqrt{|P|}} \frac{|P|}{\nu(P)} \right)^p  
    \sum_{Q\in \mathcal{D},P\subset Q}  \left(\frac{\nu(P)}{\nu(Q)}\right)^p\bigg]\\
    & \leq \sum_{ P\in \mathcal{D}:\, \epsilon\not\equiv1}  
    \left(\frac{ |\langle b, h^{\epsilon}_{P}\rangle |}{\sqrt{|P|}} \frac{|P|}{\nu(P)} \right)^p.
\end{align*}

When $1<p<\infty$, the doubling property of $\nu\in A_2$ ensures that there exists 
$\theta=\theta(n,[\nu]_{A_2})\in (0,1)$, such that for any $Q\in \mathcal D$ and  $P\subset Q$ 
with $\ell(P)=2^{-k}\ell(Q)$,  we have $\nu(P)\leq \theta^k \nu(Q)$. Note that for any $1<p<\infty$,  
we can choose some $\alpha\in (0,1)$ such that $\alpha  p' >1$.  Then
\begin{align*}
    \sum_{\substack{P\in \mathcal D,  \,  P\subset Q }}   \bigg({\nu(P)\over \nu(Q)} \bigg)^{\alpha  p'} 
    \lesssim \sum_{k=0}^\infty \sum_{\substack{P\in \mathcal D,\, P\subset Q \\  \ell(P)=2^{-k} \ell(Q)  }}    
    \theta^{(\alpha p'-1)k }\frac{\nu(P)}{\nu(Q)}\lesssim  \sum_{k=0}^\infty  \theta^{(\alpha  p'-1)k }\lesssim 1.
\end{align*} 
Based on this, we further have 
\begin{align*}
   \|b\|_{\mathbf B_{\nu}^p(\mathbb R^n,\mathcal D)}^p &\leq \sum_{Q\in \mathcal D}  
   \Bigg( \sum_{\substack{P\in \mathcal{D},\, P\subset Q\\\epsilon\not\equiv1}}   
   \frac{ |\langle b, h^{\epsilon}_{P}\rangle| }{\sqrt{|P|}} 
   \frac{|P|}{\nu(P)}    \cdot  \frac{\nu(P)}{\nu(Q)}  \Bigg)^p\\
   &\leq  \sum_{Q\in \mathcal D} \Bigg[   \bigg(\sum_{\substack{P\in \mathcal{D},\, P\subset Q\\\epsilon\not\equiv1}} 
   \bigg({\nu(P)\over \nu(Q)} \bigg)^{\alpha  p'}\bigg)^{p-1} \cdot 
   \sum_{\substack{P\in \mathcal{D},\, P\subset Q\\\epsilon\not\equiv1}}   
   \left(\frac{ |\langle b, h^{\epsilon}_{P}\rangle |}{\sqrt{|P|}} \frac{|P|}{\nu(P)} \right)^p 
   \bigg({\nu(P)\over \nu(Q)} \bigg)^{(1-\alpha )p}\Bigg] \\
   & \leq \sum_{ P\in \mathcal{D}:\, \epsilon\not\equiv1}  
   \left(\frac{ |\langle b, h^{\epsilon}_{P}\rangle |}{\sqrt{|P|}} \frac{|P|}{\nu(P)} \right)^p.
\end{align*}

{\it Step II.}
On the other hand, by the cancellation  and size properties of Haar functions, 
\begin{align*}
     \sum_{Q\in \mathcal D:\, \epsilon\not\equiv 1}   
     \left( \frac {|\langle b, h_Q^\epsilon  \rangle|}{ \sqrt{\lvert Q \rvert}} \frac{|Q| }  {\nu(Q)}   \right)^p 
     &=   \sum_{Q\in \mathcal D:\, \epsilon\not\equiv 1}   
     \left( \frac {|\langle b-\langle b\rangle _Q, h_Q^\epsilon  \rangle|}{ \sqrt{\lvert Q \rvert}} \frac{|Q| }  {\nu(Q)}   \right)^p\\
     &\lesssim \sum_{Q\in \mathcal D}  \left(\frac{1}{\nu(Q)}\int_Q |b(x)-\langle b\rangle _Q| \,dx\right)^p,
\end{align*}
as desired.   The proof is complete.
\end{proof}

As a consequence,
the dyadic Besov spaces have alternative definitions in terms of martingale differences.

\begin{proposition}  \label{prop:Martingale}
Suppose $0<p<\infty$, $\mu,\,\lambda\in A_2$ and set $\nu=\mu^{\frac{1}{2}}\lambda^{-\frac{1}{2}}$.
We have the equivalent (semi-)norms
\begin{align} \label{e:BDequiv}
     \|b\|_{ \mathbf B_{\nu}^p(\mathbb{R}^n, \mathcal D)} ^{p}  
      & \approx \sum_{Q\in \mathcal{D}}     \left( \frac{1}{ \nu(Q) }  
     \int_Q{|\Delta _Q b(x)| }  \,dx\right)^{p}\nonumber \\
    & \approx \sum_{Q\in \mathcal{D}}     \left( \frac{1}{ \mu(Q) }  
     \int_Q{|\Delta _Q b(x)|^2 } \lambda (x) \, dx\right)^{p\over 2}  \nonumber\\
     &\approx \sum_{Q\in \mathcal{D}}     \left( \frac{1}{ \lambda^{-1}(Q) }  
     \int_Q{|\Delta _Q b(x)|^2 } \mu ^{-1}(x) \,dx\right)^{p\over 2} ,
\end{align}
where the implicit constants are independent of $b$.
\end{proposition}

\begin{proof} 
Note that for $Q\in \mathcal D$,
$$
    \frac{1}{\nu(Q)} \int_Q |\Delta_Q b(x) | \,dx \approx \sum_{\epsilon\not\equiv 1} 
    \frac {|\langle b, h_Q^\epsilon  \rangle|}{ \sqrt{\lvert Q \rvert}} \frac{|Q| }  {\nu(Q)},
$$
this is simply because all norms on a finite-dimensional space are equivalent 
(the space is $\ell^1$ over the children of $Q$ which has dimension $2^n$).

Hence, the first equivalence in \eqref{e:BDequiv} follows from \eqref{eqn:equvi-character-seq} 
with taking $R(Q)$ therin as $\displaystyle \frac{|Q|}{\nu(Q)}$.

Furthermore, note that the martingale difference $\Delta _Q b$ is constant on every child of $Q$. This yields that 
\begin{align*}
    \frac{1}{ \mu(Q) }  
     \int_Q{|\Delta _Q b(x)|^2 } \lambda (x)dx &\approx \frac{\lambda(Q)}{\mu(Q)|Q|}\int_Q |\Delta_Q b(x)|^2 dx\\
     &\approx    \sum_{\epsilon\not\equiv 1}  \frac{\lambda(Q)}{\mu(Q) |Q|}   
     \left(\frac {|\langle b, h_Q^\epsilon  \rangle|}{ \sqrt{\lvert Q \rvert}}\right)^2 |Q|\\
     &\approx  \sum_{\epsilon\not\equiv 1}  \left(    \frac {|\langle b, h_Q^\epsilon  \rangle|}{ \sqrt{\lvert Q \rvert}} 
     \frac{|Q|}{\lambda^{-1}(Q)^{1\over 2}\mu(Q)^{1\over 2}}\right)^2.
\end{align*}
Thus the second equivalence in \eqref{e:BDequiv} follows from \eqref{eqn:equvi-character-seq}
with taking $R(Q)$ therin as $\displaystyle \frac{|Q|}{\lambda^{-1}(Q)^{1\over 2}\mu(Q)^{1\over 2}}$.

The proof of the remaining third one is similar, and we skip it.
\end{proof}

The weighted Besov space is the intersection of a finite number of dyadic Besov spaces.  

\begin{theorem}\label{thm:Besov-Intersect}
Suppose $\nu\in A_2$ and $0< p<\infty$.
We have
$$ 
    \mathbf B_{\nu}^p(\mathbb{R}^n)
    =\bigcap_{\omega\in \left\{0,\frac{1}{3},\frac{2}{3}\right\}^n}  \mathbf B_{\nu}^p(\mathbb{R}^n, \mathcal D^\omega)
$$
with equivalent (semi-)norms
\begin{align*}
    \|b\|_{\mathbf B_\nu^{p}(\mathbb{R}^{n})}
    \approx  \sum_{\omega\in \left\{0,\frac{1}{3},\frac{2}{3}\right\}^n}\|b\|_{\mathbf B_{\nu}^{p}(\mathbb{R}^{n},\mathcal{D}^{\omega})}.
\end{align*}
\end{theorem}

\begin{proof}
Note that every adjacent dyadic system is constructed by shifting from the standard dyadic system; see Definition~\ref{def:adjacent-dyadic}. Hence
for any $\omega\in \left\{0,\frac{1}{3},\frac{2}{3}\right\}^n$ and $Q^\omega\in \mathcal D^\omega$,
there exists   $Q\in \mathcal D$ such that 
\begin{equation}\label{eqn:aux-comparable}
    Q^\omega \subset 20\sqrt{n}Q, \quad {\rm and }\quad |20\sqrt{n}Q|\leq C_n |Q^\omega|.
\end{equation}
Then the doubling property of $\nu\in A_2$ implies that 
\begin{align*}
    \frac{1}{ \nu(Q^\omega)}  \int_{Q^\omega} \left|b(x)-\langle b\rangle _{Q^\omega}\right| dx\lesssim \frac{1}{ \nu(20\sqrt{n}Q)}  \int_{20\sqrt{n}Q}  \left|b(x)- \langle b\rangle _{20\sqrt{n}Q}\right| dx .
\end{align*}
Moreover, by the size condition and the nestedness property of dyadic cubes, 
$$
     \sup_{Q\in \mathcal D} \#\big\{ Q^\omega\in  \mathcal D^\omega:\    Q^\omega {\rm satisfies}\   \eqref{eqn:aux-comparable}\big\} 
$$
is finite and only depends on $n$. Hence, $\|b\|_{\mathbf B_\nu^p(\mathbb R^n,\mathcal D^\omega)}\lesssim \|b\|_{\mathbf B_\nu^p(\mathbb R^n)}$.

On the other hand, the adjacent dyadic systems ensure that 
for any $Q\in \mathcal D$, there exists some $\omega\in \left\{0,\frac{1}{3},\frac{2}{3}\right\}^n$ and $\widetilde Q^\omega \in \mathcal D^\omega$ such that 
$$
     20\sqrt{n}Q\subset \widetilde Q^\omega , \quad |\widetilde Q^\omega|\leq C_n |20\sqrt{n}Q|.
$$
Similarly, the number of such each $20\sqrt{n}Q$ for $Q\in \mathcal D$ contained in a same $\widetilde Q^\omega$  is bounded uniformly, and 
$$
  \frac{1}{ \nu(20\sqrt{n}Q)}  \int_{20\sqrt{n}Q} \left|b(x)-\langle b\rangle _{20\sqrt{n}Q}\right| dx \lesssim \frac{1}{ \nu(\widetilde Q^\omega)}  \int_{\widetilde Q^\omega} \left|b(x)-\langle b\rangle_{\widetilde Q^\omega}\right| dx .
$$
Then
$$
\|b\|_{\mathbf B_\nu^{p}(\mathbb{R}^{n})}
    \lesssim \sum_{\omega\in \left\{0,\frac{1}{3},\frac{2}{3}\right\}^n}\|b\|_{\mathbf B_{\nu}^{p}(\mathbb{R}^{n},\mathcal{D}^{\omega})}.
$$
The proof is complete.
\end{proof}

Now we assure ourselves that our Besov spaces are contained in the VMO spaces.  

\begin{lemma}\label{lem:dyadic-Besov-VMO}
Suppose that $\nu\in A_2$ and $0< p<\infty$. Then 
$$
    \mathbf B_\nu ^p(\mathbb R^n,\mathcal D) \subset {\rm VMO}_\nu (\mathbb R^n, \mathcal D).
$$
\end{lemma}

\begin{proof}
Note that 
\begin{align*}
    \|b\|_{{\rm BMO}_\nu(\mathbb R^n,\mathcal D)}
    := & \sup_{Q\in \mathcal D}  \frac{1}{\nu(Q)} \int_Q |b(x)-\langle b\rangle_Q|  
       dx\leq \|b\|_{\mathbf B_\nu ^p(\mathbb R^n,\mathcal D)}
\end{align*}
for any  $0<p<\infty$.

Moreover, we are able to apply the analogous argument in \cite[Lemma 3.4]{LLW2022} 
to deduce that $\mathbf B_{\nu}^p(\mathbb R^n,\mathcal D)\subset {\rm VMO}_{v}(\mathbb R^n,\mathcal D)$ for $0< p<\infty$.
\end{proof}

\begin{remark}\label{remark B in VMO}
Note that from \cite{LPW},
$$
    {\rm BMO}_{\nu}(\mathbb R^n)
    =\bigcap_{\omega\in \left\{0,\frac{1}{3},\frac{2}{3}\right\}^n} {\rm BMO}_{\nu}(\mathbb R^n,\mathcal D^\omega)
    \quad {\rm and}\quad   
    {\rm VMO}_{\nu}(\mathbb R^n)
    =\bigcap_{\omega\in \left\{0,\frac{1}{3},\frac{2}{3}\right\}^n} {\rm VMO}_{\nu}(\mathbb R^n,\mathcal D^\omega) .
$$
Hence, by Lemma~\ref{lem:dyadic-Besov-VMO}, we have
$$
   \mathbf B_\nu^p(\mathbb R^n)\subset {\rm VMO}_{\nu}(\mathbb R^n)\quad {\rm for} \quad  0< p<\infty.
$$
Therefore,  it follows from \cite[Theorem 1.1]{LL2022} that when $b\in \mathbf B_\nu^p(\mathbb R^n)$, 
the commutator of the Riesz transform $[b,R_j]$ is compact from $L_\mu^2(\mathbb R^n)$ 
to $L_\lambda^2(\mathbb R^n)$, $j=1,2,\ldots, n$.
\end{remark}

\begin{remark}\label{rem:equiv-seq}
Suppose $0<p<\infty$, $\mu,\lambda\in A_2$ and set $\nu=\mu^{\frac{1}{2}}\lambda^{-\frac{1}{2}}$. 
We have the following equivalence for the Besov space $\mathbf B_{\nu}^p(\mathbb R^n)$:
\begin{align*}
    \|b\|_{\mathbf B_{\nu}^p(\mathbb R^n)}
    &\approx    
\bigg \|  \bigg\{\bigg[{1\over \mu(20\sqrt{n}Q)} \int_{20\sqrt{n}Q} 
    \left|b(x)- \langle b\rangle_{20\sqrt{n}Q}\right|^{2} \lambda(x) dx\bigg]^{1\over 2}
    \bigg\}_{Q\in\mathcal D}  \bigg \|_{ \ell^{p}}\\
    &\approx   
     \bigg \| \bigg\{ \bigg[\frac{1}{\lambda^{-1}(20\sqrt{n}Q)} 
    \int_{20\sqrt{n}Q} \left|b-\langle b\rangle_{20\sqrt{n}Q}\right|^{2} 
    \mu^{-1} (x) \, dx \bigg]^{1\over 2} \bigg\}_{Q\in\mathcal D}\bigg \|_{ \ell^{p}}.
\end{align*}
This is implied by the proof of Lemma~\ref{lem aux Sobolev new} below. See Remark~\ref{rem:equiv-Lorentz} for further results.
\end{remark}


\section{Schatten class estimate for $n<p<\infty$: proof of Theorem \ref{thm:main1}}\label{sec:thm1}
\setcounter{equation}{0}


\subsection{Sufficiency}  \label{sub:sufficiency}

We are studying the Schatten norms of the commutators in the case that $p$ is larger than the dimension $n$.  
We first address the sufficiency of the symbol being in the Besov space for the Schatten norm estimate.  

\begin{proposition}\label{prop:commut-Riesz-1}
Suppose  $n<  p<\infty$, $\mu,\,\lambda\in A_2$ and set $\nu=\mu^{\frac{1}{2}}\lambda^{-\frac{1}{2}}$.  
Let  $b\in \mathbf B_{\nu}^p(\mathbb R^n)$, then we have 
$$
      \left\|[b,R_j]\right\|_{S^p \big(L_{\mu}^2(\mathbb R^n), L_{\lambda}^2(\mathbb R^n)\big)}
      \lesssim \|b\|_{\mathbf B_{\nu}^p(\mathbb R^n)}.
$$ 
\end{proposition}

We depend on the work of  Petermichl, Treil and Volberg \cite{PTV}. 
They have shown that Riesz transforms are averages of relatively simple dyadic shifts.  
Given a dyadic system $\mathcal D$ with Haar basis $\{h_Q^\epsilon\}$, 
let $\sigma:\, \mathcal D\to \mathcal D$ satisfying $|\sigma(Q)|=2^{-n}|Q|$ for all $Q\in \mathcal D$. 
Using the same notation for a map 
$\sigma:\, \{0,1\}^n\setminus \{1\}^n \to (\{0,1\}^n\setminus \{1\}^n)\bigcup \{0\}$, 
if $\sigma(\epsilon)=0$ then set $h^{\sigma(\epsilon)}:=0$. 
The resulting dyadic shift operator is 
$$
  \Sha  f(x):=\sum_{Q\in\mathcal D,\epsilon\not\equiv1}\langle f, h^{\epsilon}_{Q} \rangle h^{\sigma(\epsilon)}_{\sigma(Q)}(x).
$$
It is known that for any $w\in A_2$, $\|\Sha\|_{L_w^2(\mathbb R^n)\to L_w^2(\mathbb R^n)}\lesssim 1$. 
Moreover, the Riesz transforms are in the convex hull of the class of operators $\Sha$, see \cite{PTV}. 
This result involves a random choice of dyadic systems, which implies that whenever considering the commutator $[b,R_j]$, 
it suffices to prove that the norms of $[b,\Sha]$ associated to different dyadic systems $\mathcal D$, 
are uniformly controlled with respect to the choice of $\mathcal D$.  
 
The commutator with the Haar shift operator $\Sha$ can be represented in terms of the paraproducts and $\Sha$:
$$
    [b,\Sha]f=\left(\Pi_b^{\mathcal D} +\Pi_b^{*\mathcal D} +\Gamma_b^{\mathcal D}\right) (\Sha f)
    -\Sha\left(\Pi_b^{\mathcal D } +\Pi_b^{*\mathcal D}+\Gamma_b^{\mathcal D}\right)f 
    +\Pi_{\Sha f}^{\mathcal D} b-\Sha(\Pi_f^{\mathcal D}b),
$$
where  
$$
   \Pi_b^{\mathcal D} f
   :=\sum_{Q\in \mathcal D,\ \epsilon\not \equiv 1} 
   \langle b, h_Q^\epsilon\rangle  \langle f\rangle_Q h_Q^\epsilon,
  \qquad 
   \Pi_b^{*\mathcal D} f
   :=\sum_{Q\in \mathcal D,\ \epsilon\not\equiv 1} 
   \langle b,h_Q^\epsilon\rangle \langle f, h_Q^\epsilon \rangle \frac{\mathsf{1}_Q}{|Q|},
$$
and 
$$
    \Gamma_b^{\mathcal D} f
    :=\sum_{Q\in \mathcal D}\sum_{\epsilon,\eta\not\equiv 1,\, \epsilon\not\equiv \eta} 
    \langle b,h_Q^\epsilon\rangle \langle f, h_Q^\eta\rangle h_Q^\epsilon  h_Q^\eta.
$$
See \cite{HLW2017} for details.
The point of this expansion is that it gives us a finite sum of paraproducts composed with $L^2$ bounded operators. 
To give one example, we have 
\begin{align*}
    \left \lVert \Pi_b^{\mathcal D}  (\Sha \cdot ) \right \rVert
     _{S^p \big(L_{\mu}^2(\mathbb R^n) , L_{\lambda}^2(\mathbb R^n)\big)}
     & \lesssim  \left\lVert  \Pi_b^{\mathcal D}  ( \cdot ) \right\rVert 
     _{S^p \big(L_{\mu}^2(\mathbb R^n) , L_{\lambda}^2(\mathbb R^n)\big)} 
     \cdot    \left \lVert \Sha \right \rVert_{L_{\mu}^2(\mathbb R^n)\to L_{\mu}^2(\mathbb R^n)}\\
     & \lesssim  \left \lVert \Pi_b^{\mathcal D}  ( \cdot )  \right\rVert
      _{S^p \big(L_{\mu}^2(\mathbb R^n) , L_{\lambda}^2(\mathbb R^n)\big)} . 
\end{align*}
This depends upon the condition that $\mu  \in A_2$. 
If the operator $\Sha$ comes \textit{before} the paraproduct, 
a similar inequality holds, and we use the condition that $\lambda \in A_2$.  
As a consequence, we have 
\begin{align*}
    \left\| [b,\Sha]  \right \|_{S^p\big(L_{\mu}^2(\mathbb R^n) , L_{\lambda}^2(\mathbb R^n)\big)} 
    \lesssim &  \big \|\lambda^{\frac{1}{2}}\Pi_b^{\mathcal D}\mu^{-\frac{1}{2}}\big \|_{S^p(L^2(\mathbb R^n), L^2(\mathbb R^n))} 
    +\big \|\lambda^{\frac{1}{2}}\Pi_b^{* \mathcal D}\mu^{-\frac{1}{2}}\big \| _{S^p(L^2(\mathbb R^n), L^2(\mathbb R^n))}\\[4pt]
    & +  \big \|\lambda^{\frac{1}{2}}\Gamma_b^{\mathcal D}\mu^{-\frac{1}{2}}\big \|_{S^p(L^2(\mathbb R^n), L^2(\mathbb R^n))}  
    +  \big \|\lambda^{\frac{1}{2}}\mathcal R_b^{\mathcal D}\mu^{-\frac{1}{2}}\big \|_{S^p(L^2(\mathbb R^n), L^2(\mathbb R^n))},
\end{align*}
where 
\begin{equation}\label{e:R}
    \mathcal  R_b^{\mathcal D} f:=\Pi_{\Sha f}^{\mathcal D} b-\Sha(\Pi_f^{\mathcal D}b).
\end{equation}

To continue, we have to estimate the $S^p(L^2(\mathbb R^n) , L^2(\mathbb R^n))$ norm for each of the four terms above.
We address three of them here. 

\begin{proposition}\label{prop:para-Besov}
Suppose  $n<  p<\infty$, $\mu,\,\lambda\in A_2$ and set $\nu=\mu^{\frac{1}{2}}\lambda^{-\frac{1}{2}}$. 
Let  $b\in {\rm VMO}_{\nu}(\mathbb R^n)$. 
Then $\lambda^{\frac{1}{2}}\Pi_b^{\mathcal D} \mu^{-\frac{1}{2}}$, 
$\lambda^{\frac{1}{2}}\Pi_b^{*\mathcal D}\mu^{-\frac{1}{2}}$ and 
$\lambda^{\frac{1}{2}}\Gamma_b^{\mathcal D}\mu^{-\frac{1}{2}}$ 
belong to $S^p(L^2(\mathbb R^n), L^2(\mathbb R^n))$ respectively, 
if and only if $b\in \mathbf B_\nu^p(\mathbb R^n, \mathcal D)$, that is,
\begin{align}
    \big \|\lambda^{1\over 2}\Pi_b^{\mathcal D}\mu^{-{1\over 2}}\big \|_{S^p(L^2(\mathbb R^n), L^2(\mathbb R^n))}  
    \approx &\, \|b\|_{\mathbf B_\nu^p(\mathbb R^n, \mathcal D)},\label{eqn:paraproduct-1}\\[4pt]
   \big \|\lambda^{1\over 2}\Pi_b^{* \mathcal D}\mu^{-{1\over 2}}\big \|_{S^p(L^2(\mathbb R^n), L^2(\mathbb R^n))}  
   \approx & \, \|b\|_{\mathbf B_\nu^p(\mathbb R^n, \mathcal D)},   \label{eqn:paraproduct-*}  \\[4pt]
   \big \|\lambda^{1\over 2}\Gamma_b^{\mathcal D}\mu^{-{1\over 2}}\big \|_{S^p(L^2(\mathbb R^n), L^2(\mathbb R^n))}  
   \approx &\,  \|b\|_{\mathbf B_\nu^p(\mathbb R^n, \mathcal D)} \label{eqn:Gamma}
\end{align}
with the implicit constants depending only on $[\mu]_{A_2}$, $[\lambda]_{A_2}$ and $n$, 
regardless of the choice of $\mathcal D$.
\end{proposition}

For our main theorem, we only need the upper bounds, or sufficiencies, above. 
We include the lower bounds for completeness. 

\begin{proof}
Both directions will depend on the nearly weakly orthogonal sequence approach of Rochberg and Semmes. 

\smallskip
  
\noindent {\bf Sufficiency for \eqref{eqn:paraproduct-1}:} 
Suppose $b\in \mathbf B_\nu^p(\mathbb R^n,\mathcal D)$.  
It follows from the definition of $\Pi_b^{\mathcal D}$ that 
\begin{align*}
    \lambda^{1\over 2}\Pi_b^{\mathcal D}\mu^{- {1\over 2}}(f)(x)  
    =& \sum_{Q\in \mathcal D,\epsilon\not\equiv 1}  \lambda^{1\over 2}(x) \langle b,h_Q^\epsilon\rangle h_Q^\epsilon(x) 
    \frac{1}{|Q|}\int_{\mathbb R^n} \mathsf{1}_Q(y) \mu^{-  {1\over 2}} (y) f(y) dy\\
    =& \sum_{Q\in \mathcal D,\epsilon\not\equiv 1}  \frac{\langle b, h_Q^\epsilon\rangle 
    \lambda(Q)^{1\over 2} \mu^{-1}(Q)^{1\over 2}}{|Q|^{ {1 \over 2} +1 }} \cdot 
    \frac{\lambda^{\frac{1}{2}}(x) |Q|^{1\over 2} h_Q^\epsilon (x)}{\lambda(Q)^{1\over 2}} \cdot 
    \left\langle  f, \frac{\mu^{-{1\over 2}}(y) 
    \mathsf{1}_Q(y) }{\mu^{-1}(Q)^{\frac{1}{2}}}\right\rangle.
\end{align*}
Denote
$$
    G_Q(x):=\frac{\lambda^{\frac{1}{2}}(x) |Q|^{1\over 2} h_Q^\epsilon(x) }{\lambda(Q)^{1\over 2}},
    \quad\textnormal{and}\quad 
    H_Q(x):=\frac{\mu^{-{1\over 2}}(x)   \mathsf{1}_Q(x)}{\mu^{-1}(Q)^{\frac{1}{2}}}.
$$

Using the reverse H\"{o}lder inequality for $A_2$ weights in Lemma \ref{reverse}, 
both $\{G_{Q}\}_{Q\in \mathcal D}$ and $\{H_{Q}\}_{Q\in \mathcal D}$ are NWO sequences for $L^2(\mathbb R^n)$, 
which gives using \eqref{eq-NWO1} that 
$$
    \big \|\lambda^{1\over 2}\Pi_b^{\mathcal D}\mu^{-{1\over 2}}
    \big \|_{S^p(L^2(\mathbb R^n), L^2(\mathbb R^n))} 
    \lesssim \|b\|_{\mathbf B_\nu^p(\mathbb R^n,\mathcal D)}.
$$

\smallskip

\noindent{\bf Necessity  for \eqref{eqn:paraproduct-1}:} Suppose $\lambda^{1\over 2}\Pi_b^{\mathcal D}\mu^{-{1\over 2}}\in   
S^p(L^2(\mathbb R^n), L^2(\mathbb R^n))$. Note that for each $Q\in \mathcal D$ and 
$\epsilon\not\equiv 1$,  $\langle b, h_Q^\epsilon\rangle =\langle \Pi_b^{\mathcal D} (\mathsf{1}_Q), h_Q^\epsilon\rangle $. 
Then
$$
    \sum_{Q\in \mathcal D,\, \epsilon\not\equiv 1}   \left|  \frac{\langle b, h_Q^\epsilon\rangle}{|Q|^{1/2}} 
    \frac{\lambda(Q)^{\frac{1}{2}}  \mu^{-1}(Q) ^{\frac{1}{2}}}{ |Q|} \right|^p  
    =  \sum_{Q\in \mathcal D,\, \epsilon\not\equiv 1}  \left|  \left\langle \lambda^{1\over 2} 
    \Pi_b^{\mathcal D} \mu^{- {1\over 2}} \Big(\frac{\mu^{1\over 2}
    \mathsf{1}_Q   }{\mu(Q)^{1\over 2}} \Big), \,
    \frac{\lambda^{-{1\over 2}} |Q|^{1\over 2}h_Q^\epsilon}{\lambda^{-1}(Q)^{1\over 2}}\right\rangle \right|^p.
$$
Denote
$$
    G^{'}_Q(x):=\frac{\mu^{1\over 2}(x)  \mathsf{1}_Q(x)}{\mu(Q)^{1\over 2}} ,
    \quad H^{'}_Q(x):=\frac{\lambda^{-{1\over 2}}(x) |Q|^{1\over 2}h_Q^\epsilon(x)}{\lambda^{-1}(Q)^{1\over 2}}.
$$
Similarly, these two collections of functions are also NWOs.  Since $1<p<\infty$, by Lemma \ref{eq-NWO} it holds that
$$
    \|b\|_{\mathbf B_\nu^p(\mathbb R^n,\mathcal D)}^p\lesssim  \big \|\lambda^{1\over 2}
    \Pi_b^{\mathcal D}\mu^{-{1\over 2}}\big \|_{S^p(L^2(\mathbb R^n), L^2(\mathbb R^n))}^p .
$$
That is, we have established \eqref{eqn:paraproduct-1}. The second estimate \eqref{eqn:paraproduct-*} is dual to the first, 
so it holds.

\smallskip

It remains to prove the last estimate \eqref{eqn:Gamma}.

\smallskip

\noindent {\bf Sufficiency for \eqref{eqn:Gamma}:} Suppose $b\in \mathbf  B_{\nu}^p(\mathbb R^n,\mathcal D)$. 
By the definition of $\Gamma_b^{\mathcal D}f$, we have 
\begin{align*}
    &\left(\lambda^{1\over 2}\Gamma_b^{\mathcal D}\mu^{-{1\over 2}}\right)f(x) \\
    =&\sum_{Q\in \mathcal D}\sum_{\epsilon,\eta\not\equiv 1,\, \epsilon\not\equiv \eta} 
    \langle b,h_Q^\epsilon  \rangle \left \langle \mu^{-{1\over 2}} f, h_Q^\eta \right\rangle
    h_Q^\epsilon(x) h_Q^\eta(x)\lambda^{1\over 2}(x)\\
    =&\sum_{Q\in \mathcal D}\sum_{\epsilon,\eta\not\equiv 1,\, \epsilon\not\equiv \eta} 
    \frac{\langle b, h_Q^\epsilon\rangle \lambda(Q)^{1\over 2}\mu^{-1}(Q)^{1\over 2}}{|Q|^{{1\over 2}+1}} 
    \cdot \frac{h_Q^\epsilon (x) h_Q^\eta(x) \lambda^{1\over 2}(x) |Q|}{\lambda(Q)^{1\over 2}} \cdot \left\langle f, \frac{\mu^{-{1\over 2}} 
    h_Q^\eta |Q|^{1\over 2}}{\mu^{-1}(Q)^{\frac{1}{2}}}\right \rangle.
\end{align*}
Denote 
$$
    G^{''}_Q(x):=\frac{h_Q^\epsilon(x)  h_Q^\eta(x) \lambda^{1\over 2}(x) |Q|}
    {\lambda(Q)^{1\over 2}} ,
    \quad H^{''}_Q:=\frac{\mu^{-{1\over 2}}(x) h_Q^\eta(x) |Q|^{1\over 2}}{\mu^{-1}(Q)^{\frac{1}{2}}}.
$$

Similarly, 
these two collections of functions are also NWOs. 

Therefore, using \eqref{eq-NWO1} 
$$
   \big \|\lambda^{1\over 2}\Gamma_b^{\mathcal D}\mu^{-{1\over 2}}\big\|_{S^p(L^2(\mathbb R^n), L^2(\mathbb R^n))}^p
   \lesssim  \|b\|_{\mathbf B_\nu^p(\mathbb R^n, \mathcal D)}^p.
$$

\medskip

\noindent{\bf Necessity for \eqref{eqn:Gamma}:} 
For any $Q\in \mathcal D$ and $\epsilon\not\equiv 1$,
 note that for each $\eta\not\equiv 1$ and $\eta\not\equiv \epsilon$,
$$
    \left\langle \Gamma_b^{\mathcal D} (h_Q^\eta), h_Q^\epsilon h_Q^\eta |Q| \right\rangle 
    =  \left\langle \sum_{\epsilon'\not\equiv 1,\, \epsilon'\not\equiv \eta} 
    \langle b, h_Q^{\epsilon'} \rangle h_Q^{\epsilon'}h_Q^\eta, h_Q^\epsilon h_Q^\eta |Q| \right\rangle 
    =\langle b, h_Q^\epsilon\rangle .
$$
Therefore,
\begin{align*}
    \|b\|_{\mathbf B_{\nu}^p(\mathbb R^n,\mathcal D)}^p 
    \approx &\sum_{Q\in\mathcal D}\sum_{\epsilon,\eta\not\equiv 1,\, \epsilon\not\equiv \eta} 
    \left(\frac{\left\langle \Gamma_b^{\mathcal D} (h_Q^\eta), h_Q^\epsilon h_Q^\eta |Q| \right\rangle |Q| }
    {|Q|^{1\over 2}\lambda^{-1}(Q)^{1\over 2} \mu(Q)^{1\over 2}}\right)^p\\
    =&  \sum_{Q\in\mathcal D}\sum_{\epsilon,\eta\not\equiv 1,\, \epsilon\not\equiv \eta} 
    \left(\frac{\left\langle \lambda^{1\over 2}\Gamma_b^{\mathcal D} \mu^{-{1\over 2}} 
    (\mu^{1\over 2}h_Q^\eta), \lambda^{-{1\over 2}} h_Q^\epsilon h_Q^\eta |Q| \right\rangle  |Q| } 
    {  |Q|^{1\over 2}\lambda^{-1}(Q)^{1\over 2} \mu(Q)^{1\over 2}}\right)^p\\
    =&  \sum_{Q\in\mathcal D}\sum_{\epsilon,\eta\not\equiv 1,\, \epsilon\not\equiv \eta}  
    \left|   \left\langle \lambda^{1\over 2}\Gamma_b^{\mathcal D} \mu^{-{1\over 2}}  
    \left(\frac{\mu^{1\over 2}h_Q^\eta |Q|^{1\over 2}}{\mu(Q)^{1\over 2}}\right), 
    \frac{\lambda^{-{1\over 2}}h_Q^\epsilon h_Q^\eta |Q| }
    {\lambda^{-1}(Q)^{1\over 2}} \right\rangle   \right|^p.
\end{align*}
Similarly, we have by Lemma \ref{eq-NWO}
$$
    \|b\|_{\mathbf B_{\nu}^p(\mathbb R^n,\mathcal D)}^p 
    \lesssim \left\|\lambda^{1\over 2}\Gamma_b^{\mathcal D}
    \mu^{-{1\over 2}}\right\|_{S^p(L^2(\mathbb R^n), L^2(\mathbb R^n))}^p .
$$
This completes the argument.
\end{proof}

The last of the four terms to address is defined in \eqref{e:R}, which is done here. We only assert the sufficiency. 

\begin{proposition}\label{prop:Rb}
Suppose  $n<  p<\infty$, $\mu,\,\lambda\in A_2$ and set $\nu=\mu^{\frac{1}{2}}\lambda^{-\frac{1}{2}}$. 
Let  $b\in \mathbf B_{\nu}^p(\mathbb R^n, \mathcal D)$ and denote  
$\mathcal  R_b^{\mathcal D} f:=\Pi_{\Sha f}^{\mathcal D} b-\Sha(\Pi_f^{\mathcal D}b)$. 
We have
$$
    \left\|\lambda^{1\over 2}\mathcal R_b^{\mathcal D}\mu^{-{1\over 2}}\right\| 
    _{S^p(L^2(\mathbb R^n), L^2(\mathbb R^n))}\lesssim \|b\|_{\mathbf B_{\nu}^p(\mathbb R^n,\mathcal D)},
$$
with the implicit constants depending only on $[\mu]_{A_2}$, $[\lambda]_{A_2}$ and $n$, regardless of the choice of $\mathcal D$.
\end{proposition}

\begin{proof}
By definitions of $\Sha$ and the paraproduct operator $\Pi_b^{\mathcal D}$,
\begin{align*}
    \mathcal  R_b^{\mathcal D} f 
    =& \sum_{Q\in \mathcal D,\,\epsilon\not\equiv 1} \left\langle  \sum_{P\in \mathcal D,\, \eta\not\equiv 1} 
    \langle f, h_P^\eta \rangle h_{\sigma(P)}^{\sigma(\eta)} , \,  h_Q^\epsilon \right\rangle 
    \langle b \rangle_Q h_Q^\epsilon  - \sum_{Q\in \mathcal D,\, \epsilon\not\equiv 1} 
    \left\langle \sum_{P\in \mathcal D,\, \eta\not\equiv 1} \langle f, h_P^\eta \rangle 
    \langle b \rangle_P h_P^\eta ,    \,   h_Q^\epsilon \right\rangle h_{\sigma(Q)}^{\sigma(\epsilon)}\\
     =& \sum_{P\in \mathcal D,\,\eta\not\equiv 1} \langle f, h_P^\eta \rangle \langle b\rangle_{\sigma(P)} 
     h_{\sigma(P)}^{\sigma(\eta)} -\sum_{Q\in \mathcal D,\, \eta\not\equiv 1} 
     \langle f, h_Q^\epsilon \rangle  \langle b \rangle_Q  h_{\sigma(Q)}^{\sigma(\epsilon)} \\
     =&  \sum_{Q\in \mathcal D,\, \eta\not\equiv 1} \langle f, h_Q^\epsilon \rangle     
     \left(   \langle b \rangle_{\sigma(Q)}  - \langle b \rangle_Q \right) h_{\sigma(Q)}^{\sigma(\epsilon)} .
\end{align*}
Recall that 
$\displaystyle \langle b\rangle_P=\sum_{R\in \mathcal D,  \, P\subsetneqq R,\,\eta\not\equiv 1} 
\langle b, h_R^\eta\rangle h_R^\eta(P)$, 
where $h_R^\eta(P)$ denotes the constant value of $h_R^\eta$ on $P$, thus 
$$
    \langle b \rangle_{\sigma(Q)}  - \langle b \rangle_Q
    =\sum_{\eta\not\equiv 1} \langle b, h_Q^\eta\rangle h_Q^\eta(\sigma(Q)).   
$$
Therefore,
\begin{align*}
    &\lambda^{1\over 2}(x)\mathcal  R_b^{\mathcal D} (\mu^{-{1\over 2}} f)(x)\\
    =& \sum_{Q\in \mathcal D,\, \epsilon,\eta\not\equiv 1} \langle \mu^{-{1\over 2}}f, h_Q^\epsilon\rangle  
    \lambda^{1\over 2}(x) \langle b, h_Q^\eta \rangle h_Q^\eta(\sigma(Q)) h_{\sigma(Q)}^{\sigma(\epsilon)}(x)\\
    =& \sum_{Q\in \mathcal D,\, \epsilon,\eta\not\equiv 1}   \frac{\langle b, h_Q^\eta \rangle 
    h_Q^\eta(x) (\sigma(Q)) \lambda(Q)^{1\over 2} \mu^{-1}(Q)^{1\over 2}}{|Q|} \cdot 
    \frac{ \lambda^{1\over 2}(x)h_{\sigma(Q)}^{\sigma(\epsilon)}(x)  |Q|^{1\over 2}}{\lambda(Q)^{1\over 2}} \cdot 
    \left\langle \frac{\mu^{-{1\over 2}}h_Q^\epsilon |Q|^{1\over 2}}{\mu^{-1}(Q)^{1\over 2}}, \, f\right\rangle.
\end{align*}
Due to $|h_Q^\eta (\sigma(Q)) |, \, |h_{\sigma(Q)}^{\sigma(\epsilon)}|, \, |h_Q^\epsilon|\approx |Q|^{-{1\over 2}}$, 
similar to the argument in Proposition \ref{prop:para-Besov}, we have
$$
    \left\|\lambda^{1\over 2}\mathcal R_b^{\mathcal D}\mu^{-{1\over 2}}\right\|_{S^p(L^2(\mathbb R^n), L^2(\mathbb R^n))}^p
    \lesssim   \sum_{Q\in \mathcal D,\,  \eta\not\equiv 1}   
    \left|\frac{\langle b, h_Q^\eta \rangle h_Q^\eta (\sigma(Q)) \lambda(Q)^{1\over 2} \mu^{-1}(Q)^{1\over 2}}{|Q|} \right|^p
    \lesssim \|b\|_{\mathbf B_\nu^p (\mathbb R^n,\mathcal D)}^p,
$$
as desired.
\end{proof}

\smallskip

Combining Propositions \ref{prop:para-Besov}, \ref{prop:Rb}, we have 
$$
   \big\|[b,R_j]\big\|_{S^p \big(L_{\mu}^2(\mathbb R^n), L_{\lambda}^2(\mathbb R^n)\big)}
   =\big\|\lambda^{1\over 2}\mathcal  [b, R_j]\mu^{-{1\over 2}} \big\|_{S^p(L^2(\mathbb R^n), L^2(\mathbb R^n))} 
   \lesssim  \|b\|_{\mathbf B_\nu^p (\mathbb R^n)} 
$$
for $n<p<\infty$. 
That is, the Schatten norm is controlled by the Besov space norm, 
which is the conclusion of  Proposition~\ref{prop:commut-Riesz-1}. 


\subsection{Necessity} \label{sub:necessity}

We turn to the necessity of the Besov norm condition. This will complete the proof of Theorem~\ref{thm:main1}. 

\begin{proposition}\label{prop:Besov-2}
Suppose $\mu,\lambda\in A_2$ and set $\nu=\mu^{\frac{1}{2}}\lambda^{-\frac{1}{2}}$ for $n<  p<\infty$. 
Suppose $b\in {\rm VMO}_{\nu}(\mathbb R^n)$ with 
$ \left\|[b,R_j]\right\|_{S^p \big(L_{\mu}^2(\mathbb R^n), L_{\lambda}^2(\mathbb R^n)\big)}<\infty$, 
then $b\in \mathbf B_{\nu}^p(\mathbb R^n)$ with 
$$
    \|b\|_{\mathbf B_{\nu}^p(\mathbb R^n)} 
    \lesssim \left\|[b,R_j]\right\|_{S^p \big(L_{\mu}^2(\mathbb R^n), L_{\lambda}^2(\mathbb R^n)\big)}.
$$ 
\end{proposition}

We setup an alternate norm for the Besov space, adapted to a particular Riesz transform.

\begin{proposition}\label{p:BesovRiesz} 
There are absolute constants and integer $N$ so that this holds for any grid $\mathcal D$.   
Fix  $1\leq j \leq n$. 
For all $Q\in  \mathcal D_k$, for integer $k$, there are  functions 
\begin{equation} \label{e:gPP}
    g _{Q,l}  =  2 ^{(k+1)n/2} ( \mathsf{1}_{Q(l,1)}- \mathsf{1}_{Q(l,2)}) , 
    \qquad 1\leq l \leq N .  
\end{equation}
where 
\begin{itemize}
\item[(i)] $Q(l,1), Q (l,2)$ are contained in $Q$, and elements of $\mathcal D _{k+2}$;

\item[(ii)]   The span of the $\{g _{Q,l}\}$ equals the range of $\mathsf{1}_Q (E_{k+2} \cdot - E_k \cdot ) $;

\item[(iii)]   $K_j (x_1 - x_2)$ does not change signs where   $x_s\in Q(l,s)$, and   $K_j$ is the kernel of the $j$th Riesz transform; 
 
\item[(iv)]   $\lvert K_j (x_1-x_2) \rvert \gtrsim \lvert Q \rvert ^{-1}$, where   $x_s\in Q(l,s)$;

\item[(v)]   We have the equivalence of norms 
\begin{equation} \label{e:equiv}  
    \|b\|_{\mathbf B_{\nu}^p(\mathbb{R}^n, \mathcal D)} ^{p} 
    \simeq \sum_{Q\in \mathcal D}  \sum_{l=1} ^{N}  { \left(\frac{|Q|}{\nu(Q) }\right)^{p}  }  \left(\frac{|\langle b,g _{Q,l}\rangle|}{\sqrt{|Q|}}\right)^p .
 \end{equation}
\end{itemize}
\end{proposition}

\begin{proof} 
Since the weights in question are doubling, the $g _{Q,s}$ have the spanning property listed above, 
 and all norms on finite dimensional space are equivalent, we have 
$$
    \|b\|_{\mathbf B_{\nu}^p(\mathbb{R}^n, \mathcal D)} ^{p}  
    \simeq \sum _{k\in \mathbb{Z}   } 
    \sum_{Q\in \mathcal{D}_k }  \left(\frac{1}{ \nu(Q) }  
    \int_Q{| E_{k+2} b - E _{k} b  | }\; dx\right)^{p} .
$$

In comparison to our definition of the dyadic Besov space in \eqref{e:DyadicB}, 
we are using the difference of conditional expectations, which differ by $2$ instead of $1$.  
Then, by the spanning properties of the $g _{Q,l}$, the equivalence \eqref{e:equiv} follows. 

The approach is to construct the $g _{Q,l}$ meeting all but property \eqref{e:equiv}. 
And, in addition, the span of the $\{g _{Q,l} \}$ is the range of $ \mathsf{1}_Q (  E_{k+2} \cdot - E _{k} \cdot )$. 

We can fix $j=1$, for which the kernel of the Riesz transform is  
$K_1 (y) = \frac {y_1}{\lvert y \rvert ^{n+1}}$, for $y\neq 0$. 
Fix $Q\in \mathcal D_k$.  The vector space of the range of $\mathsf{1}_Q (E_{k+2} \cdot - E _{k} \cdot)$ 
has dimension $2^{2n}-1$, and consists of functions with zero integral, constant on the grandchildren of $Q$.  
The grandchildren of $Q$ are the cubes  $\mathcal P(Q)$ consisting of 
all the cubes $P\in \mathcal D _{k+2}$, which are contained in $Q$. 
Let $x_P$ be the center of $P$. Then our  basis is the collection of functions 
\begin{equation} 
    g_{P_1, P_2} \coloneqq   2 ^{(k+1)n/2} ( \mathsf{1}_{P_1} - \mathsf{1}_{P_2}), 
    \qquad  
    P_1, P_2\in \mathcal P(Q),\ \lvert (x_{P_1} - x_{P_2}) \cdot (1, 0,\ldots , 0) \rvert \geq 2 ^{-k-1}. 
\end{equation}
That is, we form the difference of the indicator function of two cubes in $\mathcal P(Q)$ 
provided they are \textit{not adjacent in the first coordinate}. 
Call the pairs $(P_1, P_2) $ that meet this condition \textit{admissible}.

$\{g_{P_1,P_2}\}$ is certainly a system of functions, bounded in number independently of the choice of $Q\in \mathcal D_k$. 
They satisfy the cancellation and size conditions claimed by Lemma~\ref{haar}, 
and satisfy the conditions relative to the kernel $K_1$. 
It is not immediately clear that their span is as required, namely 
the range of $ \mathsf{1}_Q (  E_{k+2} \cdot - E _{k} \cdot )$.  
This is an elementary argument, which completes this proof. 

We claim that every  $g_{P_1,P_2}$  with $P_1, P_2\in \mathcal P(Q)$ is in the linear span of the functions with admissible  pairs. 
This immediately implies the required spanning property.  
There are two cases. 
First, suppose that $P_1 \neq  P_2\in \mathcal P(Q)$  such that there is a 
a third cube $P_3$ so that $ (P_1,P_3)$ and $(P_2,P_3)$ are admissible. 
It is clear that $g _{P_1, P_2}$ is in the linear span of $g _{P_1,P_3}$ and $g _{P_2, P_3}$.  

Second, if this is not the case, it must be that 
$P_1 \neq  P_2\in \mathcal P(Q)$ are \textit{adjacent} in the first coordinate.  
Then, pick $P_3$ and $P_4$ so that $(P_1,P_3)$, $(P_2,P_4)$ and $(P_3,P_4)$ are admissible. 
It follows that $g _{P_1,P_2}$ is in the linear span of the three   $g _{P_1,P_3}, g _{P_2,P_4}$ 
and $g _{P_3, P_4}$.    The proof is complete.
\end{proof}

\smallskip

\begin{proof}[Proof of Proposition \ref{prop:Besov-2}]
For a fixed $b \in \mathbf B_{\nu}^p(\mathbb R^n)$, 
\begin{align}\label{eqn:aux-T}
    \lVert b \rVert _{\mathbf B_{\nu}^p(\mathbb R^n)} ^{p} \lesssim   \sum_{\omega\in \left\{0,\frac{1}{3},\frac{2}{3}\right\}^n} \|b\|_{\mathbf B_{\nu}^p(\mathbb{R}^n, \mathcal D)} ^{p} 
    \lesssim \sum_{\omega\in \left\{0,\frac{1}{3},\frac{2}{3}\right\}^n}\sum_{Q^\omega \in \mathcal D^\omega}  \sum_{l=1} ^{N}  
   { \left(\frac{|Q^\omega|^2}{\mu(Q^\omega)\lambda^{-1}(Q^\omega)}\right)^{p\over 2}  } \left( \frac{|\langle b,g _{Q^\omega,l}\rangle|}{\sqrt{|Q^\omega|}}\right)^p.
\end{align}  

We will recognize the right hand side as an expression involving NWO functions. 
Now, we inspect a single inner product in the sum above.  
Recall the definition of $g _{Q^\omega,l}$ in \eqref{e:gPP} and the property (iii) in Proposition \ref{p:BesovRiesz}, we have
\begin{align*}
 \left\vert \langle   b,g _{Q^\omega,l}\rangle \right\vert 
    & \lesssim \Bigg| \frac{1}{\sqrt{ \lvert Q^\omega \rvert} }  \int _{Q^\omega(l,1)} \int _{Q^\omega (l,2)}  \big(b(x)-b(y)\big ) K_j (x-y) \;dx\,dy \Bigg|. 
\end{align*}
On the right-hand side, we insert $\lambda ^{\frac{1} {2}} (x) \lambda ^{- \frac{1}{2}} (x) $ 
and $\mu ^{\frac{1}{2}}(y) \mu ^{-\frac{1}{2}}(y)$. 
That with some manipulations leads to 
\begin{align}
    &   { \left(\frac{|Q^\omega|^2}{\mu(Q^\omega)\lambda^{-1}(Q^\omega)}\right)^{p\over 2}  } \left( \frac{|\langle b,g _{Q^\omega,s}\rangle|}{\sqrt{|Q^\omega|}}\right)^p\nonumber\\
     \lesssim\,  & { \left(\frac{|Q^\omega|^2}{\mu(Q^\omega)\lambda^{-1}(Q^\omega)}\right)^{p\over 2}  } \frac{1}{|Q^\omega|^p}
         \nonumber \\
         & \quad \cdot \Bigg|  \int\!\!\int  \big (b(x)-b(y)\big )\lambda^{\frac{1}{2}}(x) K_j(x-y) \mu^{-\frac{1}{2}}(y)
     \mu^{\frac{1}{2}}(y)\mathsf{1}_{Q^\omega (l,2)} (y) \cdot 
     {\lambda^{-\frac{1}{2}}(x)\mathsf{1}_{Q^\omega (l,1)}(x)}  \; dy\,dx \, \Bigg|^p \nonumber\\  \label{e:ttt}
      =\,&   
     \Biggl|  \int\!\!\int  \big (b(x)-b(y)\big ) \left[\lambda^{\frac{1}{2}}(x) K_j(x-y) \mu^{-\frac{1}{2}}(y) \right ] 
     \frac{\mu^{\frac{1}{2}}(y)\mathsf{1}_{Q^\omega (l,2)} (y)   } { \mu (Q^\omega) ^{\frac{1}{2}} }
     \frac {\lambda^{-\frac{1}{2}}(x)\mathsf{1}_{Q^\omega (l,1)}(x)} { \lambda ^{-1} (Q^\omega) ^{\frac{1}{2}}}  
     \; dy\,dx \, \Biggr|^p. 
 \end{align}

Denote 
$$
    G_{Q^\omega,s}(y)= \frac{\mu^{\frac{1}{2}}(y)\mathsf{1}_{Q^\omega (l,2)} (y)   } { \mu (Q^\omega) ^{\frac{1}{2}} },
    \quad 
    H_{Q^\omega,j}(x)=\frac {\lambda^{-\frac{1}{2}}(x)\mathsf{1}_{Q^\omega (l,1)}(x)} { \lambda ^{-1} (Q^\omega) ^{\frac{1}{2}}}  .
$$
Then, we have 
$$
    \eqref{e:ttt} = \bigl\lvert \bigl \langle \lambda^{\frac{1}{2}} [b, R_j] 
    \mu^{-\frac{1}{2}}  G_{Q^\omega,l}, \,H_{Q^\omega,l}\bigr \rangle \bigr\rvert  ^p. 
$$
Moreover, both  $\{G_{Q^\omega,l}\}_{Q^\omega \in \mathcal D^{\omega}}$ and $\{H_{Q^\omega,l}\}_{Q^\omega\in \mathcal D^{\omega}}$  are NWOs.   Indeed, by Lemma \ref{reverse}, there exists a constant $\sigma_\mu>0$, such that  $\mu$ satisfies the $(1+\sigma_\mu)-$reverse H\"older inequality. Let $r=2(1+\sigma_\mu)>2$, 
\begin{align*}
    \|G_{Q^\omega,l}\|_{L^r(\mathbb R^n)} &\lesssim  \frac{1}{\mu(Q^\omega) ^{1\over 2} } \left(\int_{Q^\omega (l,2)} \mu^{r\over 2} (y)dy\right)^{1\over r} \leq \frac{|Q^\omega |^{1\over r}}{\mu(Q^\omega) ^{1\over 2} } \left(\frac{1}{|Q^\omega|}\int_{  Q^\omega} \mu^{1+\sigma_\mu} (y)dy\right)^{ {1\over {1+\sigma_\mu}} \cdot {1\over 2}}\\
    & \lesssim \frac{|Q^\omega |^{1\over r}}{\mu(Q^\omega) ^{1\over 2} }  \left(\frac{1}{|Q^\omega|}\int_{Q^\omega}  \mu(y) dy\right)^{1\over 2}
   \lesssim |Q|^{ {1\over r}- {1\over 2} }, 
\end{align*}
and the argument for $\{H_{Q^\omega,l}\}_{Q^\omega\in \mathcal D^{\omega}}$ is similar.

Hence, we have a replacement for the Schmidt decomposition for $[b,R_j]$ and so by Lemma \ref{eq-NWO} we have 
\begin{align} \label{e:needagain}
    \lVert b \rVert _{\mathbf B_{\nu}^p(\mathbb R^n,\mathcal D^\omega)} ^{p}  
    \lesssim  \sum_{Q^\omega \in \mathcal D^\omega}  \sum_{s=1} ^{N}     \bigl\lvert \bigl \langle 
    \lambda^{\frac{1}{2}} [b, R_j] \mu^{-\frac{1}{2}}  G_{Q^\omega,s}, \,H_{Q^\omega,s}
    \bigr \rangle \bigr\rvert ^p
    \lesssim \big\|\lambda^{\frac{1}{2}} [b, R_j] 
    \mu^{-\frac{1}{2}}\big\|_{S^p(L^2(\mathbb R^n), L^2(\mathbb R^n))}^p. 
\end{align}
Combining this and \eqref{eqn:aux-T}, the proof is finished .  
\end{proof}


\section{Critical index $p=n$: proof of Theorem \ref{thm:main2}}\label{sec:thm2}
\setcounter{equation}{0}

Let all the notation be the same as in Theorem \ref{thm:main2}.  
We will show that if $b$ is non-constant, then the  Schatten $S^n$ norm of the weighted commutator is infinite.  

We begin by considering non-constant $b\in C^\infty(\mathbb R^n)$. 
Thus,  $\nabla b$ is non-zero at some $x_0$. 
And, we can further assume that $x_0$ is a Lebesgue point of  $\nu ^{-1} $, 
 $\nu ^{-1} (x_0) \not= 0$, and does not fall on the boundary of any of our shifted dyadic cubes.  
From these considerations, this follows.  
There exists $N\in \mathbb {Z}_+$, so that for any dyadic grid $\mathcal D $, 
and letting $Q ^{x_0 }_k $ be the cube in $\mathcal D _{k} $ which contains $x_0$, we have 
\begin{equation}\label{eqn:Lebesgue}
    \frac{\nu^{-1}(Q_k^{x_0})}{|Q_k^{x_0}|}\approx \nu^{-1}(x_0) >0, \qquad  k > N.
\end{equation}

{      Note that the argument of the inequality \eqref{e:needagain} ensures that it also holds for $p=n$}, that is,
\begin{equation}
    \lVert b \rVert _{\mathbf B_{\nu}^n(\mathbb R^n,\mathcal D)} ^{n}  
    \lesssim \big\|\lambda^{\frac{1}{2}} [b, R_j] \mu^{-\frac{1}{2}}\big\|_{S^n(L^2(\mathbb R^n), L^2(\mathbb R^n))}^n.    
\end{equation}
{For the same dyadic grid $\mathcal D$, we apply the following equivalence again
$$
      \|b\|_{\mathbf B_{\nu}^n(\mathbb{R}^n, \mathcal D)} ^{n}  
    \simeq \sum _{k\in \mathbb{Z}   } 
    \sum_{Q\in \mathcal{D}_k }  \left(\frac{1}{ \nu(Q) }  
    \int_Q{| E_{k+2} b - E _{k} b  | }\; dx\right)^{n} ,
$$
and we claim that the sum on the right-hand side is infinite.  Fix a cube $Q_0$ containing $x_0$ with $\inf _{x\in Q_0} \lvert \nabla b (x) \rvert > c$.  Without loss of generality, assume that $ Q_0\in \mathcal D_{k_0}$ for some $k_0$ sufficiently large. Then, for any cube $Q\subset Q_0$ with $Q\in \mathcal D_k$, 
we have  $| E_{k+2} b - E _{k} b  | > c' \ell (Q) $ on at least one grandchild $Q'\in\mathcal D_{k+2} $ of $Q$, where $c'=c'(c,n)>0$ independent of $Q$, provided by the mean value theorem of integrals and the fact that there exists $\theta\in (0, {\pi \over 2})$ such that the set 
\begin{align*}
	 \big \{Q' \text{ is one grandchild of } Q: & \,   \text{dist}(x_Q, Q')\geq \ell(Q)/4 \  \   \text{and}  \\
    &     -\theta\leq  \angle( \nabla b(x), x_{Q'}-x_Q) \leq \theta \text{ for  any } x\text{ in the segment }  L_{x_Q x_{Q'}} \big \}\neq \emptyset,
\end{align*}  
where $x_P$ is one point at $P$ satisfying $b(x_P)=\langle b\rangle _P$. Roughly speaking, for given $b\in C^\infty(\mathbb R^n)$, the vector $\nabla  b(x)\approx \nabla b(x_Q)$  since $\ell(Q)$ is small enough, and the above fact can be verified by checking the angles between any given vector in $\mathbb R^n$ (without loss of generality, assume it as $(1,0,\ldots, 0)$) and the vector composed of the vertex $v_s$ chosen from a grandchild $Q_s$ of $Q$, $s=1,2$. 
}

  Hence, we have by application of the $A_2$ conditions and \eqref{eqn:Lebesgue} to see
\begin{align}\label{eqn:critical1}
    \lVert b \rVert _{\mathbf B_{\nu}^n(\mathbb R^n,\mathcal D)} ^{n}   
    &\gtrsim \sum_{j=1}^\infty \sum_{\substack{Q\in \mathcal D, \, Q\subset Q_0 \\  \ell(Q)=2^{-j}\ell(Q_0)}}
\left( \ell(Q) \frac{\nu^{-1}(Q)}{|Q|} \right)^{n}     \gtrsim   \nu^{-1}(x_0)^n  \sum_{j=1}^\infty |Q_0|=+\infty.
\end{align}

In general, consider  $b\in {\rm VMO}_{\nu}(\mathbb R^n)$. 
Following the proof of Proposition~\ref{prop:Besov-2}, we have
\begin{align*}
    \mathcal T(b)&:=\left(\sum_{Q\in \mathcal D} \left(\frac{1}{\nu(Q)}\int_Q \frac{1}{|Q|}\int_Q |b(x)-b(y)| dy dx\right)^n\right)^{1\over n}
   \\
    &\lesssim \|  [b, R_j] \|_{S^n \big(L_{\mu}^2(\mathbb R^n), L_{\lambda}^2(\mathbb R^n) \big)}.  
\end{align*}

Let  $\psi$ be a nonnegative function in  $C_0^\infty(\mathbb R^n)$ with  $\int \psi(x)dx=1$. 
Denote  $\psi_\epsilon(x)=\epsilon^{-n}\psi(x/\epsilon)$ and let $b_\epsilon(x)=\psi_\epsilon * b(x)$, 
then $b_\epsilon\in C^\infty(\mathbb R^n)$.  
Note that for $\epsilon$ sufficiently small, 
$$
    \mathcal T(b_\epsilon)\lesssim\sup_{h\in B(0,1)} \mathcal T(\tau_h b),
$$
where $\tau_h b(x)=b(x-h)$. For any fixed $h\in B(0,1)$ and for $f\in L_{\mu}^2(\mathbb R^n)$,
$$
    [\tau_h b, R_j]f(x)=\int_{\mathbb R^n} \left(b(x-h)-b(y-h)\right) K_j(x-y)f(y)dy,
$$
then the translation invariance of the kernel $K_j$ of $R_j$ yields
$$
     \left\| [\tau_h b, R_j]f\right\|_{L_{\lambda}^2(\mathbb R^n) )}^2  
     =\int_{\mathbb R^n} \left| \int_{\mathbb R^n} \left(b(x)-b(y)\right)  
     K_j(x-y) \tau_{-h}f(y)dy\right|^2  \tau_{-h}\lambda(x) dx.
$$
Observe that 
$$
     \int |\tau_{-h}f(y)|^2 \tau_{-h}\mu(y)dy=\int |f(y)|^2 \mu(y)dy,
$$
$[\tau_{-h} \mu]_{A_2}=[\mu]_{A_2}$ and  
$[\tau_{-h} \lambda]_{A_2}=[\lambda]_{A_2}$ by definition. 
Hence
$$
    \mathcal T(\tau_h b)\lesssim  \|  [b, R_j] \| 
    _{S^n \big(L_{\tau_{-h} \mu}^2(\mathbb R^n), L_{\tau_{-h}\lambda}^2(\mathbb R^n) \big)}
    =\|  [b, R_j] \|_{S^n \big(L_{\mu}^2(\mathbb R^n), L_{\lambda}^2(\mathbb R^n)\big )}.
$$
This yields that for $\epsilon$ sufficiently small, 
$$
    \mathcal T(b_\epsilon)\lesssim
    \|  [b, R_j] \|_{S^n \big(L_{\mu}^2(\mathbb R^n), L_{\lambda}^2(\mathbb R^n)\big )}.
$$
Nevertheless, a similar argument yields that if the smooth function $b_\epsilon$ is non-constant, 
then $\mathcal T(b_\epsilon)=\infty$. Hence $b_\epsilon$ is constant almost everywhere 
when $[b,R_j]$ belongs to $S^n \big (L_{\mu}^2(\mathbb R^n), L_{\lambda}^2(\mathbb R^n)\big )$.  
Since $b_\epsilon\to b$ as $\epsilon\to 0+$, we obtain $b$ is  constant as well.

\begin{remark}
We note that $p=n$ is critical. 
In fact, for $p>n$, we see that there are many non-constant functions in $\mathbf B_{\nu}^p(\mathbb R^n)$. 
For example, we take $\mu(x)=|x|^\alpha$ and $\lambda(x)=|x|^\beta$, 
where $\alpha,\beta\in(-n,n)$. 
Then both $\mu,\lambda\in A_2$. 
Next, we choose $b(x)\in C_0^\infty(\mathbb R^n)$, with $b(x)\geq0$,   
{\rm supp}\,$b\subset B(x_0,1)$, where $x_0\in\mathbb R^n$ with $|x_0|>10$, 
and $|\nabla b(x)|>0$ for $x\in B(x_0,1)$. 
Then we see that all such functions $b\in \mathbf B_{\nu}^p(\mathbb R^n)$ with $\nu=(\mu/\lambda  ) ^{1/2}$.
\end{remark}


\section{Oscillation spaces associated with two weights: proof of Theorem \ref{thm:Weight-Sobolev-revised}}\label{sec:6}
\setcounter{equation}{0}

We begin with three auxiliary lemmas, to study the structure of the 
weighted oscillation sequence space $\mathcal W_{\nu}(\mathbb R^n)$. 

\begin{lemma}\label{lem:osc-sub-BMO}
Suppose $\nu\in A_2$.
Then $\mathcal W_{\nu}(\mathbb R^n)\subset {\rm BMO}_\nu(\mathbb R^n)$.
\end{lemma}

\begin{proof}
For  each $\omega\in \left\{0,\frac{1}{3},\frac{2}{3}\right\}^n$ and $Q^\omega\in \mathcal D^\omega$, 
it follows from  Definition~\ref{def:adjacent-dyadic} that there exists $Q\in \mathcal D$ satisfying 
$|Q|=|Q^\omega|$ and $Q^\omega\subset 20\sqrt{n}Q$. 
Using the doubling property of $A_2$ weights,  there exists a positive constant $C=C(n,[\nu]_{A_2})$ such that 
$$
    {1\over \nu(Q^\omega)}\int_{Q^\omega} |b(x)-\langle b\rangle _{Q^\omega}| \,dx \leq C\, R_Q,
$$
where 
\begin{align}\label{R_Q}
    R_Q:=\frac{1}{\nu(20\sqrt{n}Q)}\int_{20\sqrt{n}Q} \left|b(x)-\langle b\rangle _{20\sqrt{n}Q}\right|dx.
\end{align}
Hence, combined with Definition \ref{def:Weight-Sobolev}, 
we see  that  for each function $b\in \mathcal W_{\nu}(\mathbb R^n)$, 
we have $ b\in \bigcap_{\omega\in \left\{0,\frac{1}{3},\frac{2}{3}\right\}^n } 
{\rm BMO}_\nu(\mathbb R^n,\mathcal D^\omega)= {\rm BMO}_\nu(\mathbb R^n)$.
\end{proof}

In fact, we further have the following argument.

\begin{lemma}
Suppose $\nu\in A_2$.
Then $\mathcal W_{\nu}(\mathbb R^n)\subset {\rm VMO}_\nu(\mathbb R^n)$.
\end{lemma}

\begin{proof}
Based on Lemma~\ref{lem:osc-sub-BMO}, for each $b\in \mathcal W_{\nu}(\mathbb R^n)$,  
it suffices to verify that $b$ satisfies the three limiting conditions 
concerning the weighted mean oscillation over balls, as listed in Definition~\ref{vmo}.

Recall that the adjacent dyadic systems possess the property that 
there exists a positive constant $C$ such that for any ball $B\subset \mathbb R^n$, 
there exists  $Q^\omega\in \mathcal{D}^{\omega}$ for some 
$\omega\in \left\{0,\frac{1}{3},\frac{2}{3}\right\}^{n}$, 
satisfying $ B\subset Q^\omega $ and $ |Q^\omega|\leq C |B|$. 
Combining the argument in Lemma~\ref{lem:osc-sub-BMO}, we have that for any ball $B$, 
there exists $Q\in \mathcal D$ such that
the weighted mean oscillation of $b$ over $B$ is governed by $R_Q$, 
which was defined as in \eqref{R_Q}. Therefore, it remains to prove that  
$$
  	\lim_{a\to0}\sup_{Q:\  \ell(Q)\leq a} R_Q=\lim_{a\to\infty}\sup_{Q:\  \ell(Q)\geq a} R_Q
	=\lim_{a\to\infty}\sup_{Q:\  Q\cap B(0,a)=\emptyset} R_Q=   0.
$$

We verify the first condition. 
That is, it  will be shown that for any $\epsilon>0$, there exists $K_\epsilon\in \mathbb N$, such that
$$
    \sup\,\left\{R_Q:\, Q\in \mathcal D,\, \ell(Q)\leq 2^{-K_\epsilon}\right\}< \epsilon. 
$$
Otherwise, there exists a sequence $\{Q_j\}_{j\in \mathbb N}$ in $\mathcal D$ 
such that $R_{Q_j}\geq \epsilon_0$ uniformly for some $\epsilon_0>0$, and $\ell(Q_j)\to 0$ as $j\to \infty$. 
Thus $\|\{R_Q\}\|_{\ell^{n,\infty}}=\sup_k \{k^{1/n}a_k^*\}=+\infty$, 
where $\{a_k^*\}$ is the non-increasing rearrangement of $\{R_Q\}_{Q\in \mathcal D}$.  
This contradicts $b\in \mathcal W_{\nu}(\mathbb R^n)$. 
Hence the first limiting condition holds. The proofs of the remaining two conditions are similar.
\end{proof}

Next we show that  $\mathcal W_{\nu}(\mathbb R^n)$ has the following properties which are crucial for Theorem \ref{thm:Weight-Sobolev-revised}.

\begin{lemma}\label{lem aux Sobolev new}
Suppose  $\lambda,\, \mu\in A_2$ and set $\nu=\mu^{1\over 2}\lambda^{-{1\over 2}}$. 
For $b\in L^1_{loc}(\mathbb R^n)$, assume that the  sequence 
$$
	  \{R_Q\}_{Q\in\mathcal D}=\left\{\frac{1}{\nu(20\sqrt{n}Q)}\int_{20\sqrt{n}Q} 
	  \left|b(x)-\langle b\rangle _{20\sqrt{n}Q}\right|dx\right\}_{Q\in\mathcal D}\in\ell^{n,\infty},
$$ 
then:
\begin{itemize}
\item [(1)] There is a smooth function $F(x,t)$ on $\mathbb R_+^{n+1}$ 
with $\lim\limits_{t\to 0} F(x,t)=b(x)$, a.e., and 
$$
    F^*(x,t):=  \sup\Big \{s|\nabla F(y,s)|:\, |y-x|<t,\  \      
    \frac{t}{2}<s<t\Big \}\cdot \int_{t/2}^t \int_{B(x,t)} 1\, \frac{dyds}{\nu(B(y,s))s}
$$
is in $ L^{n,\infty}\left (\mathbb R_+^{n+1},\,  \frac{dxdt}{t^{n+1}}\right )$, 
where $\nabla=(\nabla_x, \partial_t)$. Moreover, we have
$$ 
    \|F^*\|_{L^{n,\infty} (\mathbb R_+^{n+1},\,  \frac{dxdt}{t^{n+1}} )} 
     \lesssim \left\| \left\{ R_Q\right\}_{Q\in\mathcal D}\right\|_{\ell^{n,\infty}}.
$$

\smallskip

\item[(2)] The sequences  
$\displaystyle\bigg\{\bigg[{1\over \mu(20\sqrt{n}Q)} \int_{20\sqrt{n}Q} 
    \left|b(x)- \langle b\rangle_{20\sqrt{n}Q}\right|^{2} \lambda(x) dx\bigg]^{1\over 2}
    \bigg\}_{Q\in\mathcal D}  $
and	  
$$
    \bigg\{ \bigg[\frac{1}{\lambda^{-1}(20\sqrt{n}Q)} 
    \int_{20\sqrt{n}Q} \left|b-\langle b\rangle_{20\sqrt{n}Q}\right|^{2} 
    \mu^{-1} (x) \, dx \bigg]^{1\over 2} \bigg\}_{Q\in\mathcal D}
$$
are both in $\ell^{n,\infty}$ with norms dominated by $\left\| \left\{ R_Q\right\}_{Q\in\mathcal D}\right\|_{\ell^{n,\infty}}$.
\end{itemize}
\end{lemma}

\begin{proof}
We will begin by proving (1), and then apply it to show  (2).

Let $\varphi$ be a smooth function on $\mathbb R^n$, which is supported in $\big[-1, 1\big ]^n$ 
and satisfies $\int_{\mathbb R^n} \varphi(x) \, dx=1$.

Let $\varphi_t(x)=t^{-n}\varphi(\frac{x}{t})$ and define $F(x,t)=\varphi_t* b(x)$ for $t>0$. 
For each $(x,t)\in Q\times \big[\frac{\ell(Q)}{2},\ell(Q)\big ]$, where $Q\in \mathcal D$, 
let $T_{x,t}:=\{(y,s)\in \mathbb R_+^{n+1}:\, |y-x|<t,\  \      \frac{t}{2}<s<t \}$, we have 
\begin{align*}
	   \sup_{(y,s)\in T_{x,t}}  s|\nabla F(y,s) | 
	   &=\sup_{(y,s)\in T_{x,t}}  s\left|\int_{y-z\in [-\frac{s}{2},\frac{s}{2} ]^n } \left (b(z)
	   -\langle b\rangle _{20\sqrt{n}}\right ) \nabla \varphi_s(y-z)\, dz    \right|\\
	   &\leq c\,    \frac{1}{\ell(Q)^n} \int_{20\sqrt{n}Q}   \left|b(z)-\langle b\rangle _{20\sqrt{n}}\right| dz,
\end{align*}
where the positive constant $c$ is independent of $Q$, $(x,t)$ and $(y,s)$. 
As a consequence, there exists a positive constant $C=C(c, \nu, n)$ such that 
for each $Q\in \mathcal D$ and  $(x,t)\in Q\times \big[\frac{\ell(Q)}{2},\ell(Q)\big ]$,
$$
    F^*(x,t)\leq C\cdot  R_Q.
$$
  
Note that 
$$
     \mathbb R_+^{n+1}=\bigcup_{Q\in \mathcal D} Q\times \Big [\frac{\ell(Q)}{2},\ell(Q)\Big ]
$$
up to a zero measure set, and the interiors of 
$\left \{Q\times \big[\frac{\ell(Q)}{2},\ell(Q)\big] \right\}_{Q\in \mathcal D}$ 
are mutually disjoint. Therefore,
\begin{align*}
	\|F^*\|_{L^{n,\infty}\left (\mathbb R_+^{n+1},\,  \frac{dxdt}{t^{n+1}}\right )} 
	&=\sup_{\lambda >0} \,\lambda \left(\iint_{\{(x,t):\, F^*(x,t)>\lambda \}}   1\, \frac{dxdt}{t^{n+1}}\right)^{1\over n}\\
	&=\sup_{\lambda >0} \,\lambda \left(\sum_{Q\in \mathcal D} 
	\int_{\ell(Q)/2}^{\ell(Q)}\int_Q  \mathsf{1}_{\{ F^*>\lambda \}} \, \frac{dxdt}{t^{n+1}}\right)^{1\over  n}\\
	&\lesssim \sup_{\lambda >0}\lambda \left(  \text{the number of }  \Big \{Q\in\mathcal D:\,   
	\sup_{(x,t)\in Q\times [\frac{\ell(Q)}{2},\ell(Q)]}   F^*(x,t)  >\lambda \Big \}  \right)^{1\over n}\\
	&\leq \sup_{\lambda >0}\lambda \left(  \text{the number of }  \Big \{Q\in\mathcal D:\,   
	R_Q >C^{-1}\lambda \Big \}  \right)^{1\over n}\\
	&\lesssim  \left \|\left\{ R_Q\right\}_{Q\in\mathcal D}\right \|_{ \ell^{n,\infty}},
\end{align*}
where the last inequality follows from the fact that 
$$
    \left \|\left\{ R_Q\right\}_{Q\in\mathcal D}\right \|_{ \ell^{n,\infty}}
    =\sup_{k\geq 1} k^{1/n} R_k^*\geq \sup_{\lambda >0}\lambda 
    \left(  \text{the number of }  \Big \{Q\in\mathcal D:\,   R_Q \geq \lambda \Big \}  \right)^{1/n}
$$
with $\{R_k^*\}_{k\geq 1}$  the non-increasing rearrangement of $\{R_Q\}_{Q\in \mathcal D}$.
That is, $F^* \in  L^{n,\infty}\left (\mathbb R_+^{n+1},\,  \frac{dxdt}{t^{n+1}}\right )$, 
and we complete the proof of (1).

It remains to show (2).  Observe that 
\begin{align*}
	&\bigg[{1\over \mu(20\sqrt{n}Q)} \int_{20\sqrt{n}Q} 
	\left|b(x)- \langle b\rangle_{20\sqrt{n}Q}\right|^{2} \lambda(x) dx\bigg]^{1\over 2}\\
	\leq\, &   \bigg[{1\over \mu(20\sqrt{n}Q)} \int_{20\sqrt{n}Q} 
	\left|b(x)-  F(c_Q, \ell(Q))\right|^{2} \lambda(x) dx\bigg]^{1\over 2}  
	+ \left |\langle b\rangle_{20\sqrt{n}Q} - F(c_Q,\ell(Q))\right | 
	\left( \frac{\lambda(20\sqrt{n}Q)}{\mu(20\sqrt{n}Q)}\right)^{1\over 2}\\
	\lesssim\, &   \bigg[{1\over \mu(20\sqrt{n}Q)} \int_{20\sqrt{n}Q} 
	\left|b(x)-  F(c_Q, \ell(Q))\right|^{2} \lambda(x) dx\bigg]^{1\over 2}  \\
	&   +\left( \frac{1}{|20\sqrt{n}Q|}  \int_{20\sqrt{n}Q} 
	\left| b(x)-F(c_Q,\ell(Q))\right| \,dx \right) \frac{|20\sqrt{n}Q|}{\nu(20\sqrt{n}Q)}\\
	=:\,&   \rm{I}_Q+ \rm{II}_Q.
\end{align*}

For the term $\rm{I}_Q$, 
writing $b(x)-F(c_Q,\ell(Q))= F(x,\ell(Q))-F(c_Q,\ell(Q))  + b(x)-F(x,\ell(Q))$ to see
\begin{align*}
	{\rm{I}_Q} \lesssim & \sup_{x\in 20\sqrt{n}Q} \left| F(x,\ell(Q))-F(c_Q,\ell(Q))\right| 
	\frac{|20\sqrt{n}Q|}{\nu(20\sqrt{n}Q)} \\
	&\quad +\bigg[{1\over \mu(20\sqrt{n}Q)} \int_{20\sqrt{n}Q} 
	\left(\int_0^{\ell(Q)} \big|\partial_t F(x,t)\big| dt\right)^{2} \lambda(x) dx\bigg]^{1\over 2}.
\end{align*}

Similarly,
\begin{align*}
	{\rm{II}_Q} \lesssim & \sup_{x\in 20\sqrt{n}Q} \left| F(x,\ell(Q))-F(c_Q,\ell(Q))\right| 
	\frac{|20\sqrt{n}Q|}{\nu(20\sqrt{n}Q)} +   {1\over \nu(20\sqrt{n}Q)} \int_{20\sqrt{n}Q} 
	\int_0^{\ell(Q)} \big|\partial_t F(x,t)\big| \,dt  dx.
\end{align*}
Obviously, 
\begin{align*}
     &\sup_{x\in 20\sqrt{n}Q} \left| F(x,\ell(Q))-F(c_Q,\ell(Q))\right| \frac{|20\sqrt{n}Q|}{\nu(20\sqrt{n}Q)} \\
     &\lesssim\, \sup_{x\in 20\sqrt{n}Q} \ell(Q)  \left|\nabla F(x,\ell(Q))\right|   \frac{|20\sqrt{n}Q|}{\nu(20\sqrt{n}Q)}\\
     &\lesssim  \sum_{\substack{ P\in \mathcal D:\,   \ell(P)=\ell(Q) \\ P\cap 20\sqrt{n}Q\neq\emptyset }}   R_P.
\end{align*}
Note that the operator $\mathsf M$ defined on the sequence $\mathscr{R}=\{R_Q\}_{Q\in \mathcal D}$ by 
$$
    \mathsf M(\mathscr R)(Q) := 
    \sum_{\substack{ P\in \mathcal D:\,   \ell(P)=\ell(Q) \\ P\cap 20\sqrt{n}Q\neq\emptyset }}   R_P
$$ 
is a bounded map of $\ell^{n,\infty}$ to itself. 
Hence it suffices to consider 
$$
    {\rm{I}_Q'}:=\Bigg[{1\over \mu(20\sqrt{n}Q)} \int_{20\sqrt{n}Q} 
    \left(\int_0^{\ell(Q)} \big|\partial_t F(x,t)\big| dt\right)^{2} \lambda(x) dx\Bigg]^{1\over 2}
$$
and its analogy
$$
    {\rm{II}_Q'}:=   {1\over \nu(20\sqrt{n}Q)} 
    \int_{20\sqrt{n}Q} \int_0^{\ell(Q)} \big|\partial_t F(x,t)\big| \,dt  dx.
$$

Let $\widetilde F(y,t)=\sup_{   t/2<s<t}s\big|\nabla F(y,s)\big|$,
then $\displaystyle F^*(x,t)=  \sup_{|y-x|<t} \widetilde F(y,t)  \cdot  
\int_{t/2}^t \int_{B(x,t)} 1\, \frac{dyds}{\nu(B(y,s))s}$.
By the convexity of the power function,
\begin{align*}
	{\rm{I}_Q'} \leq\, & \bigg[{1\over \mu(20\sqrt{n}Q)} \int_{20\sqrt{n}Q} \bigg( 
	\sum_{j=-\infty}^{\log_2 \ell(Q)} \sup_{1/2\leq 2^{-j}t \leq 1} t \big|\nabla F(x,t)\big| \bigg)^{2} 
	\lambda(x) dx\bigg]^{1\over 2}\\
	\leq\, &  \bigg[{1\over \mu(20\sqrt{n}Q)} \int_{20\sqrt{n}Q} \bigg( 
	\sum_{j=-\infty}^{\log_2 \ell(Q)}  \widetilde F(x,2^j)
	\bigg)^{2} \lambda(x) dx\bigg]^{1\over 2}\\
	=\,&  \bigg[{1\over \mu(20\sqrt{n}Q)} \int_{20\sqrt{n}Q} \bigg( 
	\sum_{j=-\infty}^{\log_2 \ell(Q)}  \left[ (\log_2 \ell(Q))-j+1\right]^{-2} 
	\left[ (\log_2 \ell(Q))-j+1\right]^{2}\widetilde F(x,2^j)
	\bigg)^{2} \lambda(x) dx\bigg]^{1\over 2}\\
	\lesssim \,&  \bigg[{1\over \mu(20\sqrt{n}Q)} \int_{20\sqrt{n}Q} 
	\sum_{j=-\infty}^{\log_2 \ell(Q)} \left[ (\log_2 \ell(Q))-j+1\right]^{-2} 
	\bigg(  \left[ (\log_2 \ell(Q))-j+1\right]^{2}\widetilde F(x,2^j)
	\bigg)^{2} \lambda(x) dx\bigg]^{1\over 2}\\
	=\,& \bigg[{1\over \mu(20\sqrt{n}Q)} \int_{20\sqrt{n}Q} 
	\sum_{j=-\infty}^{\log_2 \ell(Q)}  \left[ (\log_2 \ell(Q))-j+1\right]^{2}  
	{\widetilde F(x,2^j) }^{2} \lambda(x) dx\bigg]^{1\over 2}\\
	\leq \,& \bigg[{1\over \mu(20\sqrt{n}Q)} 
	\sum_{\substack{P\in \mathcal D: \  \ell(P)\leq \ell(Q) \\ P\cap 20\sqrt{n}Q\neq \emptyset }}    
	\lambda(P) \left(\log_2 \frac{ \ell(Q)}{\ell(P)} +1\right)^{2}  
	\sup_{x\in P}{\widetilde F(x,\ell(P)) }^{2}  \bigg]^{1\over 2}.
\end{align*}
By observing that 
$$
    \sup_{x\in P}{\widetilde F(x,\ell(P)) }^{2} \leq \sup_{x\in P}  {F^*(x,\ell(P))} ^2 
    \cdot \left(\frac{\nu(P)}{|P|} \right)^2\lesssim \left(R_P\right)^2 \cdot 
    \frac{\lambda^{-1}(P)\mu(P)}{|P|^2} \approx \left(R_P\right)^2 \frac{\mu(P)}{\lambda(P)},
$$
we have 
\begin{align*}
	{\rm{I}_Q'} \lesssim\, & \Bigg[{1\over \mu(Q)} 
	\sum_{\substack{P\in \mathcal D: \  \ell(P)\leq \ell(Q) \\ P\cap 20\sqrt{n}Q\neq \emptyset }}    
	\mu(P) \left(\log_2 \frac{ \ell(Q)}{\ell(P)} +1\right)^{2}  \left(R_P\right)^2\Bigg]^{1\over 2}.
\end{align*}
We claim that the operator $\widetilde{\mathsf M}$  defined on the sequence $\mathscr{R}=\{R_Q\}_{Q\in \mathcal D}$ by 
$$
    \widetilde{\mathsf M}(\mathscr R)(Q):= \Bigg[{1\over \mu(Q)} 
    \sum_{\substack{P\in \mathcal D: \  \ell(P)\leq \ell(Q) \\ P\cap 20\sqrt{n}Q\neq \emptyset }}    
    \mu(P) \left(\log_2 \frac{ \ell(Q)}{\ell(P)} +1\right)^{2}  \left(R_P\right)^2\Bigg]^{1\over 2}  
$$ is a bounded map of $\ell^{n,\infty}$ to itself. 
By interpolation (see for example \cite[Section 5.3]{BL}), 
it suffices to show that $ \widetilde{\mathsf M}$ is bounded on $\ell^p$ for $p\geq1$.

To see this, note that for $1\leq p\leq 2$, by definition we have
\begin{align*}
    \left\| \left\{ \widetilde{\mathsf M}(\mathscr R)(Q)\right\}_{Q\in \mathcal D}\right\|_{\ell^p}
    &=\left( \sum_{Q\in\mathcal D}  \widetilde{\mathsf M}(\mathscr R)(Q)^p\right)^{1\over p}\\
    &= \Bigg( \sum_{Q\in\mathcal D} \bigg[ 
    \sum_{\substack{P\in \mathcal D: \  \ell(P)\leq \ell(Q) \\ P\cap 20\sqrt{n}Q\neq \emptyset }}  
     {\mu(P)\over \mu(Q)}  \bigg(\log_2 \frac{ \ell(Q)}{\ell(P)} +1\bigg)^{2}  
     \left(R_P\right)^2\bigg]^{p\over 2}\Bigg)^{1\over p}\\
     &\leq \Bigg( \sum_{Q\in\mathcal D}  \ 
     \sum_{\substack{P\in \mathcal D: \  \ell(P)\leq \ell(Q) \\ P\cap 20\sqrt{n}Q\neq \emptyset }}   
     \bigg({\mu(P)\over \mu(Q)} \bigg)^{p\over 2} \bigg(\log_2 \frac{ \ell(Q)}{\ell(P)} +1\bigg)^{p}  
     \left(R_P\right)^p\Bigg)^{1\over p}\\
     &\leq \Bigg( \sum_{P\in\mathcal D}  \ 
     \sum_{\substack{Q\in \mathcal D: \  \ell(P)\leq \ell(Q) \\ P\cap 20\sqrt{n}Q\neq \emptyset }}   
     \bigg({\mu(P)\over \mu(Q)} \bigg)^{p\over 2} \bigg(\log_2 \frac{ \ell(Q)}{\ell(P)} +1\bigg)^{p}  
     \left(R_P\right)^p\Bigg)^{1\over p}\\
     &\lesssim \Bigg( \sum_{P\in\mathcal D}  \  \left(R_P\right)^p\Bigg)^{1\over p}.
\end{align*}

Now we consider the case $ p>2$. 
The doubling property of $\mu\in A_2$ ensures that there exists $\theta=\theta(n,[\mu]_{A_2})\in (0,1)$, such that for any $Q\in \mathcal D$ and  $P\cap 20\sqrt{n}Q\neq\emptyset$ with $\ell(P)=2^{-k}\ell(Q)$,  we have $\mu(P)\leq \theta^k \mu(Q)$. Note that for any $p>2$, then $(\frac{p}{2})'>1$ and we can choose some $\alpha\in (0,1)$ such that $\alpha (\frac{p}{2})' >1$.  Hence 
\begin{align*}
    \sum_{\substack{P\in \mathcal D: \  \ell(P)\leq \ell(Q) \\ P\cap 20\sqrt{n}Q\neq \emptyset }}   \bigg({\mu(P)\over \mu(Q)} \bigg)^{\alpha  (\frac{p}{2})'} \lesssim \sum_{k=0}^\infty \sum_{\substack{P\in \mathcal D: \  \ell(P)=2^{-k} \ell(Q) \\ P\cap 20\sqrt{n}Q\neq \emptyset }}    \theta^{(\alpha  (\frac{p}{2})'-1)k }\frac{\mu(P)}{\mu(Q)}\lesssim  \sum_{k=0}^\infty  \theta^{(\alpha  (\frac{p}{2})'-1)k }\lesssim 1.
\end{align*} 
Based on this, we obtain that 
\begin{align*}
    & \left\| \left\{ \widetilde{\mathsf M}(\mathscr R)(Q)\right\}_{Q\in \mathcal D}\right\|_{\ell^p}\\
    \leq\, & \Bigg( \sum_{Q\in\mathcal D}   \Bigg[  \bigg(\sum_{\substack{P\in \mathcal D: \  \ell(P)\leq \ell(Q) \\ P\cap 20\sqrt{n}Q\neq \emptyset }}   
    \left({\mu(P)\over \mu(Q)} \right)^{\alpha  (\frac{p}{2})'}\bigg)^{\frac{p-2}{2}}\cdot
    \sum_{\substack{P\in \mathcal D: \  \ell(P)\leq \ell(Q) \\ P\cap 20\sqrt{n}Q\neq \emptyset }}   
    \left({\mu(P)\over \mu(Q)}\right)^{\frac{(1-\alpha)p}{2}}  \bigg(\log_2 \frac{ \ell(Q)}{\ell(P)} +1\bigg)^{p}  
    \left(R_P\right)^p \Bigg]\Bigg)^{1\over p}\\
    \lesssim\,  & \Bigg( \sum_{Q\in\mathcal D}  \  
    \sum_{\substack{P\in \mathcal D: \  \ell(P)\leq \ell(Q) \\ P\cap 20\sqrt{n}Q\neq \emptyset }}   
    \bigg({\mu(P)\over \mu(Q)} \bigg) ^{\frac{(1-\alpha)p}{2}}\bigg(\log_2 \frac{ \ell(Q)}{\ell(P)} +1\bigg)^{p}  
    \left(R_P\right)^p\Bigg)^{1\over p}\\
    =\,&\Bigg( \sum_{P\in\mathcal D}  \ 
    \sum_{\substack{Q\in \mathcal D: \  \ell(P)\leq \ell(Q) \\ P\cap 20\sqrt{n}Q\neq \emptyset }}   
    \bigg({\mu(P)\over \mu(Q)} \bigg)^{\frac{(1-\alpha)p}{2}} \bigg(\log_2 \frac{ \ell(Q)}{\ell(P)} +1\bigg)^{p}  
    \left(R_P\right)^p\Bigg)^{1\over p}\\
    \lesssim \,&\Bigg( \sum_{P\in\mathcal D}  \  \left(R_P\right)^p\Bigg)^{1\over p}.
\end{align*}

Similarly, 
\begin{align*}
    {\rm{II}_Q'}\lesssim   {1\over \nu(20\sqrt{n}Q)} 
    \sum_{\substack{P\in \mathcal D: \  \ell(P)\leq \ell(Q) \\ P\cap 20\sqrt{n}Q\neq \emptyset }}  
    |P|  \sup_{x\in P}{\widetilde F(x,\ell(P)) }
    \lesssim    {1\over \nu(Q)} 
    \sum_{\substack{P\in \mathcal D: \  \ell(P)\leq \ell(Q) \\ P\cap 20\sqrt{n}Q\neq \emptyset }}  \nu(P)  R_P,
\end{align*}
and  an analogous of the argument for the $\ell^p$-boundedness of $\widetilde{\mathsf M}$  can deduce that 
the operator $\mathsf M'$  defined on the sequence $\mathscr{R}=\{R_Q\}_{Q\in \mathcal D}$ by 
$$
    \mathsf M'(\mathscr R)(Q):=   {1\over \nu(Q)} 
    \sum_{\substack{P\in \mathcal D: \  \ell(P)\leq \ell(Q) \\ P\cap 20\sqrt{n}Q\neq \emptyset }}  \nu(P)  R_P
$$ 
is also a bounded map of $\ell^{n,\infty}$ to itself.

From the above, we obtain that the sequence  
$\displaystyle\left\{\bigg[{1\over \mu(20\sqrt{n}Q)} \int_{20\sqrt{n}Q} 
\left|b(x)- \langle b\rangle_{20\sqrt{n}Q}\right|^{2} \lambda(x) dx\bigg]^{1\over 2}\right\}_{Q\in\mathcal D}$
is in $\ell^{n,\infty}$ and its norm  is dominated by $\left\| \left\{ R_Q\right\}_{Q\in\mathcal D}\right\|_{\ell^{n,\infty}}$.

In the end,  observe that $\lambda,\,\lambda^{-1},\, \mu,\,\mu^{-1}\in A_2$, 
By definition,  for each $Q\in \mathcal D$,
$$
    \frac{|Q|}{\nu(Q)}
    \approx \frac{\nu^{-1} (Q)}{|Q|}
    \approx \frac{\lambda  (Q)^{1\over 2} 
    \mu^{-1}(Q)^{1\over 2}} {|Q|}
    \approx \frac{\mu^{-1}(Q)}{\lambda^{-1}(Q)} ,
$$
this allows that the argument of the sequence  
$$  \left\{ \bigg[\frac{1}{\lambda^{-1}(20\sqrt{n}Q)} \int_{20\sqrt{n}Q} 
\left|b-\langle b\rangle_{20\sqrt{n}Q}\right|^{2} 
\mu^{-1}(x)\, dx \bigg]^{1\over 2} \right\}_{Q\in\mathcal D}
$$
is similar, and we skip it. The proof is complete.
\end{proof}

\begin{proof}[Proof of Theorem \ref{thm:Weight-Sobolev-revised}]

We first note that by H\"older's inequality, we have 
\begin{align*}
    \bigg\|\left\{R_Q\right\}_{Q\in\mathcal D}\bigg\|_{ \ell^{n,\infty}} 
    \lesssim      
\bigg \|  \bigg\{\bigg[{1\over \mu(20\sqrt{n}Q)} \int_{20\sqrt{n}Q} 
    \left|b(x)- \langle b\rangle_{20\sqrt{n}Q}\right|^{2} \lambda(x) dx\bigg]^{1\over 2}
    \bigg\}_{Q\in\mathcal D}  \bigg \|_{ \ell^{n,\infty}}
\end{align*}
and
\begin{align*}
    \bigg\|\left\{R_Q\right\}_{Q\in\mathcal D}\bigg\|_{ \ell^{n,\infty}} 
    \lesssim      
     \bigg \| \bigg\{ \bigg[\frac{1}{\lambda^{-1}(20\sqrt{n}Q)} 
    \int_{20\sqrt{n}Q} \left|b-\langle b\rangle_{20\sqrt{n}Q}\right|^{2} 
    \mu^{-1} (x) \, dx \bigg]^{1\over 2} \bigg\}_{Q\in\mathcal D}\bigg \|_{ \ell^{n,\infty}}.
\end{align*}

Next, from Lemma \ref{lem aux Sobolev new}, we have the reverse inequalities of the above two inequalities. 
Thus, the proof of Theorem \ref{thm:Weight-Sobolev-revised} is complete.
\end{proof}

\begin{remark}\label{rem:equiv-Lorentz}
Combining the interpolation theory of Lorentz spaces  (see for example \cite[Section 5.3]{BL}), 
the argument in Lemma \ref{lem aux Sobolev new} indeed deduces the equivalence that 
\begin{align*}
    \left\|\left\{R_Q\right\}_{Q\in\mathcal D}\right\|_{\ell^{p,q}} 
    &\approx    
\bigg \|  \bigg\{\bigg[{1\over \mu(20\sqrt{n}Q)} \int_{20\sqrt{n}Q} 
    \left|b(x)- \langle b\rangle_{20\sqrt{n}Q}\right|^{2} \lambda(x) dx\bigg]^{1\over 2}
    \bigg\}_{Q\in\mathcal D}  \bigg \|_{\ell^{p,q}}\\
&\approx    
     \bigg \| \bigg\{ \bigg[\frac{1}{\lambda^{-1}(20\sqrt{n}Q)} 
    \int_{20\sqrt{n}Q} \left|b-\langle b\rangle_{20\sqrt{n}Q}\right|^{2} 
    \mu^{-1} (x) \, dx \bigg]^{1\over 2} \bigg\}_{Q\in\mathcal D}\bigg \|_{ \ell^{p,q}}
\end{align*}
 for any $0<p<\infty$ and $0<q\leq \infty$.
\end{remark}

\section{Weak Schatten class estimate at critical index $p=n$: proof of Theorem  \ref{thm:main3}}\label{sec:thm3} 
\setcounter{equation}{0}

Throughout this section, let $R_j$ be the  $j$th  Riesz transform on $\mathbb R^n$ with $n\geq 2$, $j=1,\ldots, n$. 
Suppose  $\lambda,\, \mu\in A_2$ and set $\nu=\mu^{1\over 2}\lambda^{-{1\over 2}}$.

\subsection{Proof of the sufficient condition}\label{ssec7.1}

In this subsection, we assume that $ b\in \mathcal W_{\nu}(\mathbb R^n) $, 
then prove that 
$$
    [b, R_{j}]\in {S^{n,\infty} \big(L_{\mu}^2(\mathbb R^n), L_{\lambda}^2(\mathbb R^n)\big)}.
$$ 
By noting that 
$$
    \|[b,R_j]\|_{S^{n,\infty}\big(L_{\mu}^2(\mathbb R^n), L_{\lambda}^2 (\mathbb R^n)\big)} 
    =\big \|\lambda^{\frac{1}{2}}[b, R_{j}] \mu^{-\frac{1}{2}}
    \big \|_{ S^{n,\infty}\left(L^{2}(\mathbb{R}^n),L^{2}(\mathbb{R}^n)\right)},
$$
it suffices to show that 
\begin{equation}\label{eqn:suff-1.6}
	\big \|\lambda^{\frac{1}{2}}[b, R_{j}] \mu^{-\frac{1}{2}}
\big \|_{ S^{n,\infty}\left(L^{2}(\mathbb{R}^n),L^{2}(\mathbb{R}^n)\right)}
\lesssim  \|b\|_{\mathcal W_{\nu}(\mathbb R^n)}.
\end{equation}

Let $\Lambda=\big \{(x,y)\in \mathbb{R}^{n}\times \mathbb{R}^{n}:x=y\big\}$, 
and $\Omega=\left(\mathbb{R}^{n}\times \mathbb{R}^{n}\right)\backslash \Lambda $. 
Assuming that $\mathscr{P}$ is a dyadic Whitney decomposition family of the open set $\Omega$, 
that is, $\bigcup_{P\in\mathscr{P}}P=\Omega$ and 
$\sqrt{2n}\,\ell(P)\leq {\rm distance}\, (P,\Lambda)\leq 4\sqrt{2n}\,\ell(P)$; 
see \cite[Appendix J.1]{Gra1} for instance.
Therefore, we write $K_{j}(x-y)=\Sigma_{P\in\mathscr{P}}K_{j}(x-y)\mathsf{1}_P(x,y)$, 
and $P$ can be the cubes $P_{1}\times P_{2}$, where $P_{1},P_{2}\in \mathcal{D}$, 
have the same side length and that distance between them must be comparable to this side-length, due to 
$$
    {\rm distance\,}(P,\Lambda)\approx {\rm distance\,}(c_P,\Lambda)
    =\min_{x\in \mathbb R^n} \sqrt{|x-c_{P_1}|^2+|x-c_{P_2}|^2} =\frac{|c_{P_1}-c_{P_2}|}{\sqrt{2}},
$$
where $c_P=(c_{P_1},c_{P_2})$ is the center of $P$. 
More precisely, we have  $3\sqrt{n}\,\ell(P_1)\leq |c_{P_1}-c_{P_2}|\leq 9\sqrt{n}\,\ell(P_1)$.

Thus, for each dyadic cube $P_1\in \mathcal{D}$, $P_{2}$ is related to $P_1$ 
and at most $M$ of the cubes $P_{2}$ such that $P_{1}\times P_{2}\in \mathscr{P}$, 
where $M\leq (20\sqrt{n})^n$ depends only on the dimension $n$.
Therefore, let $Q=P_{1}$, there is $R_{Q, s}$ such that $Q\times R_{Q, s}\in\mathscr{P}$, 
where $s=1,2,\ldots,M$, we can reorganize the sum
\begin{align*}
    K_{j}(x-y)=\sum_{P\in\mathscr{P}}K_{j}(x-y)\mathsf{1}_P(x,y)
    =\sum_{Q\in\mathcal{D}}\sum_{s=1}^{M}K_{j}(x-y)\mathsf{1}_{ Q\times R_{Q, s} }(x,y),
\end{align*}
where $|Q|=|R_{Q, s}|$ and $\mathrm{distance}(Q,R_{Q, s})\approx \ell(Q)$. 

Assume $\epsilon>0$ is sufficiently small such that 
$\left((1+\epsilon)Q\times (1+\epsilon) R_{Q,s}\right) \cap \Lambda=\emptyset$.
Take a smooth function $\eta$ on $\mathbb R^n$ with 
${\rm supp\,} \eta\subset [-(1+\epsilon)/2,(1+\epsilon)/2]^n$, satisfying $\eta\equiv 1$ on $[-1/2,1/2]^n$ 
and $|\partial^\alpha \eta|\leq C_\alpha $ for every multi-index $\alpha\in \mathbb Z_+^n$. 
Now define $\eta_Q(x)=\eta\left ((x-c_Q)/\ell(Q)\right )$, where $c_Q$ is the center of $Q$. 
It's clear that ${\rm supp\,} \eta_Q\subset (1+\epsilon)Q$ 
and $|\partial^\alpha \eta_Q|\leq C_\alpha \ell(Q)^{-|\alpha|}$.

For each $(x,y)\in Q\times R_{Q,s}$,
$$
    K_j(x-y)\mathsf{1}_{Q\times R_{Q, s}}(x,y)=K_j(x-y) \eta_Q(x)\eta_{R_{Q,s}}(y).
$$
Expanding in a multiple Fourier series on $(1+\epsilon)Q\times (1+\epsilon) R_{Q,s}$, we can write
$$
    K_j(x-y) \eta_Q(x)\eta_{R_{Q,s}}(y)
    =\sum_{\vec{l}\in\mathbb{Z}^{2n}}{c^{j}_{\vec{l},Q}}e^{2\pi i\vec{l}'\cdot 
    \widetilde{x}}e^{2\pi i\vec{l}''\cdot \widetilde{y}} \,\mathsf{1}
_{(1+\epsilon)Q}(x)\cdot \mathsf{1}_{(1+\epsilon)Q_{R,s}}(y),
$$
where $x=c_{Q}+(1+\epsilon)\ell(Q)\tilde{x}$, $y
=c_{R_{Q,s}}+(1+\epsilon)\ell(R_{Q,s})\tilde{y}$,  
$\vec{l}=(\vec{l}', \vec{l}'')\in \mathbb Z^n\times \mathbb Z^n$, 
and
$$
    {c^{j}_{\vec{l},Q}}
    =\int_{(1+\epsilon)R_{Q,s}}\int_{(1+\epsilon)Q}K_{j}(x-y)  \eta_{Q}(x)\eta_{R(Q,s)}(y) \, 
    e^{-2\pi i\vec{l}'\cdot \widetilde{x} }e^{-2\pi i\vec{l}'' \cdot \widetilde{y}}dxdy
    \frac{1}{|(1+\epsilon)Q|}\frac{1}{(1+\epsilon)|R_{Q,s}|}.
$$
For any  $\alpha,~\beta\in \mathbb{Z}_+^{n}$, 
note that $K_{j}(x-y)$ has the regularity condition 
$$
     \left|\partial^{\alpha}_{x}\partial^{\beta}_{y}K_{j}(x-y)\right|
     \leq C(\alpha,\beta)\frac{1}{|x-y|^{n+|\alpha|+|\beta|}},
$$
which, in combination with the fact that ${\rm distance}(Q,R_{Q,s})\approx \ell(Q)$, yields 
$$
     \left|\partial^{\alpha}_{x}\partial^{\beta}_{y}K_{j}(x-y)\eta_Q(x)\eta_{Q_{R,s}}(y)\right|
     \leq C(\alpha,\beta)\frac{1}{|x-y|^{n+|\alpha|+|\beta|}}.
$$
Hence, one may apply 
$\hat{f}(\vec{l})=\frac{1}{(2\pi i\vec{l})^{(\alpha,\beta)}}
\widehat{\big(\partial^{(\alpha,\beta)}f\big)}(\vec{l})$ to see
\begin{align*}
    \big  |{c^{j}_{\vec{l},Q}}\big|
    &\lesssim\frac{1}{|Q|}\frac{1}{(1+|\vec{l}|)^{|\alpha|+|\beta|}}.
\end{align*}
Let $\Upsilon^j_{\vec{l},Q}:=|Q|^{\frac{1}{2}}|R_{Q,s}|^{\frac{1}{2}}{c^{j}_{\vec{l},Q}}$, then 
$$
    \big|\Upsilon^j_{\vec{l},Q}\big|\lesssim \frac{1}{(1+|\vec{l}|)^{|\alpha|+|\beta|}},
$$
and so
\begin{eqnarray*}
    K_j(x-y) \mathsf{1} _{Q\times R_{Q, s}}(x,y)
    =\sum_{\vec{l}\in\mathbb{Z}^{2n}}\Upsilon^j_{\vec{l},Q}\frac{1}{|Q|^{1/2}}F_{\vec{l}',Q}(x)
    \frac{1}{|R_{Q,s}|^{1\over 2}}G_{\vec{l}'',R_{Q,s}}(y),
\end{eqnarray*}
where $F_{\vec{l}',Q}(x)=e^{2\pi i\vec{l}'\cdot \widetilde{x}}\, \mathsf{1}_{Q}(x)$ 
and $G_{\vec{l}'',R_{Q,s}}(y)=e^{2\pi i\vec{l}''\cdot \widetilde{y}}\, \mathsf{1}_{R_{Q,s}}(y)$. 
Then, we get
\begin{align*}
    K_j(x-y) &=\sum_{Q\in\mathcal{D}}\sum_{s=1}^{M}\sum_{\vec{l}\in\mathbb{Z}^{2n}}
    \Upsilon^j_{\vec{l},Q}\frac{1}{|Q|^{1\over 2}}F_{\vec{l}',Q}(x)\frac{1}{|R_{Q,s}|^{1\over 2}}G_{\vec{l}'',R_{Q,s}}(y).
\end{align*}

Thus, the kernel of $\lambda^{\frac{1}{2}}[b, R_{j}] \mu^{-\frac{1}{2}}$ can be represented as
\begin{align*}
    K^{\mu,\lambda}_{b}(x,y)
    =\sum_{Q\in\mathcal{D}}\sum_{s=1}^{M}\sum_{\vec{l}\in\mathbb{Z}^{2n}}
    (b(x)-b(y))\Upsilon^j_{\vec{l},Q} \frac{1}{|Q|^{1\over 2}}
    \lambda^{\frac{1}{2}}(x)F_{\vec{l}',Q}(x)\frac{1}{|R_{Q,s}|^{1\over 2}}
    G_{\vec{l}'',R_{Q,s}}(y)\mu^{-\frac{1}{2}}(y).
\end{align*}

Note that for each $Q$,  $20\sqrt{n}Q$ contains $Q\bigcup R_{Q,s}$.
We rewrite $b(x)-b(y)$ as 
$\left(b(x)-\langle b\rangle_{20\sqrt{n}Q}\right)+\left(\langle b\rangle_{20\sqrt{n}Q}-b(y)\right)$. 
Then
\begin{align*} 
    K^{\mu,\lambda}_{b}(x,y)
    &=\sum_{Q\in\mathcal{D}}\sum_{s=1}^{M}\sum_{\vec{l}\in\mathbb{Z}^{2n}} 
    {\Upsilon^j_{\vec{l},Q}} \left(b(x)-\langle b\rangle_{20\sqrt{n}Q}\right) \frac{1}{|Q|^{1\over 2}}
    \lambda^{\frac{1}{2}}(x)F_{\vec{l}',Q}(x)\ \ \ \frac{1}{|R_{Q,s}|^{1\over 2}}
    G_{\vec{l}'',R_{Q,s}}(y)\mu^{-\frac{1}{2}}(y)\\
    &\qquad+\sum_{Q\in\mathcal{D}}\sum_{s=1}^{M}\sum_{\vec{l}\in\mathbb{Z}^{2n}} 
     {\Upsilon^j_{\vec{l},Q}}  \frac{1}{|Q|^{1\over 2}}\lambda^{\frac{1}{2}}(x)F_{\vec{l}',Q}(x)
     \left(\langle b \rangle_{20\sqrt{n}Q}-b(y)\right) \frac{1}{|R_{Q,s}|^{1\over 2}}
     G_{\vec{l}'',R_{Q,s}}(y)\mu^{-\frac{1}{2}}(y)\\
     &=:K^{\mu,\lambda}_{b,1}(x,y)+K^{\mu,\lambda}_{b,2}(x,y).
\end{align*}

To prove \eqref{eqn:suff-1.6}, we first verify the following $S^{n,\infty}-$norm of the operator $T^{\mu,\lambda}_{b,2}$ whose kernel  is $K^{\mu,\lambda}_{b,2}(x,y)$:
\begin{equation}\label{eqn:S-n-infty-2}
	  \|T^{\mu,\lambda}_{b,2}\|_{ S^{n,\infty}\left(L^{2}(\mathbb{R}^n),L^{2}(\mathbb{R}^n)\right)}
    \lesssim\|\{\operatorname{osc}_{2}(b,Q)\}_Q\|_{\ell^{n,\infty}},
\end{equation}
where
\begin{align*}
    \operatorname{osc}_{2}(b,Q)
    :=\bigg[{1\over \lambda^{-1}(20\sqrt{n}Q)} \int_{20\sqrt{n}Q} 
    \left|b(y)- \langle b\rangle_{20\sqrt{n}Q}\right|^{2} \mu^{-1}(y) dy\bigg]^{1\over 2}.
\end{align*}

To see this, we write 
\begin{align*} 
    K^{\mu,\lambda}_{b,2}(x,y)
    &=\sum_{Q\in\mathcal{D}}\sum_{s=1}^{M}\sum_{\vec{l}\in\mathbb{Z}^{2n}}  {\Upsilon^j_{\vec{l},Q}} 
    \operatorname{osc}_{2}(b,Q) \,
    \lambda^{\frac{1}{2}}(x)F_{\vec{l}',Q}(x) {1\over \lambda(Q)^{1\over 2}}\\ 
    &\qquad \times (\operatorname{osc}_{2}(b,Q))^{-1}
    \left( \langle b\rangle_{20\sqrt{n}Q}-b(y)\right) \frac{1}{|R_{Q,s}|^{1/2}}
    G_{\vec{l}'',R_{Q,s}}(y)\mu^{-\frac{1}{2}}(y) 
    { \lambda(Q)^{1\over 2}\over |Q|^{{1\over2} }}.
\end{align*}
Then we set
$$
    F_{\vec{l}',Q,2}(x)= 
    \lambda^{\frac{1}{2}}(x)F_{\vec{l}',Q}(x) {1\over \lambda(Q)^{1\over 2}}
$$
and
$$
    G_{\vec{l}'',R_{Q,s},2}(y)=  (\operatorname{osc}_{2}(b,Q))^{-1}
    \left( \langle b\rangle_{20\sqrt{n}Q}-b(y)\right) \frac{1}{|R_{Q,s}|^{1\over 2}}
    G_{\vec{l}'',R_{Q,s}}(y)\mu^{-\frac{1}{2}}(y) 
    { \lambda(Q)^{1\over 2}\over |Q|^{ 1\over 2} }.
$$

We first consider $F_{\vec{l}',Q,2}$. 
As  $ \lambda $ is $A_2$,  it follows from Lemma \ref{reverse} that there is a   $r = r _{[ \lambda ]_ {A_2}}>2$, so 
that for each dyadic cube $Q$, 
\begin{equation*}
\Bigg(  {1\over |Q|}\int _Q  \lambda(x) ^{r\over2}  \; dx \Bigg) ^{2\over r} \lesssim \frac{ \lambda (Q)}{ \lvert  Q \rvert}. 
\end{equation*}
With this choice of $r>2$,  
\begin{align*}
\|F_{\vec{l}',Q,2}\|_{L^r(\R^n)}&\leq  {1\over \lambda(Q)^{1\over2}}\bigg( \int_{Q} {\lambda^{r\over 2}(x)  } dx\bigg)^{1\over r}
= {|Q|^{1\over r}\over \lambda(Q)^{1\over2}}\bigg( {1\over |Q|} \int_{Q} {\lambda^{r\over 2}(x)  } dx\bigg)^{1\over r}\\
&\lesssim{|Q|^{1\over r}\over \lambda(Q)^{1\over2}}  \frac{ \lambda (Q)^{1\over2}}{ \lvert  Q \rvert^{1\over2}}\\
&= |Q|^{{1\over r}-{1\over 2}}.
\end{align*}
Hence, $\left\{F_{\vec{l}',Q,2}\right\}$ is a NWO sequence.

Next, we consider $G_{\vec{l}'',R_{Q,s},2}$ for $1\leq s\leq M$. 
\begin{align}
    \|G_{\vec{l}'',R_{Q,s},2}\|_{L^{2}(\mathbb R^n)} 
    &\leq (\operatorname{osc}_{2}(b,Q))^{-1}\frac{1}{|R_{Q,s}|^{1\over 2}}{ \lambda(Q)^{1\over 2}
    \over |Q|^{  1\over 2 }} 
    \bigg(\int_{20\sqrt{n}Q} |b(y)-\langle b\rangle_{20\sqrt{n}Q}|^{2} 
    \mu^{-1}(y)dy\bigg)^{1\over 2} \nonumber\\
    &\leq \frac{1}{|R_{Q,s}|^{1\over 2}}{ \lambda(Q)^{1\over 2} \lambda^{-1}
    (20\sqrt{n}Q)^{1\over 2}\over |Q|^{1\over 2} } \nonumber\\
    &\leq C_\lambda,\label{eqn:Clambda}
\end{align}
where $C_{\lambda}>0$ is a constant depending on the $[\lambda]_{A_2}$.

Thus, for every $R_{Q,s}$ with $|x-c_{R_{Q,s}}|\leq\ell(R_{Q,s})$, we have
\begin{align}\label{eqn:GR}
    {1\over |R_{Q,s}|^{1\over 2}}\big|\big\langle f, G_{\vec{l}'',R_{Q,s},2}\big\rangle\big|
    &\leq  {1\over |R_{Q,s}|^{1\over 2}} \|G_{\vec{l}'',R_{Q,s},2}\|_{L^{2}(\mathbb R^n)} 
    \bigg(\int_{R_{Q,s}} |f(y)|^2 dy\bigg)^{1\over 2} \nonumber\\
    &\leq 
    C_{\lambda}\,\bigg({1\over |R_{Q,s}|}\int_{R_{Q,s}} |f(y)|^2  dy\bigg)^{1\over 2}.
\end{align}

Therefore, for
\begin{equation}\label{eqn:K2}
	    T^{\mu,\lambda}_{b,2}(f)(x)=\sum_{\vec{l} \in \mathbb{Z}^{2n}} \sum_{s=1}^{M}
    \sum_{Q \in \mathcal{D}} {\Upsilon^j_{\vec{l},Q}}  \operatorname{osc}_{2}(b,Q)
    \left\langle f, G_{\vec{l}'',R_{Q,s},2}\right\rangle F_{\vec{l}',Q,2},
\end{equation}
we now consider the expansion of $\langle T^{\mu,\lambda}_{b,2}(f),g\rangle$ with $f,\, g\in L^2(\mathbb R^n)$.

To continue the proof of \eqref{eqn:S-n-infty-2},   we claim a more general conclusion that 
\begin{itemize}
	\item[\textbf{(C)}] the operator $T^{\mu,\lambda}_{b,2}$ is  in $S^{p,q}$ if the sequence $\big\{ \operatorname{osc}_{2}(b,Q) \big\}_Q$ is in $\ell^{p,q}$ for $0<p<\infty$ and $0<q\leq \infty$, and 
	$$
	        \left\|T^{\mu,\lambda}_{b,2}\right \|_{S^{p,q}(L^2(\mathbb R^n), L^2(\mathbb R^n))}\lesssim \big  \|  \big \{ \operatorname{osc}_{2}(b,Q)\big  \}_Q   \,\big  \|_{\ell^{p,q}}.
	$$
\end{itemize}
That is, the classical result \eqref{eq-NWO1} still holds for $T^{\mu,\lambda}_{b,2}$, whose representation \eqref{eqn:K2} does not exactly match the original requirements (two NOW seqeunces) as in \eqref{eq-NWO1}.

Similar to the proof of \cite[Corollary 2.8]{RS1989}, this is verified by combining the following two facts:
\begin{itemize}
	\item [\textbf{(F1)}] the operator $T^{\mu,\lambda}_{b,2}$  is  in $S^p$ if the sequence $\big \{\operatorname{osc}_{2}(b,Q)\big  \}_Q$ is in $\ell^p$ for $0<p<\infty$, and 
	 $$
	 \left\|T^{\mu,\lambda}_{b,2}\right \|_{S^p(L^2(\mathbb R^n), L^2(\mathbb R^n)}\lesssim \big  \|  \big\{\operatorname{osc}_{2}(b,Q)\big  \}_Q   \,\big  \|_{\ell^p}.
	$$

	\item[\textbf{(F2)}] the operator $\mathsf M'$  defined on the sequence $\big \{\operatorname{osc}_{2}(b,Q)\big  \}_Q$ by 
	\begin{equation}\label{eqn:Max-function-Car}
		 \mathsf M'\Big(\big \{ \operatorname{osc}_{2}(b,Q)\big  \}_Q\Big)(P) :=   {1\over |P|} 
    \sum_{R\in \mathcal D: \, R\subset P }   \operatorname{osc}_{2}(b,P)\,|R| 
	\end{equation}
is  a bounded map of $\ell^{p,q}$ to itself  for any $0<p<\infty$ and $0<q\leq \infty$.
\end{itemize}
  
\smallskip

Obviously, \textbf{(F2)} follows readily by the argument for Lemma \ref{lem aux Sobolev new}. To verify \textbf{(F1)}, we quote related definitions in \cite{RS1989} and begin by  showing that the operator $T^{\mu,\lambda}_{b,2}$ is  $L^2(\mathbb R^n)$ bounded, and  
\begin{equation}\label{eqn:aim}
     \left\|T^{\mu,\lambda}_{b,2}\right \|_{L^2(\mathbb R^n)\to L^2(\mathbb R^n)}\lesssim \big  \|  \big \{ \operatorname{osc}_{2}(b,Q)\big  \}_Q   \,\big  \|_{CMd},
\end{equation}
where the discrete Carleson measure of $\big \{ \operatorname{osc}_{2}(b,Q)\big  \}_Q$ is given by 
\begin{equation}\label{eqn:CMd}
    \big  \|  \big \{ \operatorname{osc}_{2}(b,Q)\big  \}_Q   \,\big  \|_{CMd}:=\sup_{P\in \mathcal D} \mathsf M'\Big(\big \{ \operatorname{osc}_{2}(b,Q)\big  \}_Q\Big)(P).
\end{equation}
It suffices to show that for any $f,g\in C_c^\infty(\mathbb R^n)$, 
$$
     \left|\left\langle T^{\mu,\lambda}_{b,2}(f),g\right \rangle\right| \lesssim  \big  \|  \big \{ \operatorname{osc}_{2}(b,Q)\big  \}_Q   \,\big  \|_{CMd}\, \|f\|_{L^2(\mathbb R^n)}\, \|g\|_{L^2(\mathbb R^n)}.
$$

Note that 
\begin{align*}
    &\left|\left\langle T^{\mu,\lambda}_{b,2}(f),g\right \rangle\right|\nonumber\\
    &=\bigg|\sum_{\vec{l} \in \mathbb{Z}^{2n}} \sum_{s=1}^{M}
    \sum_{Q \in \mathcal{D}} {\Upsilon^j_{\vec{l},Q}}  \operatorname{osc}_{2}(b,Q)
    \left\langle f, G_{\vec{l}'',R_{Q,s},2}\right\rangle \left\langle g, F_{\vec{l}',Q,2}\right\rangle\bigg| \nonumber\\
    &\leq \bigg(\sum_{\vec{l} \in \mathbb{Z}^{2n}} \sum_{s=1}^{M}
    \sum_{Q \in \mathcal{D}} {\Upsilon^j_{\vec{l},Q}}   \operatorname{osc}_{2}(b,Q)
    \left|  \left\langle f, G_{\vec{l}'',R_{Q,s},2}\right\rangle\right|^2   \bigg)^{1/2}\cdot \bigg(\sum_{\vec{l} \in \mathbb{Z}^{2n}} \sum_{s=1}^{M}
    \sum_{Q \in \mathcal{D}} {\Upsilon^j_{\vec{l},Q}}   \operatorname{osc}_{2}(b,Q)
      \left| \left\langle g, F_{\vec{l}',Q,2}\right\rangle  \right|^2 \bigg)^{1/2},
\end{align*}
and  the coefficients  $\big\{ \Upsilon^j_{\vec{l},Q} \big\}$ satisfy 
$\big| \Upsilon^j_{\vec{l},Q}\big| \lesssim \frac{1}{(1+|\vec{l}|)^{|\alpha|+|\beta|}}$ 
for all multi-indices $\alpha, ~\beta\in\mathbb{Z}_{+}^{n}$. To continue, it suffices to show that for each $1\leq s\leq M$, we have 
\begin{equation}\label{eqn:T-aux-1}
	 \sum_{Q \in \mathcal{D}}    \operatorname{osc}_{2}(b,Q)
    \left|  \left\langle f, G_{\vec{l}'',R_{Q,s},2}\right\rangle\right|^2 \lesssim  \big  \|  \big \{ \operatorname{osc}_{2}(b,Q)\big  \}_Q   \,\big  \|_{CMd} \,\|f\|_{L^2(\mathbb R^n)}^2
\end{equation}
and 
\begin{equation}\label{eqn:T-aux-2}
	 \sum_{Q \in \mathcal{D}}    \operatorname{osc}_{2}(b,Q)
    \left| \left\langle g, F_{\vec{l}',Q,2}\right\rangle  \right|^2   \lesssim  \big  \|  \big \{ \operatorname{osc}_{2}(b,Q)\big  \}_Q   \,\big  \|_{CMd} \,\|g\|_{L^2(\mathbb R^n)}^2.
\end{equation}
Since $\left\{F_{\vec{l}',Q,2}\right\}$ is a NWO sequence, we see that from Definition \ref{def:NWO},
$$
g^*(x):= \sup_{Q: |x-c_Q|\leq\ell(Q)}|Q|^{-{1\over2}} \Big| \left\langle g, F_{\vec{l}',Q,2}\right\rangle\Big| 
$$
is bounded on $L^2(\mathbb R^n)$. This allows us to use the argument in \cite[(7.39)]{RS1989} to obtain
\begin{align*}
	\text{LHS\  of \   }\eqref{eqn:T-aux-2} &\leq   \big  \|  \big \{ \operatorname{osc}_{2}(b,Q)\big  \}_Q   \,\big  \|_{CMd}
\int_{\mathbb R^n}  (g^*(x))^2 dx\lesssim \text{RHS\  of \   }\eqref{eqn:T-aux-2} ,
\end{align*}
as desired.

It remains to verify \eqref{eqn:T-aux-1}, whose argument is inspired by \cite[(7.39)]{RS1989}, but with slightly different machinery since $\left\{G_{\vec{l}'',R_{Q,s},2}  \right\}$ is not a NWO sequence. Define for each given $1\leq s\leq M$,
$$
   \mathcal E_\kappa =\left\{R_{Q,s}\in \mathcal D:\,  \frac{|  \langle f, G_{\vec{l}'',R_{Q,s},2}\rangle|^2}{|Q|}>\kappa \right\}  \quad \text{for}\quad \kappa>0,
$$ 
and let
$$
     \mathcal E_{\kappa}^{\max}=\left\{\text{maximal \   dyadic\    cubes \   in\   }\mathcal E_\kappa \right\},
$$
obviously, any two cubes in $\mathcal E_{\kappa }^{\max}$ are mutually disjoint and $\displaystyle  \bigcup_{P\in \mathcal E_{\kappa }^{\max}}  P    = \bigcup_{P\in \mathcal E_\kappa }P $.

Without loss of generality, according to the construction of  $R_{Q,s}$ associated to $Q$ as discussed at the beginning of the proof, one may assume $R_{Q,s}=Q$ (indeed, there exists a positive integer $\mathsf K$ such that for any $P\in \mathcal D$, $\#\left\{Q\in \mathcal D:\, R_{Q,s}=P\right \}\leq \mathsf K$).  
Therefore, for $f\in C_c^\infty(\mathbb R^n)$,  
\begin{align*}
	\text{LHS\  of \   }\eqref{eqn:T-aux-1} &=\sum_{Q \in \mathcal{D}}    \operatorname{osc}_{2}(b,Q)|Q|\ 
  {  \left|  \left\langle f, G_{\vec{l}'',R_{Q,s},2}\right\rangle\right|^2\over |Q|}\\
	 &=  \sum_{Q\in \mathcal D}    \operatorname{osc}_{2}(b,Q) |Q| \int_0^{\infty} 1_{\bigg\{\frac{ |    \langle f, G_{\vec{l}'',R_{Q,s},2}   \rangle  |^2}{|Q|}>\kappa\bigg\}}(Q) \,d\kappa  \\
	&=\int_0^\infty \sum_{Q\in \mathcal E_{\kappa} }  \operatorname{osc}_{2}(b,Q) |Q|  \,d\kappa \\
	&=\int_0^{\infty }  \sum_{P\in  \mathcal E_{\kappa}^{\max}  }  \bigg ( \frac{1}{|P|}\sum_{Q\in \mathcal E_{\kappa} :\, Q\subset P}        \operatorname{osc}_{2}(b,Q) |Q|\bigg )   \, |P| \,d\kappa\\
		&\leq    \big  \|  \big \{ \operatorname{osc}_{2}(b,Q)\big  \}_Q   \,\big  \|_{CMd}   \int_0^{\infty }  \Big | \bigcup_{P\in  \mathcal E_{\kappa }^{\max} } P\Big |   \,d\kappa,
\end{align*}
where the inequality follows from the definition of $CMd$ norm \eqref{eqn:CMd}.

Since $f\in C_c^\infty(\mathbb R^n)$, one may combine \eqref{eqn:GR} and the mean value theorem of integrals to see for each $R_{Q,s}$, there exists $x_Q\in R_{Q,s}$ such that $\forall\, x\in R_{Q,s}$,
\begin{align*}
      \frac{|  \langle f, G_{\vec{l}'',R_{Q,s},2}\rangle|^2}{|Q|}
    &\leq      (C_{\lambda})^2    {1\over |R_{Q,s}|}\int_{R_{Q,s}} |f(y)|^2  dy= (C_\lambda) ^2 \left  (f (x_Q)\right )^2 \\
    &\leq (C_\lambda) ^2  \left ({\mathcal M_{\mathcal D}(f)}\right )^2 (x_Q) =(C_\lambda) ^2     \left  (\mathcal M_{\mathcal D}(f)\right )^2 (x), 
\end{align*}
where $C_\lambda$ is as in \eqref{eqn:Clambda}, and $\mathcal M_{\mathcal D}$ is the uncentered dyadic maximal function given by $$\mathcal M_{\mathcal D}(h)(x):=\sup_{Q\in \mathcal D:\, x\in Q} \frac{1}{|Q|}\int_Q |h(y)|dy$$  which is  bounded on $L^2(\mathbb R^n)$.  Let
$$
      \Omega_\kappa =\left\{ x\in \mathbb R^n:\,  \left  (\mathcal M_{\mathcal D}(f)\right )^2 (x) >(C_\lambda) ^{-2} \kappa \right\}  \quad \text{for}\quad \kappa>0,
$$
then
$$
      {\bigcup_{P\in \mathcal E_{\kappa}^{\max}}  P }  \    \subset 
      \Omega_\kappa.
$$

Therefore, for $f\in C_c^\infty(\mathbb R^n)$,  
\begin{align*}
	\text{LHS\  of \   }\eqref{eqn:T-aux-1}  
	&\leq    \big  \|  \big \{ \operatorname{osc}_{2}(b,Q)\big  \}_Q   \,\big  \|_{CMd}   \int_0^{\infty }  \Big | \Omega_\kappa \Big |   \, d\kappa \\
	&\lesssim    \big  \|  \big \{ \operatorname{osc}_{2}(b,Q)\big  \}_Q   \,\big  \|_{CMd}   \int_{\mathbb R^n}    \left  (\mathcal M_{\mathcal D}(f)\right )^2 (y)\, dy\\
	&\lesssim    \big  \|  \big \{ \operatorname{osc}_{2}(b,Q)\big  \}_Q   \,\big  \|_{CMd}   \int_{\mathbb R^n}    |f (y)|^2\, dy,
\end{align*}
as desired. Thus we obtain \eqref{eqn:aim} by  combining \eqref{eqn:T-aux-1} and \eqref{eqn:T-aux-2}.

To continue the proof of \textbf{(F1)}, we recall again that the coefficients  $\big\{ \Upsilon^j_{\vec{l},Q} \big\}$ satisfy 
$\big| \Upsilon^j_{\vec{l},Q}\big| \lesssim \frac{1}{(1+|\vec{l}|)^{|\alpha|+|\beta|}}$ 
for all multi-indices $\alpha, ~\beta\in\mathbb{Z}_{+}^{n}$ for $1\leq s\leq M$. Thus, it suffice to show there exists a constant $C>0$ such that for each $\vec{l} \in \mathbb{Z}^{2n}$ and $1\leq s\leq M$, we have 
\begin{equation}\label{eqn:Sp-2}
	 \left\|T^{\mu,\lambda}_{b,2,{\vec{l},s}}\right \|_{S^p(L^2(\mathbb R^n), L^2(\mathbb R^n))}\leq C \big  \|  \big \{ \operatorname{osc}_{2}(b,Q)\big  \}_Q   \,\big  \|_{\ell^p} \quad \text{for}\quad   0<p<\infty,
\end{equation}
 where 
$$
   T^{\mu,\lambda}_{b,2,{\vec{l},s}}:=    \sum_{Q \in \mathcal{D}}    \operatorname{osc}_{2}(b,Q)
    \left\langle \cdot, G_{\vec{l}'',R_{Q,s},2}\right\rangle F_{\vec{l}',Q,2}.
$$

Recall that by Rayleigh's quotient, the $j-$th singular values $s_j(T^{\mu,\lambda}_{b,2,\vec{l},s})$  can  be characterized by 
\begin{equation}\label{SV-represent}
		 s_j\left (T^{\mu,\lambda}_{b,2,\vec{l},s}\right )=\inf\left\{  \left \|T^{\mu,\lambda}_{b,2,\vec{l},s}-F\right \|_{ L^2\to L^2}:\, F \text{ is\  a\   linear \  operator\ on\   } L^2(\mathbb R^n)\    \text{and}\   \text{rank} (F)\leq j\right\}.
\end{equation}
For the maximal operator $M'$ defined in \eqref{eqn:Max-function-Car}, 
let $M'(\big \{ \operatorname{osc}_{2}(b,Q)\big  \}_Q)^{*}(j)$ denote the $j-$th element of the non-increasing rearrangement of the sequence $\left\{\mathsf M'(\big \{ \operatorname{osc}_{2}(b,Q)\big  \}_{Q\in \mathcal D})(P)\right\}_{P\in \mathcal D}$, and denote  (at least one) the corresponding cubes in $\mathcal D$ by $P_j$, that is, $M'(\big \{ \operatorname{osc}_{2}(b,Q)\big  \}_Q)^*(j)=M'(\big \{ \operatorname{osc}_{2}(b,Q)\big  \}_Q)(P_j)$. 
For any $k>1$, choose
$$
    F=F_k:= 
    \sum_{j<k}    \operatorname{osc}_{2}(b,P_j)
    \left\langle \cdot, G_{\vec{l}'',R_{P_j,s},2}\right\rangle F_{\vec{l}',P_j,2}.
$$
By \eqref{SV-represent} and \eqref{eqn:aim}, one may verify 
that applying a similar  argument as in \cite[p. 241]{RS1989}, we have 
$$
   s_k(T^{\mu,\lambda}_{b,2,\vec{l},s})\leq C M'(\big \{ \operatorname{osc}_{2}(b,Q)\big  \}_Q)^*(k)
$$ 
and 
$$
   \sum_{P\in \mathcal D} \left[M'(\big \{ \operatorname{osc}_{2}(b,Q)\big  \}_Q)(P)\right]^p \leq  c_p \sum_{P\in \mathcal D}\left|\operatorname{osc}_{2}(b,P)\right|^p,  
$$
thus \eqref{eqn:Sp-2} follows, and we finish the proof of \textbf{(F1)}.   Consequently, the conclusion \textbf{(C)} holds. In particular, as a consequence, we see that 
 \eqref{eqn:S-n-infty-2} holds.

W now consider $K^{\mu,\lambda}_{b,1}(x,y)$  and denote the corresponding operator by $T^{\mu,\lambda}_{b,1}$.  Let
\begin{align*}
    \operatorname{osc}_{1}(b,Q)
    =\bigg[{1\over \mu(20\sqrt{n}Q)} \int_{20\sqrt{n}Q} \left|b(x)
    - \langle b\rangle_{20\sqrt{n}Q} \right|^{2}\lambda(x) dx\bigg]^{1\over 2}.
\end{align*}
Then
\begin{align*} 
    K^{\mu,\lambda}_{b,1}(x,y)
    &=\sum_{Q\in\mathcal{D}}\sum_{s=1}^{M}\sum_{\vec{l}\in\mathbb{Z}^{2n}} 
    {\Upsilon^j_{\vec{l},Q}} \operatorname{osc}_{1}(b,Q)\,\operatorname{osc}_{1}(b,Q) ^{-1}   
    \left(b(x)-\langle b\rangle_{20\sqrt{n}Q} \right) \frac{1}{|Q|} 
    \lambda^{\frac{1}{2}}(x)F_{\vec{l}',Q}(x)\ {\mu^{-1}(Q)^{1\over 2}}  \\ 
    &\qquad \times  
    G_{\vec{l}'',R_{Q,s}}(y)\mu^{-\frac{1}{2}}(y) {1\over {\mu^{-1}(Q)^{  {1\over 2} }    }}.
\end{align*}
We set
$$
    F_{\vec{l}',Q,1}(x)=\left(\operatorname{osc}_{1}(b,Q)\right)^{-1} 
    \left(b(x)-\langle b\rangle_{20\sqrt{n}Q} \right) \frac{1}{|Q|}
    \lambda^{\frac{1}{2}}(x)F_{\vec{l}',Q}(x) {\mu^{-1}(Q)^{1\over 2}} 
$$
and
$$
    G_{\vec{l}'',R_{Q,s},1}(y)=   
    G_{\vec{l}'',R_{Q,s}}(y)\mu^{-\frac{1}{2}}(y) {1\over {\mu^{-1}(Q)^{1\over 2}}}.
$$

Note that  $\operatorname{osc}_{1}(b,Q)$ has the same construct of $\operatorname{osc}_{2}(b,Q)$.  Meanwhile, $F_{\vec{l}',Q,1}(x)$ (reps. $ G_{\vec{l}'',R_{Q,s},1}(y)$) plays the role of $G_{\vec{l}'',R_{Q,s},2}(y)$ (resp. $F_{\vec{l}',Q,2}(x)$).
Therefore,  {a similar argument  yields that  $\left\{G_{\vec{l}'',R_{Q,s},1}\right\}$  is a  NWO sequence. Meanwhile, applying the approach in the proof of  \textbf{(C)} above, we can also obtain that 
$$
    \left\|T^{\mu,\lambda}_{b,1}\right \|_{L^2(\mathbb R^n)\to L^2(\mathbb R^n)}\lesssim \big  \|  \big \{ \operatorname{osc}_{1}(b,Q)\big  \}_Q   \,\big  \|_{CMd}
$$
and that
$$
 \left\|T^{\mu,\lambda}_{b,1}\right \|_{S^p(L^2(\mathbb R^n), L^2(\mathbb R^n))}\lesssim \big  \|  \big \{ \operatorname{osc}_{1}(b,Q)\big  \}_Q   \,\big  \|_{\ell^p}\quad \text{for}\quad   0<p<\infty,
$$
analogous to \eqref{eqn:aim} and \textbf{(F1)}.

Note that the proof of Lemma \ref{lem aux Sobolev new}  implies that 
 \textbf{(F2)} with replacing $\operatorname{osc}_{2}(b,Q)$ therein by $\operatorname{osc}_{1}(b,Q)$ also holds. Hence the analog of \textbf{(C)} follows and in particular, we have
 \begin{equation}\label{eqn:S-n-infty-1}
	  \|T^{\mu,\lambda}_{b,1}\|_{ S^{n,\infty}\left(L^{2}(\mathbb{R}^n),L^{2}(\mathbb{R}^n)\right)}
    \lesssim\|\{\operatorname{osc}_{1}(b,Q)\}_Q\|_{\ell^{n,\infty}},
\end{equation}
as desired.

\smallskip

Hence, 
$$
    (\lambda^{\frac{1}{2}}[b, R_{j}] \mu^{-\frac{1}{2}}) (f)(x)
    = \sum_{\vec{l} \in \mathbb{Z}^{2n}} \sum_{m=1}^{2}\sum_{s=1}^{M}\sum_{Q \in \mathcal{D}} 
    \Upsilon^j_{\vec{l},Q} \operatorname{osc}_{m}(b,Q)
    \left\langle f, G_{\vec{l}'',R_{Q,s},m}\right\rangle F_{\vec{l}',Q,m},
$$
and coefficients $\big\{ \Upsilon^j_{\vec{l},Q} \big\}$ which satisfy 
$\big| \Upsilon^j_{\vec{l},Q}\big| \lesssim \frac{1}{(1+|\vec{l}|)^{|\alpha|+|\beta|}}$ 
for all multi-indices $\alpha, ~\beta\in\mathbb{Z}_{+}^{n}$. 
Thus, combining \eqref{eqn:S-n-infty-2}, \eqref{eqn:S-n-infty-1},     Definition~\ref{def:Weight-Sobolev} and Theorem \ref{thm:Weight-Sobolev-revised},  we have
\begin{align*}
    \|\lambda^{\frac{1}{2}}[b, R_{j}] \mu^{-\frac{1}{2}}\|_{ S^{n,\infty}\left(L^{2}(\mathbb{R}^n),L^{2}(\mathbb{R}^n)\right)}
    &\lesssim \|\{\operatorname{osc}_{1}(b,Q)\}_Q\|_{\ell^{n,\infty}}+\|\{\operatorname{osc}_{2}(b,Q)\}_Q\|_{\ell^{n,\infty}}\\
    &\lesssim \|b\|_{\mathcal W_{\nu}(\mathbb R^n)}.
\end{align*}

Then we see that $ \lambda^{\frac{1}{2}}[b, R_{j}] \mu^{-\frac{1}{2}} \in S^{n,\infty}
\big(L^{2}(\mathbb{R}^n), L^{2}(\mathbb{R}^n)\big)$.


\subsection{Proof of the necessary condition}

We now assume that 
$[b, R_{j}]\in S^{n, \infty}  \big(L_{\mu}^2(\mathbb R^n), L_{\lambda}^2(\mathbb R^n)\big)$, 
then prove that $b\in \mathcal W_{\nu}(\mathbb R^n) $.

First, for each $Q\in \mathcal D$,  similar to  the proof of  Proposition \ref{prop:Besov-2}, 
one may choose another dyadic cube $\widehat Q\in\mathcal D$, such that: 
(i) $|Q|=|\widehat Q|$ and ${\rm distance}(20\sqrt{n}Q, 20\sqrt{n}\widehat Q)\approx \ell(Q)$; 
(ii) the kernel of Riesz transform $K_j(x-\widehat x)$ does not change sign for all 
$(x,\widehat x)\in 20\sqrt{n}Q\times 20\sqrt{n}\widehat Q$, and $|K_j(x-\widehat x)|\gtrsim |Q|^{-1}$. 
Moreover,  we can assume that these three properties are also valid when enlarging  
$20\sqrt{n}$ to $(1+\epsilon')20\sqrt{n}$ for some $\epsilon'>0$ sufficiently small.
 
Define
\begin{equation}\label{eqn:JQ}
    J_{Q}(x,y)=|20\sqrt{n}Q|^{-2}K^{-1}_{j}(x-y)\mathsf{1}_{20\sqrt{n}Q}(x) \mathsf{1}_{20\sqrt{n}\widehat{Q}}(y).
\end{equation}
Similar to the proof of the sufficient condition in Section \ref{ssec7.1}, 
take a smooth function $\zeta$ on $\mathbb R^n$ with ${\rm supp\,} \zeta\subset [-(1+\epsilon')/2,(1+\epsilon')/2]^n$, 
satisfying $\eta\equiv 1$ on $[-1/2,1/2]^n$ and $|\partial^\alpha \eta|\leq C_\alpha$ for each $\alpha\in \mathbb Z_+^n$. 
Define $\zeta_Q(x)=\zeta\left ((x-c_Q)/(20\sqrt{n}\ell(Q))\right )$, we have
$$
    K^{-1}_{j}(x-y)\mathsf{1}_{20\sqrt{n}Q}(x) \mathsf{1}_{20\sqrt{n}\widehat{Q}}(y)
    =K^{-1}_{j}(x-y)\zeta_Q(x)\zeta_{\widehat Q}(y)
$$
for $(x,y)\in 20\sqrt{n}Q\times  20\sqrt{n}\widehat{Q}$.

Applying the multiple Fourier series on 
$\left ((1+\epsilon')20\sqrt{n}Q\right )\times  \left ((1+\epsilon')20\sqrt{n}\widehat{Q}\right )$, 
we can write
\begin{align*}
    K^{-1}_{j}(x-y)\zeta_Q(x)\zeta_{\widehat Q}(y)
    =\sum_{\vec{l}\in\mathbb{Z}^{2n}}\tilde{c}^{j}_{\vec{l},Q}e^{2\pi i\vec{l}'\cdot \widetilde{x}}
    e^{2\pi i\vec{l}''\cdot \widetilde{y}} \, \mathsf{1}_{(1+\epsilon')20\sqrt{n}Q}(x)
    \cdot \mathsf{1}_{(1+\epsilon')20\sqrt{n}\widehat{Q}}(y),
\end{align*}
where $x=c_{Q}+(1+\epsilon')20\sqrt{n}\ell(Q)\tilde{x}$, 
$y=c_{\widehat Q}+(1+\epsilon')20\sqrt{n}\ell(\widehat{Q})\tilde{y}$,   
$\vec{l}=(\vec{l}', \vec{l}'')\in \mathbb Z^n\times \mathbb Z^n$,
and
$$
    \tilde{c}^{j}_{\vec{l},Q}
    =\int_{(1+\epsilon')20\sqrt{n}\widehat{Q}}\int_{(1+\epsilon')20\sqrt{n}Q}K^{-1}_{j}(x-y) 
    \zeta_Q(x)\zeta_{\widehat Q}(y) e^{-2\pi i\vec{l}'\cdot \widetilde{x} }e^{-2\pi i\vec{l}''
    \cdot \widetilde{y}}dxdy\frac{1}{|(1+\epsilon')20\sqrt{n}Q|}\frac{1}{|(1+\epsilon')20\sqrt{n}\widehat{Q}|}.
$$

Combining $|Q|=|\widehat{Q}|$,  
$\mathrm{distance}((1+\epsilon')20\sqrt{n}Q,(1+\epsilon')20\sqrt{n}\widehat{Q})\approx \ell(Q)$ 
and the regularity estimate
$$
    \left|\partial^{\alpha}_{x}\partial^{\beta}_{y}K^{-1}_{j}(x-y)\zeta_Q(x)\zeta_{\widehat Q}(y)\right|
    \leq C(\alpha,\beta) \,{|x-y|^{n-|\alpha|-|\beta|}}
$$
for any $\alpha,~\beta\in \mathbb{Z}_+^{n}$, we have  
\begin{align*}
    \big|  \tilde{c}^{j}_{\vec{l},Q}\big|
    &\lesssim \frac{\ell(Q)^{|\alpha|}\ell(\widehat{Q})^{|\beta|}}{(1+|\vec{l}|)^{|\alpha|+|\beta|}}
    \int_{(1+\epsilon')20\sqrt{n}\widehat{Q}}\int_{(1+\epsilon')20\sqrt{n}Q}
    \left|\partial^{\alpha}_{x}\partial^{\beta}_{y}K^{-1}_{j}(x-y) \zeta_Q(x)
    \zeta_{\widehat Q}(y)\right|dxdy\frac{1}{|(1+\epsilon')20\sqrt{n}Q|^2} \\
    &\lesssim{|Q|}\frac{1}{(1+|\vec{l}|)^{|\alpha|+|\beta|}}.
\end{align*}

Therefore, we can denote 
$ \tilde{\lambda}^j_{\vec{l},Q}=\frac{1}{|20\sqrt{n}Q|}{\tilde{c}^{j}_{\vec{l},Q}}$, 
then 
$$
    \big|\tilde{\lambda}^j_{\vec{l},Q}\big|\lesssim \frac{1}{(1+|l|)^{|\alpha|+|\beta|}}.
$$ 
Obviously,
\begin{eqnarray*}
    J_{Q}(x,y)=\sum_{\vec{l}\in\mathbb{Z}^{2n}} \tilde{\lambda}^j_{\vec{l},Q}
    \frac{1}{|20\sqrt{n}Q|^{1\over 2}} F_{\vec{l}',Q}(x)\frac{1}{|20\sqrt{n}\widehat{Q}|^{1\over 2}}
    G_{\vec{l}'',\widehat{Q}}(y),
\end{eqnarray*}
where
$F_{\vec{l}',Q}(x)=e^{2\pi i\vec{l}'\cdot \widetilde{x}}\, \mathsf{1}_{20\sqrt{n}Q}(x)$ and 
$G_{\vec{l}'',\widehat{Q}}(y)=e^{2\pi i\vec{l}''\cdot \widetilde{y}} \,\mathsf{1}_{20\sqrt{n}\widehat{Q}}(y)$.

Next, suppose that for each $Q\in \mathcal{D}$, we have two functions 
$\varepsilon_{Q} ,\, \eta_{\widehat{Q}}:\mathbb{R}^n \rightarrow \{-1,1\}$ 
with ${\rm supp}\,\varepsilon_{Q}\subset  20\sqrt{n}Q$, 
${\rm supp}\,\eta_{\widehat{Q}}\subset  20\sqrt{n}\widehat{Q}$.
Define the operator $L_{Q}$ as
$$
    \mu^{\frac{1}{2}}(x)L_{Q}(\mu^{-\frac{1}{2}}f)(x)
    =\int_{\mathbb{R}^n} \mu^{\frac{1}{2}}(x) \varepsilon_{Q}(x)J_{Q}(x,y) 
    \eta_{\widehat Q}(y) \mu^{-\frac{1}{2}}(y)f(y)dy.
$$
Considering an arbitrary sequence $\{a_{Q}\}_{Q\in\mathcal {D}}\in \ell^{\frac{n}{n-1},1}$. 
Here $\ell^{\frac{n}{n-1},1}$ is the Lorentz sequence space defined as the set of all sequences 
$\{a_{Q}\}_{Q\in\mathcal{D}}$ such that
$$  
    \big\|\{a_{Q}\}_{Q\in\mathcal{D}}\big\|_{\ell^{\frac{n}{n-1},1}} 
    =\sum_{k=1}^\infty k^{-\frac{1}{n}  }a^*_k, 
$$
where the sequence $\{a^*_k\}$ is  the sequence $\{ |a_Q|\}$ rearranged in a decreasing order.

Define the operator $L$ as
\begin{align*}
    \mu^{\frac{1}{2}}(x)L(\mu^{-\frac{1}{2}}f)(x)
    &=\sum_{Q\in\mathcal{D}}a_{Q}\, \mu^{\frac{1}{2}}(x)L_{Q}(\mu^{-\frac{1}{2}}f)(x).
\end{align*}

{We rewrite 
\begin{align*}
    \mu^{\frac{1}{2}}(x)L(\mu^{-\frac{1}{2}}f)(x)=
    \sum_{Q\in\mathcal{D}}\sum_{\vec{l}\in\mathbb{Z}^{2n}}\tilde{\lambda}^j_{\vec{l},Q}a_{Q}\varepsilon_Q (x)
    \frac{\mu(20\sqrt{n}Q)^{1\over 2} \mu^{-1}(20\sqrt{n}\widehat{Q})^{1\over 2}}{|20\sqrt{n}Q|^{ {1\over 2}}
    |20\sqrt{n}\widehat{Q}|^  {   {1\over 2}  }   }
    \langle f,\tilde{G}_{\vec{l}'',\widehat{Q}}   \rangle \tilde{F}_{\vec{l}',Q}(x),
\end{align*}
where
$$   
    \tilde{G}_{\vec{l}'',\widehat{Q}}
    =\frac{G_{\vec{l}'',\widehat{Q}}\,\eta_{\widehat Q}\,\mu^{-\frac{1}{2}}}{\mu^{-1}(20\sqrt{n}\widehat{Q})^{\frac{1}{2}}} 
         \quad
    \text{and}\quad   
    \tilde{F}_{\vec{l}',Q}= \frac{F_{\vec{l}',Q}\,\mu^{\frac{1}{2}}}{\mu(20\sqrt{n}Q)^{\frac{1}{2}}} .
$$
Then the argument for $F_{\vec{l}',Q,2}$ in the proof of  Theorem~\ref{thm:main3} deduces that $\left\{\tilde{G}_{\vec{l}'',\widehat{Q}}\right\}$ and $\left\{\tilde{F}_{\vec{l}',Q}\right\}$ are  NWO sequences.}
Thus, applying \eqref{eq-NWO1} and the fact 
$\displaystyle  \frac{\mu(20\sqrt{n}Q)^{1\over 2} \mu^{-1}(20\sqrt{n}\widehat{Q})^{1\over 2}}{|20\sqrt{n}Q|^{ {1\over 2}}     |20\sqrt{n}\widehat{Q}|^  {   {1\over 2}  }  }\lesssim 1$ 
to give
$$
    \|L\|_{S^{\frac{n}{n-1},1}\big (L_{\mu}^{2}(\mathbb R^n), L_{\mu}^2 (\mathbb R^n)\big)}
    =\big \|\mu^{\frac{1}{2}}L\mu^{-\frac{1}{2}}\big\|_{S^{\frac{n}{n-1},1}(L^{2}(\mathbb{R}^{n}), L^2(\mathbb R^n))}
    \lesssim\|a_{Q}\|_{\ell^{\frac{n}{n-1},1}}.
$$

Recall that for an integral operator $\mathcal T$ with the associated kernel $K(x,y)$,
$$
    \mathcal T f(x) = \int _X K(x,y)f(y)\ dy,
$$
then
$$
    {\rm Trace}(\mathcal T)  = \int _X K(x,x)\ dx.
$$

By definition,
\begin{align*}
    \lambda^{1\over 2}[b,R_j]L_Q(\mu^{-{1\over 2}}f)(x) 
    &= \int_{20\sqrt{n}Q } \int_{20\sqrt{n}\widehat Q} \lambda^{1\over 2}(x)(b(x)-b(y)) K_j(x,y) 
    \varepsilon_Q(y) J_Q(y,z) \eta_{\widehat Q}(z)\mu^{-{1\over 2}}(z)f(z)\ dzdy ,
\end{align*}
then we have
$$
    {\rm Trace} (\lambda^{1\over 2}[b,R_j]L_Q\,\mu^{-{1\over 2}})
    = \int_{20\sqrt{n}Q} \int_{20\sqrt{n}\widehat Q}\lambda^{1\over 2}(x) (b(x)-b(y))K_j(x,y) 
    \varepsilon_Q(y) J_Q(y,x) \eta_{\widehat Q}(x)\mu^{-{1\over 2}}(x) \ dxdy. 
$$

Combining the definition  \eqref{eqn:JQ} of $J_Q$,  we choose $\varepsilon_{Q}(x)$ and $\eta_{\widehat Q}(y)$ so that 
\begin{align*}
    &{\rm Trace} \left(\lambda^{1\over 2}[b,R_j] L_Q\,\mu^{-{1\over 2}}\right)\\
    &={1\over |20\sqrt{n}Q|^2} \int_{20\sqrt{n}Q }\int_{20\sqrt{n}\widehat Q} 
    \lambda^{1\over 2}(x) |b(x)-b(y)| \mu^{-{1\over 2}}(x) \ dxdy \\
    &\geq {1\over |20\sqrt{n}Q|^2} \int_{\mathsf E_1^Q} \int_{ \mathsf F_1^{\widehat Q}} 
    \nu^{-1}(x) |b(y)-\mathsf m_b(\widehat Q)|  \ dxdy 
    +{1\over |Q|^2} \int_{\mathsf E_2^Q}   \int_{ \mathsf F_2^{\widehat Q}} 
    \nu^{-1}(x) |b(y)-\mathsf m_b(\widehat Q)|  \ dxdy,
\end{align*}
where 
$$
    \mathsf E_1^Q=\left\{x\in 20\sqrt{n}Q:\, b(x)< \mathsf m_b(\widehat Q)\right\},\quad 
    \mathsf E_2^Q=\left\{x\in 20\sqrt{n}Q:\, b(x)>\mathsf m_b(\widehat Q)\right\}
$$
$$
    \mathsf F_1^{\widehat Q}=\left\{y\in 20\sqrt{n}\widehat Q:\, b(y)
    \geq \mathsf m_b(\widehat Q)\right\},\quad \mathsf F_2^{\widehat Q}
    =\left\{y\in 20\sqrt{n}\widehat Q:\, b(y)\leq \mathsf m_b(\widehat Q)\right\},
$$
and $\mathsf m_b(\widehat Q)$ is a median value of $b$ over $20\sqrt{n}\widehat Q$ in the sense of 
$$
    \max\left\{  \left|\left\{y\in 20\sqrt{n}\widehat Q:\, b(y)<\mathsf m_b(\widehat Q)\right\}\right|,\   
    \left|\left\{y\in 20\sqrt{n}\widehat Q:\, b(y)>\mathsf m_b(\widehat Q)\right\}\right| \right\} 
    \leq  \frac{| 20\sqrt{n}\widehat Q|} {2}.
$$

Note that $|\mathsf F_s^{\widehat Q}|\approx |20\sqrt{n}\widehat{Q}|$ for $s=1,2$ and $\nu\in A_2$, 
we have 
$$
    {\rm Trace} \left(\lambda^{1\over 2}[b,R_j] L_Q\,\mu^{-{1\over 2}} \right)
    \gtrsim {\nu^{-1}(\mathsf F_1^{\widehat Q})\over |20\sqrt{n} Q|^2}  \int_{\mathsf E_1^Q} 
    \left|b(y)-\mathsf m_b(\widehat Q)\right|  \ dy
    \approx {1\over \nu(20\sqrt{n}Q) }  \int_{\mathsf E_1^Q} \left |b(y)-\mathsf m_b(\widehat Q) \right|  \ dy
$$
and
$$
    {\rm Trace} \left(\lambda^{1\over 2}[b,R_j] L_Q\,\mu^{-{1\over 2}}\right)
    \gtrsim {\nu^{-1}( \mathsf F_2^{\widehat Q}) \over |20\sqrt{n}Q|^2}  \int_{\mathsf E_2^Q} 
    \left|b(y)-\mathsf m_b(\widehat Q)\right|  \ dy
    \approx {1\over \nu(20\sqrt{n} Q) }  \int_{\mathsf E_2^Q} \left|b(y)-\mathsf m_b(\widehat Q)\right|  \ dy.
$$
Then
$$
    {\rm Trace} \left(\lambda^{1\over 2}[b,R_j] L_Q\,\mu^{-{1\over 2}}\right)
    \gtrsim  {1\over \nu(20\sqrt{n} Q) }  \int_{20\sqrt{n}Q}\left |b(y)-\mathsf m_b(\widehat Q)\right|  \ dy.
$$

Furthermore,
\begin{align*}
	M(b,Q) &:=\frac{1}{\nu(20\sqrt{n}Q)}\int_{20\sqrt{n}Q} \left|b(x)-\langle b\rangle _{20\sqrt{n}Q}\right|dx \\
	&\leq {1\over \nu(20\sqrt{n}Q) }  \int_{20\sqrt{n}Q} \left|b(y)-\mathsf m_b(\widehat Q)\right|  \ dy 
	+\frac{|20\sqrt{n}Q|}{\nu(20\sqrt{n}Q)} \left |\mathsf m_b(\widehat Q)-\langle b\rangle _{20\sqrt{n}Q}\right|\\
	&\leq  {2\over \nu(20\sqrt{n}Q) }  \int_{20\sqrt{n}Q} \left|b(y)-\mathsf m_b(\widehat Q) \right|  \ dy .
\end{align*}

Therefore, by duality, there exists a sequence $\{a_Q\}_{Q\in\mathcal{D}}$ with 
$\|a_{Q}\|_{\ell^{\frac{n}{n-1},1}}\leq1$ such that
\begin{align*}
    \|b\|_{\mathcal W_{\nu}(\mathbb R^n)}&=\|M(b,Q)\|_{\ell^{n,\infty}}\\
    &\lesssim \big \|\operatorname{Trace}\big (\lambda^{\frac{1}{2}}[b,R_j]L_{Q}
    (\mu^{-\frac{1}{2}})\big )\big \|_{\ell^{n,\infty}}\\
    &=\sup_{\|a_{Q}\|_{\ell^{\frac{n}{n-1},1}}\leq1}\sum_{Q\in \mathcal{D}}\operatorname{Trace}
    \big (\lambda^{\frac{1}{2}}[b,R_j]L_{Q}\mu^{-\frac{1}{2}}\big )\cdot a_{Q}\\
    &=\sup_{\|a_{Q}\|_{\ell^{\frac{n}{n-1},1}}\leq1}\operatorname{Trace}
    \big (\lambda^{\frac{1}{2}}[b,R_j]L\mu^{-\frac{1}{2}}\big )\\
    &\lesssim\sup_{\|a_{Q}\|_{\ell^{\frac{n}{n-1},1}}\leq1} \big 
    \|\lambda^{\frac{1}{2}}[b,R_j]\mu^{-\frac{1}{2}}\big \|_{S^{n,\infty}(L^{2}(\mathbb{R}^{n}),L^{2}(\mathbb{R}^{n}))} 
    \big \|\mu^{\frac{1}{2}}L\mu^{-\frac{1}{2}}\big \|_{S^{\frac{n}{n-1},1}(L^{2}(\mathbb{R}^{n}),L^{2}(\mathbb R^n))}\\
    &\lesssim\big \|\lambda^{\frac{1}{2}}[b,R_j]\mu^{-\frac{1}{2}}\big \|_{S^{n,\infty}(L^{2}(\mathbb R^n),L^{2}(\mathbb R^n))}.
\end{align*}
Hence, the proof of the necessary condition in Theorem \ref{thm:main3} is complete.

\section{Application: The Quantised Derivative in the two weight setting}\label{application}
\setcounter{equation}{0}

Let $n$ be a positive integer, and let $x_1, \ldots, x_n$ be the coordinates of $\mathbb{R}^n$. 
For $j=1, \ldots, n$, we define $\mathfrak D_j$ to be the derivative in the direction $x_j$,
$\mathfrak D_j=\frac{1}{i} \frac{\partial}{\partial x_j}=-i \partial_j$. 
When $f \in L^{\infty}\left(\mathbb{R}^n\right)$ is not a smooth function then $\mathfrak D_j f$ 
denotes the distributional derivative of $f$. 
Let $N=2^{\lfloor n / 2\rfloor}$. We use $n$-dimensional Euclidean gamma matrices, 
which are $N \times N$ self-adjoint complex matrices $\gamma_1, \ldots, \gamma_n$ satisfying the anticommutation relation, 
$\gamma_j \gamma_k+\gamma_k \gamma_j=2 \delta_{j, k}, \ 1 \leq j, k \leq n$, where $\delta$ is the Kronecker delta.
The precise choice of matrices satisfying this relation is unimportant so we assume that a choice is fixed for the rest of this paper.
Using this choice of gamma matrices, we can define the $n$-dimensional Dirac operator,
$\mathfrak D=\sum_{j=1}^n \gamma_j \otimes \mathfrak D_j$.
This is a linear operator on the Hilbert space $\mathbb{C}^N \otimes L^2\left(\mathbb{R}^n\right)$ 
initially defined with dense domain $\mathbb{C}^N \otimes \mathcal{S}\left(\mathbb{R}^n\right)$, 
where $\mathcal{S}\left(\mathbb{R}^n\right)$ is the Schwartz space of functions on $\mathbb{R}^n$.
Taking the closure we obtain a self-adjoint operator which we also denote by $\mathfrak{D}$.
We then define the $\operatorname{sign}$ of $\mathfrak{D}$ as the operator $\operatorname{sgn}(\mathfrak{D})$ 
via the Borel functional calculus, i.e., $\operatorname{sgn}(\mathfrak{D}) = {\mathfrak{D}\over |\mathfrak{D}|}.$

Given $f \in {\rm BMO}_{\nu}(\mathbb R^n)$, 
denote by $M_f$ the operator of pointwise multiplication by $f$ on the Hilbert space $L^2\left(\mathbb{R}^n\right)$.
The operator $1 \otimes M_f$ is a bounded linear operator on $\mathbb{C}^N \otimes L^2\left(\mathbb{R}^n\right)$, 
where $1$ denotes the identity operator on $\mathbb{C}^N$. The commutator,
$$
    \bar{d} f:=i\left[\operatorname{sgn}(\mathfrak{D}), 1 \otimes M_f\right]
$$
denotes the quantised derivative of Alain Connes introduced in \cite[$\mathrm{IV}$]{C1994}.

Along with our main results in Theorems \ref{thm:main1}, \ref{thm:main2}, and \ref{thm:main3}, 
we have the analogy of \cite{CST1994,LMSZ2017,FSZ2022,F2022} in the two weight setting.

\begin{theorem}\label{Athm3}
Suppose $n\geq2$, $\lambda,\, \mu\in A_2$ and set $\nu=\mu^{1\over 2}\lambda^{-{1\over 2}}$. 
Suppose $f\in {\rm BMO}_{\nu}(\mathbb R^n)$. 
Then $\bar{d} f\in S^{n,\infty}\big(\mathbb{C}^N \otimes  (L_{\mu}^2(\mathbb R^n) 
\to L_{\lambda}^2(\mathbb R^n) )\big)$  
if and only if  $f\in \mathcal W_{\nu}(\mathbb R^n)$.
Moreover,
$$
     \|\bar{d} f\|_{S^{n,\infty}\big(\mathbb{C}^N \otimes  (L_{\mu}^2(\mathbb R^n) 
     \to  L_{\lambda}^2(\mathbb R^n) )\big)}
     \approx \|f\|_{ \mathcal W_{\nu}(\mathbb R^n)},
$$
where the implicit constants depends on $p,$ $n,$ $[\lambda]_{A_2}$ and $[\mu]_{A_2}$.
\end{theorem}


\vskip .8cm

\noindent{\bf Acknowledgements:} 
Lacey is supported in part by National Science Foundation DMS Grant. 
Li, Wick and Wu are supported by ARC DP 220100285.
Wick is also supported by National Science Foundation Grants DMS \# 2054863. 
Wu is also supported by  National Natural Science Foundation of China \# 12201002, Anhui Natural Science Foundation of China \# 2208085QA03 and Excellent University Research and Innovation Team in Anhui
Province \#2024AH010002.



\vskip .8cm
\begin{thebibliography}{100}

\bibitem{BL} J. Bergh and J. L\"ofstr\"om.  {\it Interpolation spaces. An introduction.} Grundlehren der Mathematischen Wissenschaften, No. 223. Springer-Verlag, Berlin-New York, 1976.

\bibitem{B} S. Bloom. A commutator theorem and weighted BMO. {\it Trans. Amer. Math. Soc.} {\bf  292} (1985), no. 1, 103--122.

\bibitem{CRW1976} R. Coifman, R. Rochberg and G. Weiss. Factorization theorems for Hardy spaces in several variables. {\it Ann. of Math. (2)}  {\bf  103} (1976),  611--635.
	
\bibitem{C1994} A. Connes.  {\it Noncommutative Geometry}. Academic Press, Inc., San Diego, CA, 1994.
	
\bibitem{CST1994} A. Connes, D. Sullivan and N. Teleman.  Quasiconformal mappings, operators on Hilbert space, and local formulae for characteristic classes.  
{\it Topology} {\bf 33} (1994), no. 4,  663--681.
	
\bibitem{FLMSZ} Z.J. Fan, J. Li, E. McDonald, F. Sukochev and D. Zanin. Endpoint weak Schatten class estimates and trace formula for commutators of Riesz transforms with multipliers on Heisenberg groups. {\it J. Funct. Anal.} {\bf  286} (2024), no. 1, Paper No. 110188.

\bibitem{F2022} R.L. Frank. A characterization of $\dot{W}^{1, p}(\mathbb{R}^{d})$.  {\it Pure Appl. Funct. Anal. } {\bf 9} (2024), no. 1, 53--68.

\bibitem{FSZ2022} R.L. Frank, F. Sukochev and D. Zanin. Asymptotics of singular values for quantum derivatives. {\it Trans. Amer. Math. Soc. }{\bf 376} (2023), no. 3, 2047--2088.

\bibitem{GG2017} H. Gimperlein and M. Goffeng.  Nonclassical spectral asymptotics and Dixmier traces: from circles to contact manifolds. 
{\it Forum Math. Sigma} {\bf 5} (2017), Paper No. e3, 57pp.

\bibitem{GLW2022} Z.B. Gong, J. Li and B.D. Wick.
Besov spaces, Schatten classes and weighted versions of the quantised derivative. {\it Anal. Math. } {\bf  49} (2023), no. 4, 971--1006.

\bibitem{Gra1} L. Grafakos.  {\it Classical Fourier analysis}. Third edition. Graduate Texts in Mathematics, {\bf 249}. Springer, New York, 2014.
 
\bibitem{HLW2017} I. Holmes, M.T. Lacey and B.D. Wick. Commutators in the two-weight setting.  {\it Math. Ann. } {\bf 367} (2017), no. 1--2, 51--80. 
 
\bibitem {HK2012} T. Hyt\"onen and A. Kairema.  Systems of dyadic cubes in a doubling metric space. {\it Colloq. Math.} {\bf 126} (2012),   1--33.

\bibitem{HL1} T. Hyt\"onen and  S. Lappas.  Extrapolation of compactness on weighted spaces. {\it Rev. Mat. Iberoam.} {\bf 39} (2023), no.1, 91--122.

\bibitem{HL2} T. Hyt\"onen and S. Lappas. Extrapolation of compactness on weighted spaces II: Off-diagonal and limited range estimates. arXiv:2006.15858v2.

\bibitem{I2013} J. Isralowitz. Schatten $p$ class commutators on the weighted Bergman space $L_a^2 (\mathbb{B}_n, \mathrm{~d} \nu_\gamma)$ for $2 n /(n+1+\gamma)<p<\infty$. 
{\it Indiana Univ. Math. J.} {\bf 62} (2013), no. 1, 201--233.

\bibitem{JW1982} S. Janson and T.H. Wolff.  Schatten classes and commutators of singular integral operators. {\it Ark. Mat.} {\bf 20} (1982), no.2,  301-310.

\bibitem{JX10} M. Junge,  Q.H. Xu. Representation of certain homogeneous Hilbertian operator spaces and applications. {\it Invent. Math.} {\bf 179} (2010), no. 1, 75--118.
 
\bibitem{LL2022} M.T. Lacey and J. Li.  Compactness of commutator of Riesz transforms in the two weight setting. {\it J. Math. Anal. Appl. } {\bf 508} (1), 125869 (2022). 

\bibitem{LLW2022} M.T. Lacey, J. Li and B.D. Wick. Schatten classes and commutator in the two weight setting, I. Hilbert transform. {\it Potential Anal.} (2023). 

\bibitem{LLWW} M.T. Lacey, J. Li, B.D. Wick and L.C. Wu.
Schatten classes and commutators in the two weight setting, II. Riesz transforms.
arXiv:2404.00329 V1.
 
\bibitem{LPW} J. Li, J. Pipher and  L.A. Ward. Dyadic structure theorems for multiparameter function spaces. {\it Rev. Mat. Iberoam. } {\bf 31}  (2015), no. 3, 767--797.

\bibitem{LMSZ2017} S. Lord, E. McDonald, F. Sukochev and D. Zanin. Quantum differentiability of essentially bounded functions on Euclidean space. 
{\it J. Funct. Anal.} 273 (2017) 2353--2387.

\bibitem{Pau16} J. Pau. Characterization of Schatten-class Hankel operators on weighted Bergman spaces. {\it  Duke Math. J.} {\bf 165} (2016), no. 14, 2771--2791.

\bibitem{P2003} V.V. Peller. {\it  Hankel operators and their applications}. Springer Monographs in Mathematics. Springer-Verlag, New York, 2003.

\bibitem{PTV} S. Petermichl, S. Treil and A. Volberg. Why the Riesz transforms are averages of the dyadic shifts?
Proceedings of the 6th International Conference on Harmonic Analysis and Partial Differential Equations (El Escorial, 2000), {\it Publ. Mat.,} Vol. Extra (2002), 209--228.

\bibitem{Pisier} G. Pisier. {\it Non-commutative vector valued $L_p$-spaces and completely $p$-summing maps}.  Ast\'erisque No. {\bf 247} (1998), vi+131 pp.

\bibitem{RS1988} R. Rochberg and S. Semmes. End point results for estimates of singular values of singular integral operators. 
In: {\it Contributions to operator theory and its applications } (Mesa, AZ, 1987), 217--231, Oper. Theory Adv. Appl., 35, Birkh\"auser, Basel, 1988.

\bibitem{RS1989} R. Rochberg and S. Semmes.  Nearly weakly orthonormal sequences, singular value estimates, and Calderon-Zygmund operators. 
{\it J. Funct. Anal.} {\bf 86} (2), 237--306 (1989). 

\bibitem{Zhu91} K.H. Zhu. Schatten class Hankel operators on the Bergman space of the unit ball. {\it Amer. J. Math. } {\bf 113} (1991), no. 1, 147--167.

\end{thebibliography}
\end{document}